\documentclass[12pt]{article}

\usepackage[utf8]{inputenc}
\usepackage[english]{babel}
\usepackage[T1]{fontenc}
\usepackage[title, titletoc]{appendix}

\usepackage{amsmath, amssymb, amsthm}
\usepackage{mathtools}
\usepackage{enumitem}

\usepackage[margin=1in]{geometry}
\usepackage{graphicx}
\usepackage{xcolor}

\definecolor{mirojo}{RGB}{255, 0, 0}

\newcommand{\revon}{\bfseries\boldmath\color{mirojo}}   

\allowdisplaybreaks
\numberwithin{equation}{section}

\usepackage{hyperref}
\hypersetup{
    colorlinks=true,
    linkcolor=blue,
    citecolor=blue,
    urlcolor=blue,
}

\newtheorem{theorem}{Theorem}[section]
\newtheorem{lemma}[theorem]{Lemma}
\newtheorem{proposition}[theorem]{Proposition}
\newtheorem{corollary}[theorem]{Corollary}
\newtheorem{definition}[theorem]{Definition}
\newtheorem{assumption}[theorem]{Assumption}
\newtheorem{convention}[theorem]{Convention}

\newtheorem{remark}[theorem]{Remark}
\theoremstyle{remark}

\newcommand{\step}[1]{\par\medskip\noindent\textbf{Step #1.}}

\newcommand{\Beta}{\mathrm{B}}
\newcommand{\IA}{I_{A}}
\newcommand{\IB}{I_{B}}

\title{Statistical inference for rough Burgers-type stochastic partial differential equations}
\author{  Simon Chony Acosta \footnote{Departamento de Matem\'atica, Universidade Estadual de Campinas, Brazil. \texttt{s252865@dac.unicamp.br}.}
	\and    Christian Olivera \footnote{Departamento de Matem\'atica, Universidade Estadual de Campinas, Brazil. \texttt{colivera@ime.unicamp.br}.} \and 
	Ciprian Tudor  \footnote{Université de Lille, France . \texttt{ciprian.tudor@univ-lille.fr}.} 
}
\date{}

\begin{document}

\maketitle

\begin{abstract}
We consider the Burgers-type stochastic PDE
\[
    du = \theta\,\partial_x^2 u\,dt + f(u)\,dt + g(u)\,\partial_x u\,dt + \sigma\,dW(t),
\]
driven by space--time white noise on the one-dimensional torus, where \(\theta>0\) is the diffusivity parameter and \(\sigma>0\) the noise intensity. The transport term \(g(u)\,\partial_x u\) admits no classical interpretation, and we work with the mild solutions defined via rough paths following~\cite{Hairer2011}, for which the solution decomposes as \(u=\psi+R\) with \(\psi\) a stationary centred Gaussian process solving a damped stochastic heat equation and \(R\) a remainder.

We first show that, away from the initial time, \(R\) is Lipschitz in space and H\"older in time of order strictly greater than \(1/2\), hence strictly smoother than \(\psi\) in both variables. We then compute exactly the spatial quadratic variation and the temporal quartic variation of \(\psi\), almost surely and not merely in \(L^{2}(\Omega)\), and prove that the corresponding variations of \(u\) are unaffected by the remainder. Combined with an exact time-scaling identity, this separation of regularity scales yields strongly consistent pathwise estimators of \(\theta\). To our knowledge this is the first inference procedure for an SPDE whose solution is defined only pathwise, as a measurable function of a rough path, rather than as an It\^o process.
\end{abstract}

\section{Introduction}

This work is concerned with statistical inference for a class of singular stochastic partial differential equations (SPDEs) of Burgers type. We consider
\begin{equation}\label{eq:Burgers_intro}
    du = \theta\,\partial_x^2 u\,dt + f(u)\,dt + g(u)\,\partial_x u\,dt + \sigma\,dW(t), \qquad (t,x)\in[0,T]\times\mathbb{T},
\end{equation}
endowed with periodic boundary conditions. Here \(\theta>0\) is the diffusivity parameter of interest, \(\sigma>0\) is the noise intensity and it is assumed to be known, \(u\) takes values in \(\mathbb{R}^{n}\), the coefficients \(f:\mathbb{R}^{n}\to\mathbb{R}^{n}\) and \(g:\mathbb{R}^{n}\to\mathbb{R}^{n\times n}\) are smooth functions, and \(W\) denotes an \(\mathbb{R}^{n}\)-valued space--time white noise on \([0,T]\times\mathbb{T}\), where \(\mathbb{T}\) is the one-dimensional torus, i.e.\ the interval \([0,2\pi]\) with periodic boundary conditions. The operator \(\partial_x^2\) is understood with the same boundary conditions.

Equation~\eqref{eq:Burgers_intro} is a prototypical \emph{singular} SPDE. The driving noise is so irregular that, for fixed \(t\), the solution is expected to have spatial H\"older regularity strictly below \(1/2\), while the nonlinearity involves the product \(g(u)\,\partial_x u\) of a function of \(u\) with a distribution of negative regularity. We assume throughout that the transport term is not of divergence form (Assumption (A3)); this is what makes the term genuinely singular. A classical interpretation of this term is therefore impossible. More importantly for the present work, the difficulty is not merely technical: different smooth approximations of the noise may converge to \emph{different} limits, so that~\eqref{eq:Burgers_intro} does not determine a solution until an interpretation of the ill-defined product has been fixed. In~\cite{Hairer2011} it is shown how the theory of rough paths supplies such an interpretation and yields pathwise well-posedness for \(\theta=1\). The key idea is to decompose the solution as
\[
    u = \psi + U + v,
\]
where \(\psi\) is the stationary centred Gaussian process solving the linear damped equation
\begin{equation}\label{eq:estacionario_H}
    d\psi = (\partial_x^2-1)\psi\,dt + \sigma\,dW(t),
\end{equation} 
\(U_t = S_t(u_0-\psi_0)\) (with \(S_t\) the periodic heat semigroup) carries the initial condition, and \(v\) solves a rough partial differential equation driven by the lifted rough path \(\Psi\) associated with \(\psi\). Both \(U\) and \(v\) enjoy better space--time regularity than \(\psi\); we write \(R \coloneqq U+v\) for their sum and refer to it as the \emph{remainder}.

Our goal is to estimate \(\theta\) from a single observed trajectory of \(u\) on a finite time horizon, sampled on a refining grid in space or in time. The statistical difficulty mirrors the analytic one. Since the solution is constructed for almost every realisation of the driving rough path, and is a measurable function of that rough path rather than an It\^o process adapted to a filtration generated by \(W\), the classical toolkit of SPDE inference---likelihood ratios, martingale limit theorems, spectral decomposition of the linear part---is not directly available. What survives is the pathwise structure itself, and it is precisely this structure that we exploit.

Parameter estimation for SPDEs is by now a substantial field; we refer to~\cite{Cianello} for a survey and to~\cite{Tudor2} for a monograph treatment of equations with additive Gaussian noise. We follow the variation-based estimation approach. It was introduced in~\cite{PospisilTribe2007}, where exact quadratic and quartic variations for the linear stochastic heat equation driven by space--time white noise are computed and used to estimate the parameters of the equation; see also~\cite{Sw} for a fine analysis of the variations of the solution. The method has since been extended in many directions: to fractional-colored noise~\cite{Torres}, to equations with piecewise constant coefficients~\cite{ZZ}, to nonlinear fractional heat equations~\cite{GT}, to temporal variations of arbitrary even order for the nonlinear stochastic heat equation~\cite{OT}, to the quartic variation of the semilinear heat equation and its asymptotic  independence from the data~\cite{CimpNachTudor}, and to the stochastic Allen--Cahn equation~\cite{AcostaOliveraTudor2026}. In all of these works the mechanism is the same one we use: the solution splits into an explicit Gaussian part carrying all the roughness and a remainder that is too smooth to affect the variation at leading order.

\medskip
A parallel line of research studies what can be estimated from realistic observation schemes. Discretely sampled solutions are treated in~\cite{Cianello2,Cianello3}; estimators based on discrete observations in both time and space are obtained in~\cite{Hil}; a two-dimensional parabolic model under small noise is studied in~\cite{To}. In the local measurement paradigm of~\cite{Alt,Gau}, one observes the solution integrated against localised test functions and lets the localisation parameter tend to zero, obtaining central limit theorems and optimal rates.

\medskip
Two works on Burgers-type equations are especially close to ours. The first,~\cite{AssaadTudor2021}, constructs estimators for the diffusivity of a stochastic Burgers equation driven by space--time white noise, from its quadratic and quartic variations; that architecture is the direct ancestor of ours. Our setting is more general: the nonlinearity considered there is of divergence form, so that the equation is solved in the mild It\^o sense and its solution is an It\^o process, whereas the transport term we allow admits no such interpretation and the solution exists only in the rough path sense. The estimates on the remainder are correspondingly obtained from bounds on rough integrals rather than from It\^o calculus. Two further differences are technical: the analysis there is carried out on the whole line whereas we work on the torus, and its nonlinearity is quadratic whereas we allow an arbitrary smooth \(g\).

The work~\cite{Janak} estimates parameters for a class of stochastic Burgers equations from local measurements, in the spirit of~\cite{Alt}. Those results and ours are complementary rather than competing. In a setting where the equation is well posed by classical means, the stronger statistical conclusions---asymptotic normality and explicit rates---are obtained there, whereas we prove only strong consistency, but for a class of equations that is analytically out of reach of those methods. Establishing which functionals of the trajectory are stable under the rough-path remainder is logically prior to any distributional theory: one cannot prove a central limit theorem for a statistic before knowing that its limit exists and is free of the remainder; a central limit theorem in the present setting remains open.

Our contribution is the following. We first show that the remainder \(R=U+v\) is, away from the initial time, Lipschitz in space and H\"older in time of order strictly greater than \(1/2\), hence strictly smoother than the Gaussian component \(\psi\) in both variables (Section~\ref{sec:regularidad}). We then compute exactly the spatial quadratic and temporal quartic variations of \(\psi\), almost surely rather than merely in \(L^{2}(\Omega)\), and combine them with an exact time-scaling identity to prove that the corresponding variations of \(u\) are left unchanged by the remainder (Sections~\ref{sec:variaciones}--\ref{sec:escaling}). Inverting the resulting limits yields strongly consistent pathwise estimators of
\(\theta\), built from a single observed component of the solution
(Section~\ref{sec:estimadores}).

Our method uses rough path analysis. It is needed to define the solution concept, and to supply the technical estimates from which the regularity of the remainder is derived. Everything downstream is classical Gaussian analysis.

The paper is organised as follows. Section~\ref{sec:preliminares} fixes notation, recalls the elements of rough path theory that we use, introduces the stationary process \(\psi\) and its lift, records the regularising properties of the heat semigroup, and states the notion of mild solution together with the well-posedness theorem of~\cite{Hairer2011}. Section~\ref{sec:regularidad} establishes the spatial and temporal H\"older regularity of the remainder \(R=U+v\). Section~\ref{sec:variaciones} computes the spatial quadratic and temporal quartic variations of \(\psi\), in \(L^{2}(\Omega)\) and almost surely. Section~\ref{sec:escaling} develops the time-scaling identity relating \(u_\theta\) to the unit-diffusivity equation, and Section~\ref{sec:estimadores} defines the estimators and proves their strong consistency. Three appendices collect technical material: scaling identities for the heat kernel (Appendix~\ref{appendix:A}), the boundedness of certain deterministic integrals (Appendix~\ref{appendix:B}), and the asymptotic expansion of the covariance function \(\Upsilon\) near the origin (Appendix~\ref{appendix:C}).

\section{Preliminaries}\label{sec:preliminares}
We gather here the material needed in the sequel: function spaces and notation, the elements of rough path theory that give meaning to the transport term, and the notion of mild solution that we adopt.

\subsection{Notation and function spaces}\label{subsec:notation}

For \(\alpha\in(0,1)\) we denote by \(C^{\alpha}\) the space of all \(\alpha\)-Hölder
continuous functions on \([0,2\pi]\) and by \(\|\cdot\|_{\alpha}\) the corresponding
seminorm,
\[
\|u\|_{\alpha}\;\overset{\mathrm{def}}{=}\;\sup_{x\neq y}\frac{|u(x)-u(y)|}{|x-y|^{\alpha}} .
\]
Abusing notation slightly, we write \(\|\cdot\|_{\infty}\) for the supremum norm and
set
\[
\|u\|_{C^{\alpha}}=\|u\|_{\alpha}+\|u\|_{\infty},
\]
which, on bounded intervals, is equivalent to \(\|u\|_{\alpha}+|u(0)|\). For an
integer \(k\ge 0\) we set
\[
\|u\|_{C^{k}}=\sum_{j=0}^{k}\|D^{j}u\|_{\infty},
\]
where \(D^{j}u\) denotes the \(j\)-th derivative of \(u\), and for \(k+\alpha\) with
\(k\in\mathbb{N}\) and \(\alpha\in(0,1)\) we write
\(\|u\|_{C^{k+\alpha}}=\|u\|_{C^{k}}+\|D^{k}u\|_{\alpha}\).

More generally, for an interval \(I\) and a finite-dimensional normed space \(E\) we
write \(C^{\alpha}(I;E)\) for the space of \(\alpha\)-Hölder continuous functions
from \(I\) into \(E\), with the same seminorm, \(|\cdot|\) denoting the norm of
\(E\); and \(C(I;E)=C^{0}(I;E)\) for the continuous functions from \(I\) into \(E\),
with the supremum norm. We omit \(I\) from the notation when \(I=[0,2\pi]\), and we
omit \(E\) whenever it is clear from the context; thus \(C^{\alpha}\) and \(C\)
without arguments always refer to functions on the torus. This convention is used
throughout: the same symbols will denote Hölder spaces in the spatial variable
(Sections~\ref{sec:preliminares} and~\ref{sec:variaciones}), in the parameter of a
rough path (Section~\ref{subsec:rp}), and in the time variable
(Proposition~\ref{prop:regularity_v}), the relevant interval being indicated
explicitly whenever it is not \([0,2\pi]\).

\medskip
We now fix the standing hypotheses of the paper.

\begin{assumption}\label{ass:standing}
	Throughout the paper we assume the following.
	\begin{enumerate}
		\item[(A1)] \(\theta>0\) and \(\sigma>0\) are constants, and the driving noise is
		the cylindrical Wiener process \(W\) on \(L^{2}(\mathbb{T};\mathbb{R}^{n})\)
		introduced in Section~\ref{subsec:psi_decomp}; equivalently, \(W\) gives rise to
		an \(\mathbb{R}^{n}\)-valued space--time white noise. The noise intensity
		\(\sigma\) is assumed to be known: only \(\theta\) is estimated.
		
		\item[(A2)] The coefficients \(f:\mathbb{R}^{n}\to\mathbb{R}^{n}\) and
		\(g:\mathbb{R}^{n}\to\mathbb{R}^{n\times n}\) are of class \(C^{\infty}\).
		Moreover \(g\) is bounded and all the derivatives of \(f\) and of \(g\) are
		bounded.
		
		\item[(A3)] The transport term is not of divergence form: there is no smooth
		\(G:\mathbb{R}^{n}\to\mathbb{R}^{n}\) such that
		\(g(u)\,\partial_x u=\partial_x\bigl(G(u)\bigr)\).
		
		\item[(A4)] The initial datum satisfies \(u_{0}\in C^{\beta}\) for some exponent
		\(\beta\in\bigl(\tfrac13,\tfrac12\bigr)\).
		
		\item[(A5)] We fix an exponent \(\alpha\) with \(\tfrac13<\alpha<\beta\).
	\end{enumerate}
\end{assumption}

Some comments on these hypotheses are in order. In (A2), the boundedness
requirements are what upgrade the local well-posedness of
Theorem~\ref{theorem:Existence_Uq} to a global solution on \([0,T]\); we use the
global statement without further comment. Only \(g\in C^{3}\) and \(f\in C^{1}\) are
needed for local existence, see \cite[Remark~3.10]{Hairer2011}. Hypothesis (A3) is
the standing structural assumption discussed in the introduction: when it fails, the
derivative may be moved onto the heat kernel in the mild formulation and the
equation is classically well posed; by the chain rule this forces \(n\ge2\), since
in dimension one every \(g\) admits an antiderivative. In (A5) the exponent
\(\alpha\) is not an additional hypothesis on the noise but a choice: by
Lemma~\ref{lem:rp_levantamiento} the stationary process \(\psi\) admits a lift of
regularity \(\alpha\) for \emph{every} \(\alpha<\tfrac12\), and \(\alpha>\tfrac13\)
is what the rough integral of Theorem~\ref{thm:rough_integral} requires. Finally,
since \(\psi_{0}\in C^{\beta}\) almost surely for every \(\beta<\tfrac12\), (A4)
gives \(u_{0}-\psi_{0}\in C^{\beta}\) almost surely, which is the form in which the
initial datum enters below.

\medskip
The heat semigroup \(\{S_t=e^{t\partial_x^{2}}\}_{t\ge0}\) on the torus
\(\mathbb{T}\) is characterised by two complementary regularising behaviours. On the
one hand it smooths instantaneously: for any \(t>0\), \(S_t\) maps \(L^{\infty}\)
into \(C^{\infty}\), and more generally \(C^{\eta}\) into \(C^{\kappa}\) whenever
\(0\le\eta<\kappa\),
\begin{equation}\label{eq:heat_holder_smoothing}
	\|S_t f\|_{C^{\kappa}}\;\le\;C\,t^{-\frac{\kappa-\eta}{2}}\,\|f\|_{C^{\eta}},
\end{equation}
which follows from the heat-kernel representation and standard Hölder-space theory
(see, e.g., \cite{Lunardi1995,Pazy1983,EngelNagel2000}).

On the other hand, the semigroup approximates the identity at a rate determined by
the regularity of \(f\): for \(\kappa\ge0\),
\begin{equation}\label{eq:heat_approx}
	\|(S_t-I)f\|_{L^{\infty}}\;\le\;C\,t^{\frac{\kappa}{2}}\,\|f\|_{C^{\kappa}}.
\end{equation}
Alongside the elementary \(L^{\infty}\)-contraction
\begin{equation}\label{eq:heat_contrac}
	\|S_t f\|_{L^{\infty}}\;\le\;\|f\|_{L^{\infty}},
\end{equation}
these estimates provide the core analytical toolkit for controlling the nonlinear
terms in the mild formulation of the SPDE (cf.\ \cite{Hairer2011,HairerWeber2013}).

\subsection{The stationary solution \texorpdfstring{$\psi$}{psi}}\label{subsec:psi_decomp}
This subsection introduces the three objects on which the decomposition
\(u=\psi+U+v\) rests: the heat semigroup and its damped version, the driving
noise, and the stationary process \(\psi\), which carries all the roughness of
the solution and whose exact covariance is what the estimation procedure
ultimately relies on.

\subsubsection*{The heat semigroup and the damped kernel}

We denote by \(\{S_t\}_{t\ge0}\), \(S_t=e^{t\partial_x^{2}}\), the heat semigroup on
the torus, characterised by
\[
(S_t u)(x)=\int_{\mathbb{T}} p_t(x-y)\,u(y)\,dy ,
\qquad t>0,\quad x\in\mathbb{T},
\]
for every \(u\in C(\mathbb{T};\mathbb{R}^{n})\), where the periodic heat kernel is
\[
p_t(x)=\frac{1}{\sqrt{4\pi t}}\sum_{k\in\mathbb{Z}}
e^{-\frac{(x-2\pi k)^{2}}{4t}},
\qquad t>0,\quad x\in\mathbb{T}.
\]
Alongside it we shall use the \emph{damped} semigroup
\(\{S^{E}_t\}_{t\ge0}\), \(S^{E}_t=e^{t(\partial_x^{2}-1)}\), generated by
\(\partial_x^{2}-1\), whose kernel is
\begin{equation}\label{eq:damped_kernel}
	p^{E}_t(x)=e^{-t}\,p_t(x),
	\qquad t>0,\quad x\in\mathbb{T},
\end{equation}
so that \(S^{E}_t=e^{-t}S_t\). The damping is introduced solely so that the linear
equation below possesses a stationary solution; see
Remark~\ref{rem:why_damping}.

\subsubsection*{The driving noise}

Let \(H:=L^{2}(\mathbb{T};\mathbb{R}^{n})\) with inner product
\(\langle\cdot,\cdot\rangle\), and let \((\Omega,\mathcal{F},\mathbb{P})\) be a
probability space carrying a standard cylindrical Wiener process \(W\) on \(H\),
indexed by \(t\in\mathbb{R}\), as in~\cite{Hairer2011,DPZ92}. We recall what this
means: \(W\) is a family \(\{\langle W(t),h\rangle : t\in\mathbb{R},\ h\in H\}\) of
centred Gaussian random variables, linear in \(h\), such that for every
\(h,k\in H\)
\begin{equation}\label{eq:cyl_cov}
	\mathbb{E}\bigl[\langle W(t),h\rangle\,\langle W(s),k\rangle\bigr]
	=(t\wedge s)\,\langle h,k\rangle ,
	\qquad t,s\ge0,
\end{equation}
with the analogous formula for negative times. The process \(W(t)\) does not take
values in \(H\), and \(\langle W(t),h\rangle\) is a notation for the real-valued
process \(W_t(h)\), not an inner product in \(H\); we nevertheless keep the
bracket, which is standard, and use the same symbol for the genuine \(H\)-pairing
of two functions in Section~\ref{subsec:mild}. Writing \(W^{i}\) for the
component acting on the \(i\)-th coordinate, \(W=(W^{1},\dots,W^{n})\) gives rise
to an \(\mathbb{R}^{n}\)-valued space--time white noise on
\(\mathbb{R}\times\mathbb{T}\).
Only the restriction of \(W\) to \([0,T]\times\mathbb{T}\) drives
equation~\eqref{eq:Burgers_intro}; the process is indexed by the whole time line
because the stationary solution~\eqref{eq:psi_mild} below is built from its entire
past.

We shall use two equivalent descriptions of \(W\), and it is convenient to fix them
here once and for all. The equations of this paper are written in the cylindrical
form \(dW(t)\) used in~\cite{Hairer2011}. For the explicit covariance computations
of Section~\ref{sec:variaciones} it is more convenient to use the associated
martingale measure, denoted by the same letter: for
\(h\in L^{2}(\mathbb{R}\times\mathbb{T})\) the Walsh integral
\(\int_{\mathbb{R}}\!\int_{\mathbb{T}} h(s,y)\,W^{i}(ds,dy)\) is well defined and
satisfies the isometry
\begin{equation}\label{eq:wn_cov}
	\mathbb{E}\Bigl[\int_{\mathbb{R}}\!\int_{\mathbb{T}} h\,W^{i}(ds,dy)
	\int_{\mathbb{R}}\!\int_{\mathbb{T}} k\,W^{j}(ds,dy)\Bigr]
	=\delta_{ij}\int_{\mathbb{R}}\!\int_{\mathbb{T}} h(s,y)\,k(s,y)\,ds\,dy ,
	\qquad 1\le i,j\le n ,
\end{equation}
see~\cite{Wal86}. The two descriptions agree in the sense that, for every
\(t\in\mathbb{R}\) and every \(x\in\mathbb{T}\),
\begin{equation}\label{eq:cyl_vs_walsh}
	\Bigl(\int_{-\infty}^{t} S^{E}_{t-s}\,dW(s)\Bigr)(x)
	=\int_{-\infty}^{t}\!\int_{\mathbb{T}} p^{E}_{t-s}(x-y)\,W(ds,dy)
	\qquad\text{almost surely.}
\end{equation}
In particular the components \(W^{1},\dots,W^{n}\) are independent, which is what
Lemma~\ref{lem:psi_decoupling} below rests on.

\subsubsection*{The stationary solution}

Consider the linear damped equation
\begin{equation}\label{eq:estacionario_H1}
	d\psi = (\partial_x^{2}-1)\,\psi\,dt + \sigma\,dW(t),
	\qquad (t,x)\in\mathbb{R}\times\mathbb{T}.
\end{equation}
It admits the stationary mild solution
\begin{equation}\label{eq:psi_mild}
	\psi(t,x)=\sigma\int_{-\infty}^{t}\!\int_{\mathbb{T}} p^{E}_{t-s}(x-y)\,W(ds,dy),
	\qquad (t,x)\in\mathbb{R}\times\mathbb{T},
\end{equation}
the stochastic integral being well defined by~\eqref{eq:wn_cov}, since
\(\int_{-\infty}^{t}\!\int_{\mathbb{T}}p^{E}_{t-s}(x-y)^{2}\,dy\,ds<\infty\).

\begin{remark}
\label{not:psi_t}
	For a space--time function \(w\) we write interchangeably \(w(t,x)\) and
	\(w_t(x)\); the subscript form is used whenever \(w_t\) is regarded as a function of
	the spatial variable alone, as in \(\psi_t\in C^{\alpha}(\mathbb{T})\), and the
	two-argument form inside integrals and covariances. The same convention applies to
	\(u\), \(v\), \(U\) and \(R\).
\end{remark}

\begin{remark}\label{rem:why_damping}
	The stationarity of \(\psi\) is essential to what follows, and it is what the damping
	term \(-\psi\) in~\eqref{eq:estacionario_H1} is for. It also fails if the noise is
	switched on at \(t=0\): the solution of~\eqref{eq:estacionario_H1} started from
	\(\psi_0=0\) has covariance
	\[
	\frac{1}{4\pi}\sum_{k\in\mathbb{Z}}
	\frac{e^{-(k^{2}+1)|t-s|}-e^{-(k^{2}+1)(t+s)}}{k^{2}+1},
	\]
	which does not depend on \(t-s\) alone. The second term would contribute to the
	variations computed in Section~\ref{sec:variaciones} and destroy the exact constants
	obtained there. This is why \(W\) is indexed by \(\mathbb{R}\) and the integral
	in~\eqref{eq:psi_mild} starts at \(-\infty\). The constant \(1\) in the damping plays
	no role: the notion of solution adopted in Section~\ref{subsec:mild} is independent
	of it, the compensating term being absorbed into the nonlinearity
	(see~\cite[Remark~3.5]{Hairer2011}).
\end{remark}

The process \(\psi\) is a centred Gaussian field whose components have covariance
(the computation is carried out in Section~\ref{sec:variaciones})
\begin{equation}\label{eq:psi_cov}
	\mathbb{E}\bigl[\psi^{i}(t,x)\,\psi^{j}(s,y)\bigr]
	=\delta_{ij}\,\frac{\sigma^{2}}{4\pi}\sum_{k\in\mathbb{Z}}
	\frac{e^{-(1+k^{2})|t-s|}}{1+k^{2}}\,e^{ik(x-y)},
	\qquad 1\le i,j\le n .
\end{equation}
The factor \(\delta_{ij}\) in \ref{eq:psi_cov} reflects the following elementary fact, which is what
allows the variations of Section~\ref{sec:variaciones} to be computed for a scalar
field.

\begin{lemma}[Decoupling of the components of \(\psi\)]\label{lem:psi_decoupling}
	The components \(\psi^{1},\dots,\psi^{n}\) of the stationary process \(\psi\) are
	independent, and each of them solves the scalar equation
	\[
	d\psi^{i}=(\partial_x^{2}-1)\,\psi^{i}\,dt+\sigma\,dW^{i}(t),
	\qquad i=1,\dots,n .
	\]
	In particular they are independent copies of one and the same scalar centred
	Gaussian field.
\end{lemma}

\begin{proof}
	The operator \(\partial_x^{2}-1\) acts diagonally on the components, so the \(i\)-th
	component of the mild formula~\eqref{eq:psi_mild} reads
	\[
	\psi^{i}(t,x)=\sigma\int_{-\infty}^{t}\!\int_{\mathbb{T}}
	p^{E}_{t-s}(x-y)\,W^{i}(ds,dy)
	\]
	and involves only \(W^{i}\). The independence of \(\psi^{1},\dots,\psi^{n}\) is then
	immediate from~\eqref{eq:wn_cov}, and the common law does not depend on \(i\)
	because the \(W^{i}\) are identically distributed.
\end{proof}

\subsection{Elements of rough path theory}\label{subsec:rp}

We present the notions of rough path theory used in the sequel, keeping the notation
of~\cite{Hairer2011}, which is the closest reference for this work. The subsection
closes with the two results on which Section~\ref{sec:regularidad} rests: the
estimate of Proposition~\ref{prop:hk_result} and the lift of \(\psi\)
(Lemma~\ref{lem:rp_levantamiento}).

\begin{remark}\label{not:spatial_parameter}
	Throughout this subsection the paths are parametrised by the \emph{spatial}
	variable: the rough path we shall eventually use is \(x\mapsto\psi_t(x)\) on
	\([0,2\pi]\), for each fixed time \(t\). To keep the notation of~\cite{Hairer2011}
	we denote the parameter by \(s,t,u\), but in this subsection these letters
	denote points of \([0,2\pi]\) and not times.
\end{remark}

Following the convention of Section~\ref{subsec:notation}, \(C(I;E)\) denotes the
continuous functions from \(I\) into \(E\). We write in addition \(C_2(I;E)\) for
the continuous functions from \(I^{2}\) into \(E\) that vanish on the diagonal, and
we omit \(I\) when \(I=[0,2\pi]\) and \(E\) whenever it is clear from the context.
The difference operator \(\delta:C(I;E)\to C_2(I;E)\) is defined by
\[
\delta X_{s,t}=X_t-X_s .
\]
In this subsection \(E\) will be \(\mathbb{R}^{n}\) for the path \(X\),
\(\mathbb{R}^{n\times n}\) for its area process \(\mathbb{X}\), and
\(\mathbb{R}^{m}\) for a path controlled by \(X\). The difference
operator \(\delta:C\to C_2\) is defined by
\[
\delta X_{s,t}=X_t-X_s .
\]

\begin{definition}[Rough path]\label{def:rough_path}
	A \emph{rough path} consists of two parts:
	\begin{enumerate}
		\item a continuous function \(X:[0,2\pi]\to\mathbb{R}^{n}\);
		\item a continuous \emph{area process}
		\(\mathbb{X}:[0,2\pi]^{2}\to\mathbb{R}^{n\times n}\) with
		\(\mathbb{X}_{t,t}=0\) for every \(t\), such that the algebraic relations
		\begin{equation}\label{eq:chen}
			\mathbb{X}^{ij}_{s,t}-\mathbb{X}^{ij}_{u,t}-\mathbb{X}^{ij}_{s,u}
			=\delta X^{i}_{s,u}\,\delta X^{j}_{u,t}
		\end{equation}
		hold for every triple \((u,s,t)\) and every pair of indices \((i,j)\).
	\end{enumerate}
	We write \(\mathbf{X}=(X,\mathbb{X})\) for the pair. The quantity \(\mathbb{X}\)
	should be interpreted as postulating the value of
	\begin{equation}\label{eq:iterated}
		\int_s^t \delta X^{i}_{s,r}\,dX^{j}_r\;\overset{\mathrm{def}}{=}\;\mathbb{X}^{ij}_{s,t},
	\end{equation}
	the right-hand side being taken as a \emph{definition} of the left-hand side, and not
	the other way around.
\end{definition}

\begin{remark}\label{rem:chen_not_enough}
	Chen's relations~\eqref{eq:chen} do not determine \(\mathbb{X}\) as a function of
	\(X\): for any \(F:[0,2\pi]\to\mathbb{R}^{n\times n}\), the substitution
	\(\mathbb{X}^{ij}_{s,t}\mapsto\mathbb{X}^{ij}_{s,t}+F^{ij}_t-F^{ij}_s\) leaves them
	invariant. Their purpose is to ensure that~\eqref{eq:iterated} behaves like an
	integral over two adjacent intervals. The choice of \(\mathbb{X}\) is therefore part
	of the data, and it is this choice that
	Lemma~\ref{lem:rp_levantamiento} fixes for \(\psi\).
\end{remark}

For \(\alpha\in(0,\tfrac12)\) we denote by \(\mathcal{D}^{\alpha}\) the set of rough
paths \(\mathbf{X}=(X,\mathbb{X})\) such that \(X\in C^{\alpha}\) and
\[
\|\mathbb{X}\|_{2\alpha}:=\sup_{s\neq t\in[0,2\pi]}
\frac{|\mathbb{X}_{s,t}|}{|s-t|^{2\alpha}}<\infty ,
\]
\(|\cdot|\) being the Euclidean norm on \(\mathbb{R}^{n\times n}\). The rough
integral constructed below additionally requires \(\alpha>\tfrac13\); we shall
indicate explicitly where this is used.

\begin{remark}\label{rem:not_vector_space}
	Although \(\mathcal{D}^{\alpha}\) is a closed subset of a vector space, it is not
	itself a vector space, because of the nonlinear constraint~\eqref{eq:chen}. A
	natural quantity on \(\mathcal{D}^{\alpha}\) is
	\[
	\vert\!\vert\!\vert\mathbf{X}\vert\!\vert\!\vert_{\alpha}
	=\|X\|_{\alpha}+\sqrt{\|\mathbb{X}\|_{2\alpha}}\,,
	\]
	which respects the natural dilation \((X,\mathbb{X})\mapsto(\lambda X,\lambda^{2}\mathbb{X})\)
	(see, e.g., \cite{FrizHairer2020}).
\end{remark}

\begin{definition}[Controlled rough path]\label{def:controlled}
	Let \(\mathbf{X}\in\mathcal{D}^{\alpha}\). We denote by
	\(C^{\alpha}_{X}([0,2\pi];\mathbb{R}^{m})\) the set of pairs \((Y,Y')\) with
	\[
	Y\in C^{\alpha}([0,2\pi];\mathbb{R}^{m}),
	\qquad
	Y'\in C^{\alpha}([0,2\pi];\mathbb{R}^{m\times n}),
	\]
	such that the \emph{remainder} \(\bar R\in C_2\) defined by
	\begin{equation}\label{eq:controlled_remainder}
		\bar R_{s,t}=\delta Y_{s,t}-Y'_s\,\delta X_{s,t}
	\end{equation}
	satisfies \(\|\bar R\|_{2\alpha}<\infty\). An element of
	\(C^{\alpha}_{X}\) is called a \emph{path controlled by \(X\)}, and \(Y'\) is called
	its \emph{derivative process}. We endow \(C^{\alpha}_{X}\) with the norm
	\[
	\|(Y,Y')\|_{X,\alpha}=\|Y\|_{C^{\alpha}}+\|Y'\|_{C^{\alpha}}+\|\bar R\|_{2\alpha}.
	\]
\end{definition}

Thus \(C^{\alpha}_{X}\) is defined by Definition~\ref{def:controlled}; the subscript
records the rough path \(X\) against which the pair is controlled, and must not be
confused with the interval or target space of Section~\ref{subsec:notation}. As is
customary, we often write \(\|Y\|_{X,\alpha}\) for \(\|(Y,Y')\|_{X,\alpha}\), the
derivative process being clear from the context.

\begin{remark}\label{rem:Y_alpha}
	Since \(X\) is \(\alpha\)-Hölder, one verifies directly
	from~\eqref{eq:controlled_remainder} that
	\(\|Y\|_{\alpha}\le C\|\bar R\|_{2\alpha}+\|Y'\|_{\infty}\|X\|_{\alpha}\);
	the term \(\|Y\|_{C^{\alpha}}\) in the norm therefore serves only to control the
	supremum of \(Y\).
\end{remark}

Three constructions of controlled paths will be used.

\emph{(i) The path \(X\) itself.} One identifies \(X\) with the pair
\((X,I)\in C^{\alpha}_{X}\), where \(I\) is the constant path equal to the identity
matrix; the corresponding remainder vanishes.

\emph{(ii) Regular paths.} There is a canonical embedding
\(\iota:C^{2\alpha}\hookrightarrow C^{\alpha}_{X}\), \(\iota Y=(Y,0)\); indeed
\(\bar R=\delta Y\) then satisfies \(\|\bar R\|_{2\alpha}<\infty\) (recall
\(\alpha<\tfrac12\)). If moreover \(Y\in C^{1}\), the integral of \(Y\) against
\(X\) may be defined by
\[
Z_t=\int_0^t Y_s\otimes dX_s
= Y_t\otimes X_t-Y_0\otimes X_0-\int_0^t \dot Y_s\otimes X_s\,ds ,
\]
and \((Z,Y\otimes I)\) again belongs to \(C^{\alpha}_{X}\).

\emph{(iii) Composition with regular functions.} Let
\(\phi:\mathbb{R}^{m}\times[0,2\pi]\to\mathbb{R}^{n}\) be uniformly \(C^{2}\) in its
first argument and uniformly \(C^{2\alpha}\) in its second, and let
\((Y,Y')\in C^{\alpha}_{X}([0,2\pi];\mathbb{R}^{m})\). Then
\begin{equation}\label{eq:composition}
	\phi(Y)_t=\phi(Y_t,t),
	\qquad
	\phi(Y)'_t=D_y\phi(Y_t,t)\,Y'_t ,
\end{equation}
defines an element \((\phi(Y),\phi(Y)')\) of
\(C^{\alpha}_{X}([0,2\pi];\mathbb{R}^{n})\), where \(D_y\) is the derivative with
respect to the first argument.

\begin{lemma}\label{lem:rp_composition}
	Let \(\phi\), \((Y,Y')\) and \((\phi(Y),\phi(Y)')\) be as above. Then there is a
	constant \(C\) such that
	\[
	\|\phi(Y)\|_{X,\alpha}\le C\bigl(\|D_y^{2}\phi\|_{\infty}+\|\phi\|_{\infty}
	+\|\phi\|_{2\alpha;t}\bigr)\bigl(1+\|Y\|_{X,\alpha}\bigr)^{2},
	\]
	where \(\|\phi\|_{2\alpha;t}\) denotes the supremum over \(y\) of the
	\(2\alpha\)-Hölder norm of \(\phi(y,\cdot)\).
\end{lemma}

In particular, if \(f\in C^{2\alpha}\) and \((Y,Y')\in C^{\alpha}_{X}\), then
\(f\cdot Y\) and \(f+Y\) belong to \(C^{\alpha}_{X}\), with
\begin{align}
	\|(fY,fY')\|_{X,\alpha}&\le 2\|f\|_{C^{2\alpha}}\|(Y,Y')\|_{X,\alpha},
	\label{eq:algebra}\\
	\|(f+Y,Y')\|_{X,\alpha}&\le 2\|f\|_{C^{2\alpha}}+\|(Y,Y')\|_{X,\alpha};
	\nonumber
\end{align}
so \(C^{\alpha}_{X}\) is an algebra for every reference rough path \(X\).

\subsubsection*{Integration of controlled rough paths}

We now give a meaning to \(\int Y_t\otimes dX_t\) for
\(\mathbf{X}\in\mathcal{D}^{\alpha}\) and \((Y,Y')\in C^{\alpha}_{X}\). The naive
definition as a limit of Riemann sums,
\begin{equation}\label{eq:riemann_naive}
	\int_0^{2\pi} Y_t\otimes dX_t
	=\lim_{|\mathcal{P}|\to0}\sum_{[s,t]\in\mathcal{P}} Y_s\otimes\delta X_{s,t},
\end{equation}
where \(\mathcal{P}\) is a partition of \([0,2\pi]\) and \(|\mathcal{P}|\) the
length of its largest element, does not converge in general. Using the
approximation \(Y_t\approx Y_s+Y'_s\,\delta X_{s,t}\) suggested
by~\eqref{eq:controlled_remainder} together with~\eqref{eq:iterated}, one defines
instead
\begin{equation}\label{eq:rough_riemman}
	\int_0^{2\pi} Y_t\otimes dX_t
	=\lim_{|\mathcal{P}|\to0}\sum_{[s,t]\in\mathcal{P}}
	\bigl(Y_s\otimes\delta X_{s,t}+Y'_s\,\mathbb{X}_{s,t}\bigr).
\end{equation}

\begin{theorem}\label{thm:rough_integral}
	Let \(\mathbf{X}\in\mathcal{D}^{\alpha}\) with \(\alpha>\tfrac13\). Then the map
	\[
	(Y,Y')\longmapsto\Bigl(\int_0^{\cdot} Y_t\otimes dX_t,\;Y\otimes I\Bigr),
	\]
	with the integral defined by~\eqref{eq:rough_riemman}, is continuous from
	\(C^{\alpha}_{X}\) to \(C^{\alpha}_{X}\), and
	\begin{equation}\label{eq:rough_int_bound}
		\Bigl\|\int_0^{\cdot}\delta Y_{0,t}\otimes dX_t\Bigr\|_{\alpha}
		\le C\bigl(\|X\|_{\alpha}\|\bar R\|_{2\alpha}
		+\|\mathbb{X}\|_{2\alpha}\|Y'\|_{C^{\alpha}}\bigr)
	\end{equation}
	for some constant \(C>0\).
\end{theorem}

The estimate on which Section~\ref{sec:regularidad} depends concerns the integral of
a controlled path against a rescaled kernel. For \(f\in C^{1}(\mathbb{R})\) set
\begin{equation}\label{eq:norm11}
	\|f\|_{1,1}:=\sum_{m\in\mathbb{Z}}\sup_{0\le s\le 2\pi}
	\bigl(|f(2\pi m+s)|+|f'(2\pi m+s)|\bigr).
\end{equation}

\begin{proposition}\label{prop:hk_result}
	Let \(\alpha>\tfrac13\), let \(\mathbf{X}\in\mathcal{D}^{\alpha}\), let
	\((Y,Y')\in C^{\alpha}_{X}\) and let \(f\in C^{1}(\mathbb{R})\) with
	\(\|f\|_{1,1}<\infty\). There is a universal constant \(C_{\alpha}\) such that
	\begin{equation}\label{eq:hk_ineq}
		\Bigl\|\int_0^{2\pi} f(\lambda t)\,Y(t)\otimes dX(t)\Bigr\|
		\le C_{\alpha}\,\lambda^{-\alpha}\,\|Y\|_{X,\alpha}\,\|f\|_{1,1}
		\bigl(\|X\|_{\alpha}+\|\mathbb{X}\|_{2\alpha}\bigr)
	\end{equation}
	for every \(\lambda\ge1\), the integral being understood as
	in~\eqref{eq:rough_riemman}.
\end{proposition}

The gain in~\eqref{eq:hk_ineq} is the factor \(\lambda^{-\alpha}\): concentrating the
kernel improves the bound, at a rate governed by the roughness of \(X\). Since the
heat kernel at time \(t\) concentrates at scale \(\sqrt{t}\), this is what converts
the smoothing of the semigroup into regularity of the rough integral; it is applied
in exactly that way in Section~\ref{sec:regularidad}, through the scaling identities
of Lemma~\ref{lem:heat_prop}. A more general statement, with corollaries adapted to
the heat kernel, is given in~\cite[Corollary~3.16]{Hairer2013}; the form above
suffices here.

\subsubsection*{The lift of the stationary process}

It remains to produce a rough path from \(\psi\). Recall from
Notation~\ref{not:spatial_parameter} that the parameter is the spatial variable:
for each fixed time \(t\), the object to be lifted is the \(\mathbb{R}^{n}\)-valued
path
\[
\Psi_t:[0,2\pi]\to\mathbb{R}^{n},
\qquad \Psi_t(x):=\psi_t(x),
\]
which is the first level of the lift; a distinct symbol is kept to signal that
\(\psi_t\) is being regarded as a rough path. What has to be supplied is a second
level \(\boldsymbol{\Psi}_t\), and by Remark~\ref{rem:chen_not_enough} this is a
genuine choice. The next lemma makes it, in the way that is canonical for a Gaussian
process.

\begin{lemma}[Canonical lift of \(\psi\)]\label{lem:rp_levantamiento}
	Let \(\alpha\in(\tfrac13,\tfrac12)\) and let \(\psi^{\varepsilon}_t\) be any sequence
	of processes, smooth in the spatial variable, such that
	\(\sup_{t\le T}\mathbb{E}\,|\psi^{\varepsilon}_t-\psi_t|^{2}\to0\) as
	\(\varepsilon\to0\). Then:
	\begin{enumerate}
		\item[(i)] for every \(t\ge0\) and every \(x,y\in[0,2\pi]\), the limit
		\begin{equation}\label{eq:canonical_lift}
			\boldsymbol{\Psi}_t(x,y)
			:=\lim_{\varepsilon\to0}\int_x^y
			\bigl(\psi^{\varepsilon}_t(r)-\psi^{\varepsilon}_t(x)\bigr)
			\otimes\partial_r\psi^{\varepsilon}_t(r)\,dr
		\end{equation}
		exists almost surely and does not depend on the approximating sequence;
		\item[(ii)] almost surely, \((\Psi_t,\boldsymbol{\Psi}_t)\in\mathcal{D}^{\alpha}\)
		for every \(t\ge0\);
		\item[(iii)] almost surely, the map
		\(t\mapsto(\Psi_t,\boldsymbol{\Psi}_t)\) is Hölder continuous from
		\(\mathbb{R}_{+}\) into \(\mathcal{D}^{\alpha}\).
	\end{enumerate}
	The pair \((\Psi_t,\boldsymbol{\Psi}_t)\) is called the \emph{canonical lift} of
	\(\psi_t\); it coincides with the lift obtained by Stratonovich integration.
\end{lemma}

\begin{proof}
	For fixed \(t\), the path \(x\mapsto\psi_t(x)\) is a centred Gaussian process on
	\([0,2\pi]\) whose covariance is, by~\eqref{eq:psi_cov}, a function of \(x-y\) that
	is smooth away from the origin with a jump in its first derivative at the origin. Its
	two-dimensional \(\varrho\)-variation is therefore finite for some
	\(\varrho<2\), and (i)--(ii) follow from~\cite[Theorem~35]{FrizVictoir2008};
	see also~\cite[Remark~3.2]{Hairer2011}. For (iii), the covariance
	\eqref{eq:psi_cov} gives
	\(\mathbb{E}\,|\psi(t,x)-\psi(s,x)|^{2}\le C|t-s|^{1/2}\),
	whence by~\cite[Theorem~37]{FrizVictoir2008} there are \(\vartheta>0\) and
	constants \(C_q\) with
	\(\mathbb{E}\bigl(\|\Psi_s-\Psi_t\|_{\alpha}^{2q}
	+\|\boldsymbol{\Psi}_s-\boldsymbol{\Psi}_t\|_{2\alpha}^{q}\bigr)
	\le C_q|t-s|^{\vartheta q}\)
	for every \(q\ge1\); Kolmogorov's continuity criterion then gives the claim. The
	details are those of~\cite[Lemma~3.1]{Hairer2011}.
\end{proof}

\begin{remark}\label{rem:lift_meaning}
	The content of~\eqref{eq:canonical_lift} is that the second level is \emph{not}
	postulated arbitrarily: it is the limit of the ordinary iterated integrals of smooth
	approximations of \(\psi_t\), and that limit is the same for all of them. Different
	choices of \(\boldsymbol{\Psi}\) would give rise to different solutions
	of~\eqref{eq:Burgers_intro} — this is precisely the ill-posedness described in the
	introduction — and~\eqref{eq:canonical_lift} is the choice singled out by
	approximation of the noise. Everything proved below about the remainder holds for the
	solution associated with any lift in \(\mathcal{D}^{\alpha}\); we fix the canonical
	one for concreteness.
\end{remark}

\subsection{Adopted notion of mild solution}\label{subsec:mild}

In this subsection we take \(\theta=1\); the general case is reduced to it in
Section~\ref{sec:escaling} by a time change. We thus consider, on a fixed horizon
\(T>0\),
\begin{equation}\label{eq:Burgers_theta1}
	du=\partial_x^{2}u\,dt+f(u)\,dt+g(u)\,\partial_x u\,dt+\sigma\,dW(t),
	\qquad (t,x)\in[0,T]\times\mathbb{T},
\end{equation}
with initial datum \(u_0\). We first recall the notions of weak and mild solution
of~\cite{Hairer2011}, and then introduce the variant that we adopt, which separates
the contribution of the initial datum from the remainder.

\begin{definition}[Weak solution, \cite{Hairer2011}]\label{def:Hairer_weak}
	A continuous stochastic process \(u\) is a \emph{weak solution}
	of~\eqref{eq:Burgers_theta1} if \(\hat v:=u-\psi\) belongs to
	\(C\bigl([0,T];C\bigr)\cap L^{1}\bigl([0,T];C^{1}\bigr)\) and, for every smooth
	periodic test function \(\phi:\mathbb{T}\to\mathbb{R}^{n}\), the identity
	\begin{equation}\label{eq:weak_H}
		\begin{aligned}
			\langle\hat v_t,\phi\rangle
			&=\langle u_0-\psi_0,\phi\rangle
			+\int_0^t\bigl\langle(\partial_x^{2}-1)\phi,\hat v_s\bigr\rangle\,ds
			+\int_0^t\bigl\langle\phi,\,g(u_s)\,\partial_x\hat v_s\bigr\rangle\,ds\\
			&\quad+\int_0^t\!\int_0^{2\pi}\phi(x)\,g\bigl(\hat v_s(x)+\Psi_s(x)\bigr)\,d\Psi_s(x)\,ds
			+\int_0^t\bigl\langle\phi,\hat f(u_s)\bigr\rangle\,ds
		\end{aligned}
	\end{equation}
	holds almost surely for all \(t\in[0,T]\), where \(\hat f(u)=f(u)+u\) and
	\(\langle\cdot,\cdot\rangle\) is the \(L^{2}(\mathbb{T};\mathbb{R}^{n})\) pairing.
\end{definition}

The inner integral in the second line is the rough integral of
Theorem~\ref{thm:rough_integral}: since \(\hat v_s\in C^{1}\), the path
\(x\mapsto\phi(x)\,g(\hat v_s(x)+\Psi_s(x))\) is controlled by \(\Psi_s\) by the
constructions (ii) and (iii) of Section~\ref{subsec:rp}, and its integral is a
measurable function of \(s\) bounded in terms of \(\|\Psi_s\|_{C^\alpha}\),
\(\|\boldsymbol\Psi_s\|_{2\alpha}\) and \(\|\hat v_s\|_{C^{1}}\), so that the outer
integral makes sense (\cite[Remark~3.4]{Hairer2011}).

\begin{definition}[Mild solution, \cite{Hairer2011}]\label{def:Hairer_mild}
	A continuous stochastic process \(u\) is a \emph{mild solution}
	of~\eqref{eq:Burgers_theta1} if \(\hat v:=u-\psi\) belongs to
	\(C\bigl([0,T];C\bigr)\cap L^{1}\bigl([0,T];C^{1}\bigr)\) and, for every
	\(x\in\mathbb{T}\) and \(t\in(0,T]\), the identity
	\begin{equation}\label{eq:mild_H}
		\begin{aligned}
			\hat v_t(x)
			&=\bigl(S_t(u_0-\psi_0)\bigr)(x)
			+\int_0^t\Bigl(S_{t-s}\bigl(g(u_s)\,\partial_x\hat v_s+f(u_s)+\psi_s\bigr)\Bigr)(x)\,ds\\
			&\quad+\int_0^t\!\int_0^{2\pi}p_{t-s}(x-y)\,g\bigl(u_s(y)\bigr)\,d\Psi_s(y)\,ds
		\end{aligned}
	\end{equation}
	holds almost surely.
\end{definition}

\begin{proposition}[\cite{Hairer2011}, Proposition~3.6]\label{prop:weak_mild}
	Every mild solution is a weak solution and conversely.
\end{proposition}

\begin{remark}[On the term \(\psi_s\)]\label{rem:compensating_term}
	The term \(\psi_s\) in~\eqref{eq:mild_H} is not a typographical accident. Since
	\(\psi\) solves the \emph{damped} equation~\eqref{eq:estacionario_H1}, subtracting it
	from~\eqref{eq:Burgers_theta1} leaves
	\[
	\partial_t\hat v=\partial_x^{2}\hat v+f(u)+g(u)\,\partial_x u+\psi ,
	\]
	the last term compensating the damping \(-\psi\) in~\eqref{eq:estacionario_H1};
	\eqref{eq:mild_H} is the mild form of this equation for the heat semigroup \(S_t\),
	the term \(g(u)\partial_x\psi\) being the rough integral. In~\cite{Hairer2011} the
	same compensation is written as \(\hat f(u)=f(u)+u\) together with the damped
	operator \(\partial_x^{2}-1\), as in~\eqref{eq:weak_H}; the two bookkeepings are
	equivalent, since \((\partial_x^{2}-1)\hat v+u=\partial_x^{2}\hat v+\psi\).
\end{remark}

We now introduce the notion of solution adopted throughout. It differs from
Definition~\ref{def:Hairer_mild} only in that the contribution of the initial datum
is subtracted as well, so that the remainder starts from zero; an analogous
formulation is used in~\cite[Section~3]{HairerMatetski2016} and, for a related
equation, in~\cite[Section~5]{HairerWeber2013}. For \(v\in C\bigl([0,T];C^{1}\bigr)\)
we write
\begin{equation}\label{eq:norm_1T}
	\|v\|_{1,T}:=\sup_{t\in[0,T]}\|v_t\|_{C^{1}} .
\end{equation}

\begin{definition}[Adopted mild solution]\label{def:mild_principal}
	Set \(U_t:=S_t(u_0-\psi_0)\) for \(t\in[0,T]\). A continuous stochastic process \(u\)
	is a \emph{mild solution} of~\eqref{eq:Burgers_theta1} if the function
	\[
	v:=u-\psi-U
	\]
	belongs to \(C\bigl([0,T];C^{1}\bigr)\), so that \(\|v\|_{1,T}<\infty\), and, for
	every \(x\in\mathbb{T}\) and \(t\in[0,T]\), the identity
	\begin{equation}\label{eq:mild_principal}
		\begin{aligned}
			v_t(x)
			&=\int_0^t\Bigl(S_{t-s}\bigl(g(u_s)\,(\partial_x v_s+\partial_x U_s)
			+f(u_s)+\psi_s\bigr)\Bigr)(x)\,ds\\
			&\quad+\int_0^t\!\int_0^{2\pi}p_{t-s}(x-y)\,g\bigl(u_s(y)\bigr)\,d\Psi_s(y)\,ds
		\end{aligned}
	\end{equation}
	holds almost surely. We denote the first integral on the right-hand side by
	\(\mathcal{M}^{(1)}_{\Psi}v\) and the second, the rough integral, by
	\(\mathcal{M}^{(2)}_{\Psi}v\), so that~\eqref{eq:mild_principal} reads
	\(v=\mathcal{M}^{(1)}_{\Psi}v+\mathcal{M}^{(2)}_{\Psi}v\).
\end{definition}

\begin{remark}\label{rem:why_this_definition}
	Definitions~\ref{def:Hairer_mild} and~\ref{def:mild_principal} are equivalent:
	\(\hat v=U+v\), and \(U\) satisfies the homogeneous heat equation, so
	\eqref{eq:mild_H} for \(\hat v\) and~\eqref{eq:mild_principal} for \(v\) differ by
	the identity \(S_t(u_0-\psi_0)=U_t\) together with
	\(\partial_x\hat v_s=\partial_x v_s+\partial_x U_s\). We adopt the second because it
	lets the initial-datum contribution \(U\), whose regularity degenerates as
	\(t\to0^{+}\), and the remainder \(v\), which is regular on all of \([0,T]\), be
	analysed independently in Section~\ref{sec:regularidad}.
\end{remark}

\begin{theorem}[Well-posedness, \cite{Hairer2011}]\label{theorem:Existence_Uq}
	Let Assumption~\ref{ass:standing} hold. Then for almost every realisation of the
	driving rough path \(\Psi\) there exists \(T_{\ast}>0\), depending on the realisation,
	such that equation~\eqref{eq:Burgers_theta1} has a unique mild solution on
	\([0,T_{\ast}]\) in the sense of Definition~\ref{def:mild_principal}, taking values
	in \(C\bigl([0,T_{\ast}];C^{\beta}\bigr)\). Under the boundedness hypotheses of
	Assumption~\ref{ass:standing}(A2) the solution is global: for every \(T>0\) it exists
	and is unique on \([0,T]\), with \(T\) independent of \(\Psi\).
\end{theorem}

\begin{proof}
	The solution is obtained by a Picard iteration on the fixed-point
	formulation~\eqref{eq:mild_principal}; see~\cite[Theorem~3.7]{Hairer2011}. Local
	existence requires only \(u_0\in C^{\beta}\) (and, by
	\cite[Remark~3.10]{Hairer2011}, only \(g\in C^{3}\) and \(f\in C^{1}\)); globality
	uses that \(g\) and the derivatives of \(f\) and \(g\) are bounded.
\end{proof}

Throughout the rest of the paper we use the global statement: \(T>0\) is the fixed
horizon of~\eqref{eq:Burgers_theta1}, and \(u\), \(v\), \(U\) are defined on all of
\([0,T]\).

\section{Regularity of the remainder}\label{sec:regularidad}

In this section we prove that the remainder \(R=U+v\) in the decomposition
\(u=\psi+U+v\) is strictly more regular than \(\psi\), in space and in time. This
is the analytic input of the estimation procedure: the variations of
Section~\ref{sec:variaciones} are computed for \(\psi\), and the results of this
section are what allow them to be transferred to \(u\) in
Section~\ref{sec:estimadores}.

Recall from Definition~\ref{def:mild_principal} that \(v\) solves the fixed-point
equation \(v=\mathcal{M}^{(1)}_{\Psi}v+\mathcal{M}^{(2)}_{\Psi}v\), where
\[
\bigl(\mathcal{M}^{(1)}_{\Psi}v\bigr)(t,x)
=\int_0^t\Bigl(S_{t-s}\bigl(g(u_s)(\partial_x v_s+\partial_x U_s)+f(u_s)+\psi_s\bigr)\Bigr)(x)\,ds
\]
is a Bochner integral of the heat semigroup applied to a bounded function, and
\[
\bigl(\mathcal{M}^{(2)}_{\Psi}v\bigr)(t,x)
=\int_0^t\!\int_0^{2\pi}p_{t-s}(x-y)\,g\bigl(u_s(y)\bigr)\,d\Psi_s(y)\,ds
\]
has as inner integral the rough integral of Theorem~\ref{thm:rough_integral}. The
two are treated by different means: the first by the Schauder
estimates~\eqref{eq:heat_holder_smoothing}--\eqref{eq:heat_approx}, the second by
Proposition~\ref{prop:hk_result}. The spatial regularity of \(R\) is immediate from
the solution concept and is recorded in Corollary~\ref{cor:R_properties} at the
end of the section; the substance of the section is the temporal regularity,
Proposition~\ref{prop:regularity_v}.

Throughout the section \(T>0\) is the fixed horizon of Section~\ref{subsec:mild},
and \(u\), \(v\), \(U\) are defined on all of \([0,T]\) by
Theorem~\ref{theorem:Existence_Uq}. We set
\begin{equation}\label{eq:KT}
	K_T:=\sup_{t\le T}\bigl(\|\Psi_t\|_{C^{\alpha}}+\|\boldsymbol{\Psi}_t\|_{2\alpha}\bigr),
\end{equation}
which is almost surely finite by Lemma~\ref{lem:rp_levantamiento}(iii). All the
constants below are allowed to depend on \(T\), on \(K_T\), on \(\|v\|_{1,T}\), on
\(f\) and \(g\) through the bounds of Assumption~\ref{ass:standing}(A2), and on the
exponents \(\alpha,\beta\); the dependence is indicated by subscripts when it is
relevant, and never involves the time \(t\) or the increment \(h\).

\subsection{Two preliminary estimates}\label{subsec:prelim_estimates}

Proposition~\ref{prop:hk_result} bounds the rough integral of \(f(\lambda\,\cdot)\,Y\)
against \(\Psi_s\) by \(\lambda^{-\alpha}\|f\|_{1,1}\|Y\|_{\Psi_s,\alpha}\). To apply
it to the heat kernel we need two things: that \(p_t\) and its derivatives have the
rescaled form \(f(\lambda x)\) with \(\lambda=t^{-1/2}\) — this is what produces the
gain \(t^{\alpha/2}\) from the concentration of the kernel — and that the functions
\(f\) so obtained have \(\|\cdot\|_{1,1}\)-norms bounded \emph{uniformly in \(t\)},
since otherwise the constant in Proposition~\ref{prop:hk_result} would depend on
\(t\) and could compensate the gain. The following lemma provides both.

\begin{lemma}\label{lem:heat_prop}
	Let \(T>0\). For every \(t\in(0,T]\) there exist functions
	\(f_t,g_t,h_t\in C^{\infty}_{c}(\mathbb{R})\) with uniformly bounded
	\(\|\cdot\|_{1,1}\)-norms,
	\[
	\sup_{t\in(0,T]}\|f_t\|_{1,1}<\infty,\qquad
	\sup_{t\in(0,T]}\|g_t\|_{1,1}<\infty,\qquad
	\sup_{t\in(0,T]}\|h_t\|_{1,1}<\infty,
	\]
	such that for every \(x\in[0,2\pi]\) the scaling identities
	\[
	p_t(x)=\frac{1}{\sqrt t}\,f_t\!\Bigl(\frac{x}{\sqrt t}\Bigr),\qquad
	\partial_x p_t(x)=\frac{1}{t}\,g_t\!\Bigl(\frac{x}{\sqrt t}\Bigr),\qquad
	\partial_x^{2}p_t(x)=\frac{1}{t\sqrt t}\,h_t\!\Bigl(\frac{x}{\sqrt t}\Bigr)
	\]
	hold.
\end{lemma}

\begin{proof}
	Deferred to Appendix~\ref{appendix:A}.
\end{proof}

The second ingredient is a bound on the size of the integrand \(g(u_s)\) of
\(\mathcal{M}^{(2)}_{\Psi}v\) as a controlled path. For fixed \(s\), the norm
\(\|\cdot\|_{\Psi_s,\alpha}\) is the controlled-path norm of
Definition~\ref{def:controlled} taken with respect to the rough path
\(X=\Psi_s\), i.e.\ \(\psi_s\) regarded as a path in the spatial variable together
with its canonical lift \(\boldsymbol{\Psi}_s\). The path \(u_s=\psi_s+U_s+v_s\) is
controlled by \(\Psi_s\) with derivative process \(I\): the summand \(\psi_s\) by
construction (i) of Section~\ref{subsec:rp}, and \(U_s+v_s\in C^{1}\subset C^{2\alpha}\)
by the embedding \(\iota\) of construction (ii). By construction (iii), \(g(u_s)\)
is then controlled by \(\Psi_s\) with derivative process \(Dg(u_s)\), and, following
the convention of Section~\ref{subsec:rp}, we write \(\|g(u_s)\|_{\Psi_s,\alpha}\)
for the norm of the pair \(\bigl(g(u_s),Dg(u_s)\bigr)\). The next lemma bounds it in
terms of \(K_T\), \(\|v\|_{1,T}\) and \(s\); the dependence on \(s\) comes entirely
from \(U_s\), whose \(C^{2\alpha}\)-norm degenerates as \(s\to0^{+}\).

\begin{lemma}\label{lem:Lambda}
	For \(s\in(0,T]\) set \(\Lambda(s):=\|g(u_s)\|_{\Psi_s,\alpha}\). Then there is an
	almost surely finite constant \(C\) such that
	\begin{equation}\label{eq:Lambda_bound}
		\Lambda(s)\;\le\;C\bigl(1+s^{\frac{\beta-2\alpha}{2}}\bigr),\qquad s\in(0,T],
	\end{equation}
	and consequently
	\begin{equation}\label{eq:Lambda_props}
		\int_0^T\Lambda(s)\,ds<\infty
		\qquad\text{and}\qquad
		\sup_{s\in[a,T]}\Lambda(s)<\infty\ \text{ for every } a>0 .
	\end{equation}
\end{lemma}

\begin{proof}
	Fix \(s\in(0,T]\) and define
	\[
	\varphi_s(y,x):=g\bigl(v_s(x)+U_s(x)+y\bigr),\qquad y\in\mathbb{R}^{n},\ x\in[0,2\pi].
	\]
	Then \(\varphi_s(\psi_s(x),x)=g(u_s(x))\) for every \(x\), that is,
	\(g(u_s)=\varphi_s(\psi_s)\) in the sense of construction (iii) of
	Section~\ref{subsec:rp}, applied to the controlled path \((\psi_s,I)\). The point of
	this rewriting is that all the dependence on \(s\) has been moved into the
	\emph{second} argument of \(\varphi_s\), where Lemma~\ref{lem:rp_composition} is
	linear, whereas the path \((\psi_s,I)\) that enters quadratically carries no
	\(s\)-dependence beyond \(K_T\).
	
	We verify the hypotheses of Lemma~\ref{lem:rp_composition}. Uniform \(C^{2}\)
	regularity in the first variable follows from the boundedness of \(g\), \(Dg\) and
	\(D^{2}g\) in Assumption~\ref{ass:standing}(A2). For the Hölder regularity in the
	second variable, fix \(y\in\mathbb{R}^{n}\); since \(Dg\) is bounded and
	\(v_s+U_s\in C^{2\alpha}\),
	\[
	\|g(v_s+U_s+y)\|_{C^{2\alpha}}
	\le\|g\|_{C^{1}}\bigl(1+\|v_s+U_s\|_{C^{2\alpha}}\bigr),
	\]
	and, using \(\|v_s\|_{C^{2\alpha}}\le\|v_s\|_{C^{1}}\le\|v\|_{1,T}\) together with
	the Schauder estimate~\eqref{eq:heat_holder_smoothing} applied to
	\(U_s=S_s(u_0-\psi_0)\) with \(u_0-\psi_0\in C^{\beta}\),
	\[
	\|v_s+U_s\|_{C^{2\alpha}}
	\le\|v\|_{1,T}+C\,s^{\frac{\beta-2\alpha}{2}} .
	\]
	Hence, writing \(\|\varphi_s\|_{2\alpha;t}:=\sup_{y}\|\varphi_s(y,\cdot)\|_{C^{2\alpha}}\)
	as in Lemma~\ref{lem:rp_composition},
	\[
	\|\varphi_s\|_{2\alpha;t}\le C_{v}\bigl(1+s^{\frac{\beta-2\alpha}{2}}\bigr).
	\]
	On the other hand the controlled path \((\psi_s,I)\in C^{\alpha}_{\Psi_s}\) has
	vanishing remainder, so that \(\|(\psi_s,I)\|_{\Psi_s,\alpha}\le C(1+K_T)\).
	By Lemma~\ref{lem:rp_composition},
	\[
	\|\varphi_s(\psi_s)\|_{\Psi_s,\alpha}
	\le C\bigl(\|D_y^{2}\varphi_s\|_{\infty}+\|\varphi_s\|_{\infty}+\|\varphi_s\|_{2\alpha;t}\bigr)
	\bigl(1+\|(\psi_s,I)\|_{\Psi_s,\alpha}\bigr)^{2}
	\le C_{v}\bigl(1+s^{\frac{\beta-2\alpha}{2}}\bigr)(1+K_T)^{2},
	\]
	and absorbing \((1+K_T)^{2}\) into the constant gives~\eqref{eq:Lambda_bound}.
	For~\eqref{eq:Lambda_props}, 
   
   note that \(\beta-2\alpha>-1/2\) since \(1/3<\alpha <\beta<1/2\)  by Assumption~\ref{ass:standing}
   , so
	\(s\mapsto s^{(\beta-2\alpha)/2}\) is integrable at the origin, the exponent
	being greater than   \(-1/4\) ; on \([a,T]\) it is bounded, the exponent being
	negative.
\end{proof}

\begin{remark}\label{rem:unbounded_derivates_g}
	If \(g\) and its derivatives are not globally bounded, one can localise by replacing
	\(\varphi_s\) with \(\varphi_s(y,x):=g\bigl(v_s(x)+U_s(x)+y\bigr)\chi_s(y)\), where
	\(\chi_s\) is a smooth cut-off with \(\chi_s\equiv1\) on a neighbourhood of the range
	of \(\psi_s\). Since \(\chi_s(\psi_s(x))=1\), the identity
	\(\varphi_s(\psi_s(x),x)=g(u_s(x))\) is preserved, while all \(y\)-derivatives of
	\(\varphi_s\) become bounded. The conclusion of Lemma~\ref{lem:Lambda} therefore
	holds without global boundedness of \(g\), the constant now depending on the range
	of \(\psi_s\) as well.
\end{remark}

\subsection{Temporal regularity of \texorpdfstring{$v$}{v}}\label{subsec:temporal_reg}

The exponent to beat is \(1/4\): the temporal quartic variation of
Section~\ref{sec:variaciones} is calibrated to a temporal Hölder exponent of
\(1/4\) for \(\psi\), and any perturbation of strictly higher regularity leaves it
unchanged. We obtain much more.

\begin{proposition}[Temporal Hölder regularity of \(v\)]\label{prop:regularity_v}
	Let \([T_1,T_2]\subset(0,T]\) with \(T_1>0\). Then, almost surely, for every fixed
	\(x\in\mathbb{T}\) the trajectory \(t\mapsto v(t,x)\) belongs to
	\(C^{\rho}([T_1,T_2])\) for every \(\rho<\tfrac{1+\alpha}{2}\). In particular, since
	\(\alpha>\tfrac13\), one may take \(\rho>\tfrac12\).
\end{proposition}

\begin{remark}\label{rem:global_vs_local}
	The restriction to a compact subinterval of \((0,T]\) is the only regime used in
	Section~\ref{sec:estimadores}, and it is what allows the weight \(\Lambda\) of
	Lemma~\ref{lem:Lambda} to be dispensed with in Step~3 below. A statement valid on all
	of \([0,T]\) can be obtained by carrying \(\Lambda\) through the change of
	variables, at the cost of the smaller exponent \((1+\beta-\alpha)/2\), which still
	exceeds \(1/4\) since \(\beta>\alpha\); see Appendix~\ref{appendix:B}.
\end{remark}

\begin{remark}\label{rem:unbounded_all}
	By Remark~\ref{rem:unbounded_derivates_g}, Proposition~\ref{prop:regularity_v} also
	holds when \(g\) and the derivatives of \(f\) and \(g\) are not globally bounded,
	on any interval on which the solution exists, provided the constants are allowed to
	depend on \(f\), \(g\) and their derivatives on the range of \(u\).
\end{remark}

\begin{proof}[Proof of Proposition~\ref{prop:regularity_v}]
	Fix \(\gamma\in[0,\alpha)\); we shall prove Hölder continuity of order
	\(\rho=(1+\gamma)/2\), which gives the claim since \(\gamma\) is arbitrary. We
	treat the two summands of \(v=\mathcal{M}^{(1)}_{\Psi}v+\mathcal{M}^{(2)}_{\Psi}v\)
	separately, and use throughout that \(u\in C([0,T];C^{\beta})\) by
	Theorem~\ref{theorem:Existence_Uq}, so that \(\sup_{s\le T}\|u_s\|_{\infty}<\infty\)
	and, \(f\) being \(C^{\infty}\), there is a constant \(C_f\) with
	\(\|f(u_s)\|_{\infty}\le C_f\) for all \(s\in[0,T]\).
	
	\medskip
	\step{1} \textbf{Spatial regularity of \(\mathcal{M}^{(1)}_{\Psi}v\).}
	We show that \(\bigl(\mathcal{M}^{(1)}_{\Psi}v\bigr)_t\in C^{1+\gamma}\), uniformly
	in \(t\in(0,T]\). By~\eqref{eq:heat_holder_smoothing} with \(\eta=0\),
	\(\kappa=1+\gamma\), and the bound \(\|U_s\|_{C^{1}}\le Cs^{\frac{\beta-1}{2}}\)
	(again from~\eqref{eq:heat_holder_smoothing}),
	\begin{align*}
		\|\mathcal{M}^{(1)}_{\Psi}v_t\|_{C^{1+\gamma}}
		&\le C\int_0^t(t-s)^{-\frac{1+\gamma}{2}}
		\bigl\|g(u_s)(\partial_x v_s+\partial_x U_s)+f(u_s)+\psi_s\bigr\|_{\infty}\,ds\\
		&\le C\int_0^t(t-s)^{-\frac{1+\gamma}{2}}
		\Bigl(\|g\|_{\infty}\bigl(\|v_s\|_{C^{1}}+\|U_s\|_{C^{1}}\bigr)+C_f+K_T\Bigr)ds\\
		&\le C_{f,K_T,v}\Bigl(\int_0^t(t-s)^{-\frac{1+\gamma}{2}}ds
		+\int_0^t(t-s)^{-\frac{1+\gamma}{2}}s^{\frac{\beta-1}{2}}ds\Bigr),
	\end{align*}
	where the subscript \(v\) records dependence on \(\|v\|_{1,T}\). The first integral
	equals \(\frac{2}{1-\gamma}t^{\frac{1-\gamma}{2}}\). For the second, the substitution
	\(s=tz\) gives
	\[
	\int_0^t(t-s)^{-\frac{1+\gamma}{2}}s^{\frac{\beta-1}{2}}ds
	=t^{\frac{\beta-\gamma}{2}}\,\Beta\!\Bigl(\tfrac{1-\gamma}{2},\tfrac{\beta+1}{2}\Bigr),
	\]
	finite because both arguments of the Euler Beta function are positive. Since
	\(\gamma<\alpha<\beta\), both exponents of \(t\) are positive and we conclude
	\begin{equation}\label{eq:M1_spatial}
		\sup_{t\in(0,T]}\|\mathcal{M}^{(1)}_{\Psi}v_t\|_{C^{1+\gamma}}\le C_{f,K_T,T,v}.
	\end{equation}
	
	\medskip
	\step{2} \textbf{Temporal regularity of \(\mathcal{M}^{(1)}_{\Psi}v\).}
	Fix \(x\in\mathbb{T}\), \(t\in[T_1,T_2]\) and \(0<h<1\) with \(t+h\le T\). By the
	semigroup property,
	Writing \(F_s:=g(u_s)(\partial_x v_s+\partial_x U_s)+f(u_s)+\psi_s\) for the
	integrand of \(\mathcal{M}^{(1)}_{\Psi}v\), we have
	\begin{align*}
		\bigl|(\mathcal{M}^{(1)}_{\Psi}v)(t+h,x)-(\mathcal{M}^{(1)}_{\Psi}v)(t,x)\bigr|
		&\le\bigl|(S_h-I)\bigl(\mathcal{M}^{(1)}_{\Psi}v_t\bigr)(x)\bigr|
		+\int_t^{t+h}\bigl\|S_{t+h-s}F_s\bigr\|_{\infty}\,ds .
	\end{align*}
	For the first term, the spatial regularity~\eqref{eq:M1_spatial} allows us to
	apply~\eqref{eq:heat_approx} with \(\kappa=1+\gamma\):
	\[
	\bigl|(S_h-I)\bigl(\mathcal{M}^{(1)}_{\Psi}v_t\bigr)(x)\bigr|
	\le Ch^{\frac{1+\gamma}{2}}\|\mathcal{M}^{(1)}_{\Psi}v_t\|_{C^{1+\gamma}}
	\le C_{f,K_T,T,v}\,h^{\frac{1+\gamma}{2}} .
	\]
	For the second, the contraction~\eqref{eq:heat_contrac}, the bound
	\(\|U_s\|_{C^{1}}\le Cs^{\frac{\beta-1}{2}}\) and the concavity of
	\(r\mapsto r^{\frac{\beta+1}{2}}\) give
	\[
	\begin{aligned}
	\int_t^{t+h}\|S_{t+h-s}F_s\|_{\infty}\,ds
	&\le C_{f,K_T,v}\int_t^{t+h}\bigl(1+s^{\frac{\beta-1}{2}}\bigr)ds
	=C_{f,K_T,v}\Bigl(h+(t+h)^{\frac{\beta+1}{2}}-t^{\frac{\beta+1}{2}}\Bigr)\\
	&\le C_{f,K_T,v}\,h^{\frac{\beta+1}{2}} .
	\end{aligned}
	\]
	Since \(\frac{1+\gamma}{2}<\frac{1+\beta}{2}\) and \(h<1\), the map
	\(t\mapsto(\mathcal{M}^{(1)}_{\Psi}v)(t,x)\) is Hölder continuous of order
	\(\frac{1+\gamma}{2}\).
	
	\medskip
	\step{3} \textbf{Temporal regularity of \(\mathcal{M}^{(2)}_{\Psi}v\).}
	Combining the fundamental theorem of calculus with the fact that \(p\) solves the
	heat equation, and applying the rough Fubini theorem
	(see~\cite{FrizHairer2020}; for a similar use cf.~\cite{HairerWeber2013}), we
	obtain
	\begin{align}
		&\bigl|(\mathcal{M}^{(2)}_{\Psi}v)(t+h,x)-(\mathcal{M}^{(2)}_{\Psi}v)(t,x)\bigr|\notag\\
		&\qquad\le\int_0^t\!\int_0^h\Bigl|\int_0^{2\pi}\partial_x^{2}p_{t-s+\tau}(x-y)\,g(u_s(y))\,d\Psi_s(y)\Bigr|d\tau\,ds
		\label{eq:int_past}\\
		&\qquad\quad+\int_t^{t+h}\Bigl|\int_0^{2\pi}p_{t+h-s}(x-y)\,g(u_s(y))\,d\Psi_s(y)\Bigr|ds
		=:\mathcal{I}_{\mathrm{past}}+\mathcal{I}_{\mathrm{tail}} .
		\label{eq:int_tail}
	\end{align}
	
	Proposition~\ref{prop:hk_result} applies only for \(\lambda\ge1\); with
	\(\lambda=(t-s+\tau)^{-1/2}\) this means \(t-s+\tau\le1\). We accordingly split
	\(\mathcal{I}_{\mathrm{past}}=\mathcal{I}_{\mathrm{sing}}+\mathcal{I}_{\mathrm{reg}}\)
	according to whether \(t-s+\tau\le1\) or \(>1\); on the second range the kernel is
	smooth and Theorem~\ref{thm:rough_integral} is used directly.
	
	\emph{The singular range.} By Lemma~\ref{lem:heat_prop},
	\(\partial_x^{2}p_r(x)=r^{-3/2}h_r(x/\sqrt r)\) with \(\sup_{r\le1}\|h_r\|_{1,1}<\infty\),
	so Proposition~\ref{prop:hk_result} with \(\lambda=r^{-1/2}\) and the controlled path
	\(g(u_s)\) gives, with \(\Lambda\) as in Lemma~\ref{lem:Lambda},
	\[
	\mathcal{I}_{\mathrm{sing}}
	\le C_{K_T}\int_0^t\!\int_0^h(t-s+\tau)^{\frac{\alpha}{2}-\frac32}\,\Lambda(s)\,d\tau\,ds .
	\]
	We split the outer integral at \(s=T_1/2\). For \(s\le T_1/2\) we have
	\(t-s+\tau\ge T_1/2\), so the kernel is bounded by \((T_1/2)^{\frac{\alpha}{2}-\frac32}\)
	and, by the integrability of \(\Lambda\) in~\eqref{eq:Lambda_props},
	\[
	\int_0^{T_1/2}\!\!\int_0^h(t-s+\tau)^{\frac{\alpha}{2}-\frac32}\Lambda(s)\,d\tau\,ds
	\le C_{T_1}\,h\int_0^{T_1/2}\Lambda(s)\,ds\le C_{T_1}\,h .
	\]
	For \(s\ge T_1/2\) we have \(\Lambda(s)\le\sup_{[T_1/2,T]}\Lambda<\infty\)
	by~\eqref{eq:Lambda_props}, whence
	\[
	\int_{T_1/2}^{t}\!\!\int_0^h(t-s+\tau)^{\frac{\alpha}{2}-\frac32}\Lambda(s)\,d\tau\,ds
	\le C_{T_1}\int_0^t\!\int_0^h(t-s+\tau)^{\frac{\alpha}{2}-\frac32}\,d\tau\,ds .
	\]
	The substitution \(s=hx\), \(\tau=hy\), \(t=hz\) contributes \(h^{2}\) from
	\(ds\,d\tau\) and \(h^{\frac{\alpha}{2}-\frac32}\) from the kernel, giving
	\[
	\int_0^t\!\int_0^h(t-s+\tau)^{\frac{\alpha}{2}-\frac32}\,d\tau\,ds
	=h^{\frac{1+\alpha}{2}}\,\IA(z),\qquad
	\IA(z)=\int_0^z\!\int_0^1(z-x+y)^{\frac{\alpha}{2}-\frac32}\,dy\,dx ,
	\]
	and \(\IA\) is bounded uniformly in \(z>0\) by Appendix~\ref{appendix:B}. Since
	\(\frac{1+\alpha}{2}<1\) and \(h<1\),
	\[
	\mathcal{I}_{\mathrm{sing}}\le C_{K_T,T_1}\bigl(h+h^{\frac{1+\alpha}{2}}\bigr)
	\le C_{K_T,T_1}\,h^{\frac{1+\alpha}{2}} .
	\]
	
	\emph{The regular range.} For \(r=t-s+\tau>1\), the bound~\eqref{eq:rough_int_bound}
	of Theorem~\ref{thm:rough_integral} together with the algebra
	property~\eqref{eq:algebra} give
	\[
	\Bigl|\int_0^{2\pi}\partial_x^{2}p_r(x-y)\,g(u_s(y))\,d\Psi_s(y)\Bigr|
	\le C_{K_T}\,\|\partial_x^{2}p_r\|_{C^{1}([0,2\pi])}\,\Lambda(s)
	\le C_{K_T,T}\,\Lambda(s),
	\]
	the last step by Lemma~\ref{lem:heat_prop}, since \(\|\partial_x^{2}p_r\|_{C^{1}}\le C\|h_r\|_{1,1}\).
	Hence, by~\eqref{eq:Lambda_props},
	\[
	\mathcal{I}_{\mathrm{reg}}\le C_{K_T,T}\int_0^t\!\int_0^h\Lambda(s)\,d\tau\,ds
	=C_{K_T,T}\,h\int_0^t\Lambda(s)\,ds\le C_{v,K_T,T}\,h .
	\]
	
	\emph{The tail.} In \(\mathcal{I}_{\mathrm{tail}}\) we have \(0<t+h-s<h<1\), so
	Lemma~\ref{lem:heat_prop} and Proposition~\ref{prop:hk_result} apply with
	\(p_r(x)=r^{-1/2}f_r(x/\sqrt r)\), giving
	\[
	\mathcal{I}_{\mathrm{tail}}\le C_{K_T}\int_t^{t+h}(t+h-s)^{\frac{\alpha}{2}-\frac12}\Lambda(s)\,ds .
	\]
	Here \(s\ge t\ge T_1\), so \(\Lambda(s)\le C_{T_1}\) by~\eqref{eq:Lambda_props}, and
	the substitution \(r=t+h-s\) yields
	\[
	\mathcal{I}_{\mathrm{tail}}\le C_{K_T,T_1}\int_0^h r^{\frac{\alpha-1}{2}}dr
	=\frac{2C_{K_T,T_1}}{\alpha+1}\,h^{\frac{1+\alpha}{2}},
	\]
	the integral converging because \(\frac{\alpha-1}{2}>-1\).
	
	\medskip
	\noindent\textbf{Conclusion.} Collecting the three bounds and using \(h<1\),
	\[
	\bigl|(\mathcal{M}^{(2)}_{\Psi}v)(t+h,x)-(\mathcal{M}^{(2)}_{\Psi}v)(t,x)\bigr|
	\le C_{v,K_T,T,T_1}\,h^{\frac{1+\alpha}{2}} ,
	\]
	and together with Step~2,
	\[
	|v(t+h,x)-v(t,x)|\le C_{v,f,K_T,T,T_1}\,h^{\rho},
	\qquad
	\rho=\min\Bigl\{\tfrac{1+\gamma}{2},\tfrac{1+\alpha}{2}\Bigr\}=\tfrac{1+\gamma}{2},
	\]
	which is the claim.
\end{proof}

\begin{remark}\label{remark:sobre_lambda}
	For the rough component \(\mathcal{M}^{(2)}_{\Psi}v_t\), \(t>0\) fixed, finiteness
	of the \(C^{1+\gamma}\) spatial norm also requires \(\gamma<\alpha\), but the bound
	is not uniform in \(t\in(0,T]\): it carries a weight of order \(t^{\gamma/2}\)
	which blows up as \(t\to0^{+}\); see \cite[Remark~3.6]{HairerMatetski2016}
	and~\cite[Corollary~3.16]{Hairer2013}. This is why the improved exponent
	\(1+\gamma\) is used only for \(\mathcal{M}^{(1)}_{\Psi}v\), where it is uniform
	(Step~1), while for \(\mathcal{M}^{(2)}_{\Psi}v\) we rely on \(C^{1}\) regularity
	in space, which is all that is needed below.
\end{remark}

\subsection{The initial condition component}\label{subsec:initial_data}

The term \(U_t=S_t(u_0-\psi_0)\) needs no separate argument away from the origin.
For any fixed \(t>0\) the map \(x\mapsto U_t(x)\) is \(C^{\infty}\) on the torus,
and for any fixed \(x\) the map \(t\mapsto U_t(x)\) is \(C^{\infty}\) on every compact
subinterval \([T_1,T_2]\subset(0,T]\), with
\[
\|\partial_t U_t\|_{\infty}=\|\partial_x^{2}S_t(u_0-\psi_0)\|_{\infty}
\le C\,t^{\frac{\beta-2}{2}}\|u_0-\psi_0\|_{C^{\beta}}\le C_{T_1}
\qquad\text{for } t\ge T_1,
\]
by~\eqref{eq:heat_holder_smoothing}. In particular \(U\) is Lipschitz in time on
\([T_1,T_2]\), which is better than the exponent of
Proposition~\ref{prop:regularity_v}.

This smoothness degenerates as \(t\to0^{+}\). By~\eqref{eq:heat_holder_smoothing},
for \(\kappa>\beta\),
\[
\|U_t\|_{C^{\kappa}}\le C\,t^{\frac{\beta-\kappa}{2}}\|u_0-\psi_0\|_{C^{\beta}},
\]
so the best spatial regularity that holds uniformly on \([0,T]\) is \(\beta\), the
regularity of the datum itself, and the corresponding uniform temporal exponent is
\(\beta/2<1/4\).

\begin{remark}\label{rem:initial}
	Were the variations of Section~\ref{sec:estimadores} computed over an interval
	containing \(t=0\), the term \(U\) would be no smoother there than \(\psi\), and the
	initial singularity would contaminate the limits. We avoid this by working on
	intervals \([T_1,T_2]\) with \(T_1>0\), which is all that is needed; the stronger
	hypothesis \(u_0-\psi_0\in C^{1/2+\varepsilon}\) would be required only to include
	the origin, and is not assumed here.
\end{remark}

We collect the conclusions in the form in which Section~\ref{sec:estimadores} will
use them; nothing about the rough-path construction is needed beyond this point.

\begin{corollary}\label{cor:R_properties}
	Almost surely, the remainder \(R=U+v\) satisfies:
	\begin{enumerate}
		\item[(R1)] for every \(t\in(0,T]\), the map \(x\mapsto R(t,x)\) is Lipschitz on
		\(\mathbb{T}\);
		\item[(R2)] for every \([T_1,T_2]\subset(0,T]\) with \(T_1>0\) and every
		\(x\in\mathbb{T}\), the map \(t\mapsto R(t,x)\) is \(\rho\)-Hölder on \([T_1,T_2]\)
		for some \(\rho>\tfrac12\).
	\end{enumerate}
\end{corollary}

\begin{proof}
	(R1): \(v_t\in C^{1}\) by Definition~\ref{def:mild_principal}, and
	\(U_t\in C^{\infty}\) for \(t>0\). (R2): Proposition~\ref{prop:regularity_v} for
	\(v\), and the Lipschitz bound for \(U\) of Section~\ref{subsec:initial_data}.
\end{proof}

The thresholds that matter are \(1\) for (R1) and \(1/4\) for (R2): the spatial
quadratic variation is insensitive to any spatially Lipschitz perturbation, and the
temporal quartic variation to any perturbation that is \(\rho\)-Hölder in time with
\(\rho>1/4\). Corollary~\ref{cor:R_properties} clears both with room to spare.

\section{Variations of the stationary process \texorpdfstring{$\psi$}{psi}}\label{sec:variaciones}

In this section we determine the exact asymptotic behaviour of the quadratic and
quartic variations of the stationary solution \(\psi\). Our approach for the
\(L^{2}\) convergence follows ideas analogous to those of~\cite{PospisilTribe2007},
while the almost sure convergence is obtained using arguments inspired by
Section~4 of~\cite{AssaadTudor2021}.

Convergence is first established in \(L^{2}(\Omega)\) by a direct analysis of the
covariance structure; subsequently, the hypercontractivity property of the Wiener
chaos together with the moment bounds obtained in the \(L^{2}\) step lift these
limits to almost sure convergence. The almost sure results are essential for our
purposes: in the rough path framework the solution is defined trajectory by
trajectory, and the estimators constructed in Section~\ref{sec:estimadores} rely on
pathwise limits that hold for almost every realisation of the driving process
\(\Psi\).

The variation of order \(2\) in space and of order \(4\) in time is not an arbitrary
choice: \(\psi\) is \(\eta\)-Hölder in space for every \(\eta<\tfrac12\) and
\(\kappa\)-Hölder in time for every \(\kappa<\tfrac14\), and the order of the
variation is calibrated to the regularity, \(2\cdot\tfrac12=4\cdot\tfrac14=1\).

\subsection{The stochastic heat equation with damping}

Consider the stationary solution \(\psi\) of~\eqref{eq:estacionario_H}, given by
\[
\psi(t,x)=\sigma\int_{-\infty}^{t}\!\int_{\mathbb{T}}p^{E}_{t-s}(x-y)\,W(ds,dy),
\qquad (t,x)\in\mathbb{R}\times\mathbb{T}.
\]
In what follows we set \(\sigma=1\); the general case is recovered by a simple
scaling argument (see Remark~\ref{rem:scaling_variations}).

\begin{convention}\label{conv:component}
	Throughout Sections~\ref{sec:variaciones}--\ref{sec:estimadores} we work with a
	single, arbitrarily fixed component of the solution. By
	Lemma~\ref{lem:psi_decoupling} the components \(\psi^{1},\dots,\psi^{n}\) are
	independent copies of one scalar centred Gaussian field; we fix an index \(i\),
	suppress it from the notation, and write \(\psi\), \(u_\theta\) and \(R\) for
	\(\psi^{i}\), \(u^{i}_{\theta}\) and \(R^{i}\) respectively. All quantities appearing
	in the variations below are therefore real-valued, and \((\cdot)^{2}\), \((\cdot)^{4}\)
	denote powers of real numbers, not Euclidean norms on \(\mathbb{R}^{n}\).
	Accordingly, the estimators of Section~\ref{sec:estimadores} are built from a
	single observed component of the solution. To avoid any confusion with the
	dimension \(n\), the number of intervals of a partition is denoted by \(N\).
\end{convention}

\begin{remark}\label{not:increments}
	We follow Remark~\ref{not:psi_t}: increments along a partition are written in the
	two-argument form, \(\psi(t,x_{j})-\psi(t,x_{j-1})\) for a spatial partition at a
	fixed time \(t\), and \(\psi(t_{j},x)-\psi(t_{j-1},x)\) for a temporal partition at
	a fixed point \(x\). For a mesh \(\delta>0\) we write
	\[
	\Delta(x,\delta):=\psi(t,x+\delta)-\psi(t,x),
	\qquad
	\Gamma(s,\delta):=\psi(s+\delta,x)-\psi(s,x),
	\]
	for the spatial increment at fixed \(t\) and the temporal increment at fixed \(x\);
	the variable held fixed is suppressed from the notation.
\end{remark}

\begin{remark}[Why one component suffices, and what it costs]\label{rem:one_component}
	The coupling between the components of \(u\) enters only through the nonlinearity,
	hence only through the remainder \(R=U+v\); the Gaussian part decouples by
	Lemma~\ref{lem:psi_decoupling}. Since \(|R^{i}(t,x)-R^{i}(t,y)|\le|R(t,x)-R(t,y)|\)
	for every \(i\), each component of the remainder inherits the regularity
	established in Section~\ref{sec:regularidad}, so the arguments of
	Section~\ref{sec:estimadores} apply verbatim to \(u^{i}=\psi^{i}+R^{i}\).
	
	Two remarks are in order. First, it is the independence of the \(W^{i}\) that is
	doing the work here: for correlated noise components, or for a non-diagonal
	diffusivity, the reduction would fail. Second, using a single component discards
	information. Working with the Euclidean norm \(|\cdot|\) on \(\mathbb{R}^{n}\)
	instead, one finds
	\[
	\sum_{j=1}^{N}\bigl|\Delta(x_{j-1},\delta)\bigr|^{2}
	\longrightarrow\frac{n\,\sigma^{2}}{2\theta}(A_2-A_1),
	\qquad
	\sum_{j=1}^{N}\bigl|\Gamma(t_{j-1},\delta)\bigr|^{4}
	\longrightarrow\frac{n(n+2)\,\sigma^{4}}{\pi\theta}(T_2-T_1),
	\]
	the second constant coming from \(\mathbb{E}|Z|^{4}=n(n+2)\tau^{4}\) for
	\(Z\sim\mathcal N(0,\tau^{2}I_{n})\); the first limit follows at once by
	independence, the second requires expanding \(|\Gamma|^{4}=(\sum_i\Gamma_i^{2})^{2}\)
	and repeating the moment estimates of this section with the cross terms
	\(\Gamma_i^{2}\Gamma_j^{2}\). This reduces the variance of the estimators by a
	factor of order \(n\), which is immaterial for the strong consistency proved here
	but would matter for a central limit theorem; see Remark~\ref{rem:no_clt}.
\end{remark}

The covariance structure is derived from the semigroup property and the
isometry~\eqref{eq:wn_cov}. Changing variable \(v=t+s-2r\) (\(dr=-dv/2\)) we
obtain
\[
\begin{aligned}
	\mathbb{E}[\psi(t,x)\psi(s,y)]
	&=\int_{-\infty}^{t\wedge s}\!\int_{\mathbb{T}}p^{E}_{t-r}(x-u)\,p^{E}_{s-r}(y-u)\,du\,dr\\
	&=\int_{-\infty}^{t\wedge s}p^{E}_{t+s-2r}(x-y)\,dr
	=\frac12\int_{|t-s|}^{\infty}e^{-r}p_r(x-y)\,dr\\
	&=\frac{1}{4\pi}\sum_{k\in\mathbb{Z}}\frac{e^{-(k^{2}+1)|t-s|}e^{ik(x-y)}}{k^{2}+1}
	=\frac{1}{4\pi}\sum_{k\in\mathbb{Z}}\frac{e^{-(k^{2}+1)|t-s|}\cos(k(x-y))}{k^{2}+1},
\end{aligned}
\]
which is~\eqref{eq:psi_cov} with \(\sigma=1\). When \(t=s\) this reduces to
\[
\mathbb{E}[\psi(t,x)\psi(t,y)]
=\frac{1}{4\pi}\sum_{k\in\mathbb{Z}}\frac{\cos(k(x-y))}{k^{2}+1}
=\frac14\,\frac{\cosh(|x-y|-\pi)}{\sinh\pi},
\qquad |x-y|\le\pi,
\]
and when \(x=y\) to
\[
\mathbb{E}[\psi(t,x)\psi(s,x)]
=\frac{1}{4\pi}\sum_{k\in\mathbb{Z}}\frac{e^{-(k^{2}+1)|t-s|}}{k^{2}+1}
=\frac{1}{4\pi}\,\Upsilon(|t-s|),
\]
where the auxiliary function \(\Upsilon\) is introduced in
Appendix~\ref{appendix:C}.

\subsection{Spatial quadratic variation}

Fix \(t>0\) and an interval \([A_1,A_2]\subset\mathbb{T}\), so that
\(0<A_2-A_1\le2\pi\). We study the quadratic variation of the spatial trajectory
\(x\mapsto\psi(t,x)\) over \([A_1,A_2]\). From the covariance computed above, the
spatial covariance at equal times is
\[
C(x,y):=\mathbb{E}[\psi(t,x)\psi(t,y)]
=\frac14\,\frac{\cosh(|x-y|-\pi)}{\sinh\pi},
\qquad |x-y|\le\pi,
\]
extended by \(2\pi\)-periodicity; in particular \(C(x,x)=\tfrac14\coth\pi\). Note
that \(z\mapsto\cosh(|z|-\pi)\) is smooth except at \(z=0\), where its derivative
jumps.

Using the explicit formula for \(C\) and expanding in powers of \(\delta\), we obtain
\begin{equation}\label{eq:var_inc}
	\mathbb{E}[\Delta^{2}(x,\delta)]=2\bigl(C(x,x)-C(x,x+\delta)\bigr)
	=\tfrac12\,\delta+O(\delta^{2}).
\end{equation}
Let now \(x<y\) with \(x+\delta\le y\), and set \(\kappa=y-x\ge\delta\). Then
\[
\begin{aligned}
	\mathbb{E}[\Delta(x,\delta)\Delta(y,\delta)]
	&=2C(\kappa)-C(\kappa-\delta)-C(\kappa+\delta)\\
	&=\frac{1}{2}\frac{\cosh(\kappa-\pi)}{\sinh\pi}
	-\frac14\,\frac{\cosh(\kappa-\pi-\delta)+\cosh(\kappa-\pi+\delta)}{\sinh\pi}
	=O(\delta^{2}),
\end{aligned}
\]
the last estimate following from the Taylor expansion
\(\cosh(z-\delta)+\cosh(z+\delta)=2\cosh z+\cosh z\,\delta^{2}+O(\delta^{4})\) with
\(z=\kappa-\pi\), which is legitimate because the arguments \(\kappa\pm\delta\)
stay in \([0,2\pi]\), away from the kink at \(0\) modulo \(2\pi\). By Wick's theorem
for Gaussian fields (see Section~2 in~\cite{PospisilTribe2007}) we obtain the moment
bounds
\begin{align}
	\mathbb{E}[\Delta^{4}(x,\delta)]&=O(\delta^{2}),\label{eq:fourth_inc}\\
	\mathbb{E}[\Delta^{2}(x,\delta)\Delta^{2}(y,\delta)]&=\tfrac14\delta^{2}+O(\delta^{3}).\label{eq:sq_cross}
\end{align}

\begin{proposition}[Spatial quadratic variation]\label{prop:L2_spatial_conv}
	Fix \(A_1<A_2\) in \(\mathbb{T}\) and \(t>0\). For \(j=0,1,\dots,N\) define the
	partition
	\[
	x_j=A_1+j\delta,\qquad \delta=\frac{A_2-A_1}{N}.
	\]
	Then, in \(L^{2}(\Omega)\),
	\[
	\lim_{N\to\infty}\sum_{j=1}^{N}\bigl(\psi(t,x_j)-\psi(t,x_{j-1})\bigr)^{2}
	=\tfrac12\,(A_2-A_1).
	\]
\end{proposition}

\begin{proof}
	Let \(c=1/2\). Convergence in \(L^{2}(\Omega)\) is equivalent to
	\[
	\lim_{N\to\infty}\sum_{j=1}^{N}\sum_{k=1}^{N}
	\mathbb{E}\Bigl[\bigl(\Delta^{2}(x_{j-1},\delta)-c\delta\bigr)
	\bigl(\Delta^{2}(x_{k-1},\delta)-c\delta\bigr)\Bigr]=0 .
	\]
	
	\noindent\textbf{Diagonal terms (\(j=k\)).} We expand
	\[
	\mathbb{E}\bigl[(\Delta^{2}(x_{j-1},\delta)-c\delta)^{2}\bigr]
	=\mathbb{E}[\Delta^{4}(x_{j-1},\delta)]-2c\delta\,\mathbb{E}[\Delta^{2}(x_{j-1},\delta)]+c^{2}\delta^{2}.
	\]
	Using \(\mathbb{E}[\Delta^{4}(x_{j-1},\delta)]=O(\delta^{2})\) and
	\(\mathbb{E}[\Delta^{2}(x_{j-1},\delta)]=c\delta+O(\delta^{2})\), we obtain
	\[
	\sum_{j=1}^{N}\mathbb{E}\bigl[(\Delta^{2}(x_{j-1},\delta)-c\delta)^{2}\bigr]
	\le N\,O(\delta^{2})=O(\delta)\to0 .
	\]
	
	\noindent\textbf{Off-diagonal terms (\(j\neq k\)).} By symmetry it suffices to
	bound
	\[
	2\sum_{1\le j<k\le N}
	\mathbb{E}\Bigl[\bigl(\Delta^{2}(x_{j-1},\delta)-c\delta\bigr)
	\bigl(\Delta^{2}(x_{k-1},\delta)-c\delta\bigr)\Bigr].
	\]
	Since \(\mathbb{E}[\Delta^{2}(x_{j-1},\delta)\Delta^{2}(x_{k-1},\delta)]=c^{2}\delta^{2}+O(\delta^{3})\)
	by~\eqref{eq:sq_cross}, and the mixed terms cancel the \(c^{2}\delta^{2}\)
	contribution, each summand is \(O(\delta^{3})\). Hence the total off-diagonal
	contribution is bounded by \(N^{2}O(\delta^{3})=O(\delta)\to0\).
\end{proof}

\medskip
\noindent\textbf{Almost sure convergence.}
We now strengthen the \(L^{2}\) convergence of Proposition~\ref{prop:L2_spatial_conv}
and show that the limit holds almost surely. The argument relies on the
hypercontractivity property of the Wiener chaos, Markov's inequality and the
Borel--Cantelli lemma; an analogous argument can be found
in~\cite[Section~4, Proposition~6]{AssaadTudor2021}.

\begin{proposition}[Almost sure convergence of the spatial quadratic variation]\label{prop:spatial_qv_as}
	Under the notation and hypotheses of Proposition~\ref{prop:L2_spatial_conv},
	\[
	\lim_{N\to\infty}\sum_{j=1}^{N}\bigl(\psi(t,x_j)-\psi(t,x_{j-1})\bigr)^{2}
	=\tfrac12\,(A_2-A_1)
	\qquad\text{almost surely.}
	\]
\end{proposition}

\begin{proof}
	Let \(c=1/2\) and define the centred sum
	\[
	V_N=\sum_{j=1}^{N}\bigl(\Delta^{2}(x_{j-1},\delta)-\mathbb{E}[\Delta^{2}(x_{j-1},\delta)]\bigr).
	\]
	As shown in the proof of Proposition~\ref{prop:L2_spatial_conv}, the increments
	satisfy~\eqref{eq:var_inc} and~\eqref{eq:sq_cross}, which yield
	\(\mathbb{E}[V_N^{2}]=O(N^{-1})\): there is \(C_0\) with
	\(\mathbb{E}[V_N^{2}]\le C_0/N\). Each \(\Delta(x_{j-1},\delta)\) is a centred
	Gaussian variable, hence \(\Delta^{2}(x_{j-1},\delta)-\mathbb{E}[\Delta^{2}(x_{j-1},\delta)]\)
	belongs to the second Wiener chaos, and so does the linear combination \(V_N\). By
	the hypercontractivity inequality on a fixed chaos
	(see, e.g., \cite[Theorem~1.4.1]{Nualart2006}), for every \(p\ge1\) there is a
	constant \(C_p\), independent of \(N\), such that
	\[
	\mathbb{E}\bigl[|V_N|^{2p}\bigr]\le C_p\bigl(\mathbb{E}[V_N^{2}]\bigr)^{p}
	\le C_pC_0^{p}N^{-p}.
	\]
	Fix \(\varepsilon\in(0,\tfrac12)\). Markov's inequality gives
	\[
	\mathbb{P}\bigl(|V_N|\ge N^{-\varepsilon}\bigr)
	\le N^{2\varepsilon p}\,\mathbb{E}\bigl[|V_N|^{2p}\bigr]
	\le C_pC_0^{p}N^{(2\varepsilon-1)p}.
	\]
	Because \(2\varepsilon-1<0\), we may choose \(p\) so large that
	\((2\varepsilon-1)p<-1\); the series \(\sum_{N\ge1}N^{(2\varepsilon-1)p}\) then
	converges, and the Borel--Cantelli lemma implies
	\[
	\mathbb{P}\bigl(|V_N|\ge N^{-\varepsilon}\ \text{for infinitely many }N\bigr)=0 .
	\]
	Hence \(V_N\to0\) almost surely. The deterministic part is \(O(N^{-1})\)
	by~\eqref{eq:var_inc}, therefore
	\[
	\sum_{j=1}^{N}\Delta^{2}(x_{j-1},\delta)-\tfrac12(A_2-A_1)
	=V_N+O(N^{-1})\xrightarrow{\ \text{a.s.}\ }0 .
	\qedhere
	\]
\end{proof}

\subsection{Temporal quartic variation}

Fix a spatial point \(x\in\mathbb{T}\) and an interval \([T_1,T_2]\) with
\(0<T_1<T_2\). We examine the quartic variation of the temporal trajectory
\(t\mapsto\psi(t,x)\) over \([T_1,T_2]\). Let \(s,t\in[T_1,T_2]\) and let the step
\(\delta\) satisfy \(0<\delta\le t-s\); the increments \(\Gamma(s,\delta)\),
\(\Gamma(t,\delta)\) are those of Notation~\ref{not:increments}. In what follows we
assume \(\delta\) to be small enough for the asymptotic expansion~\eqref{eq:ups:expansion}
of \(\Upsilon\), derived in Appendix~\ref{appendix:C}, to apply; this is justified
because the final limit procedure sends \(\delta\) to zero.

\medskip
\noindent\textbf{Expected square of an increment.}
By the definition~\eqref{eq:ups_definition} of \(\Upsilon\) and the
expansion~\eqref{eq:ups:expansion},
\[
\begin{aligned}
\mathbb{E}[\Gamma^{2}(s,\delta)]
&=\mathbb{E}[\psi(s+\delta,x)^{2}]-2\mathbb{E}[\psi(s+\delta,x)\psi(s,x)]
  +\mathbb{E}[\psi(s,x)^{2}]\\
&=\frac{1}{2\pi}\bigl(\Upsilon(0)-\Upsilon(\delta)\bigr)
 =\frac{1}{2\pi}\bigl(2\sqrt{\pi\delta}+O(\delta^{3/2})\bigr),
\end{aligned}
\]
that is,
\begin{equation}\label{square_time}
	\mathbb{E}[\Gamma^{2}(s,\delta)]=\frac{1}{\sqrt\pi}\,\delta^{1/2}+O(\delta^{3/2}).
\end{equation}

\medskip
\noindent\textbf{Expected cross product.}
Similarly, by~\eqref{eq:ups_definition}, with \(\tau=t-s\ge\delta>0\),
\[
\begin{aligned}
\mathbb{E}[\Gamma(s,\delta)\Gamma(t,\delta)]
&=\frac{1}{4\pi}\Bigl[\Upsilon(|s+\delta-t-\delta|)+\Upsilon(|s-t|)
-\Upsilon(|s+\delta-t|)-\Upsilon(|s-t-\delta|)\Bigr]\\
&=\frac{1}{4\pi}\Bigl[2\Upsilon(\tau)-\Upsilon(\tau-\delta)-\Upsilon(\tau+\delta)\Bigr].
\end{aligned}
\]
Using the identity (Taylor's theorem applied to the second difference)
\[
2\Upsilon(\tau)-\Upsilon(\tau-\delta)-\Upsilon(\tau+\delta)
=-\int_{-\delta}^{\delta}(\delta-|u|)\,\Upsilon''(\tau+u)\,du
\]
and the fact that \(\Upsilon''\) is positive~\eqref{eq:ups_2_der} and
decreasing~\eqref{eq:ups_3_der}, we estimate this expression in three regimes.

\noindent\textbf{Case \(\tau\ge1\).} Provided \(\delta\) is small,
\(\tau-\delta\ge\tau/2\ge1/2\), and the monotonicity of \(\Upsilon''\) yields
\[
|2\Upsilon(\tau)-\Upsilon(\tau-\delta)-\Upsilon(\tau+\delta)|
\le\delta^{2}\Upsilon''(\tau-\delta)\le\delta^{2}\Upsilon''(1/2)=O(\delta^{2}).
\]

\noindent\textbf{Case \(2\delta\le\tau<1\).} Here \(\tau-\delta\ge\tau/2\ge\delta\).
The asymptotic behaviour~\eqref{eq:ups:expansion2} of \(\Upsilon''\) gives
\(\Upsilon''(v)\le Cv^{-3/2}\) for \(v\in(0,1)\). Hence
\[
|2\Upsilon(\tau)-\Upsilon(\tau-\delta)-\Upsilon(\tau+\delta)|
\le\delta^{2}\Upsilon''(\tau-\delta)\le C\delta^{2}(\tau/2)^{-3/2}=O(\tau^{-3/2}\delta^{2}).
\]

\noindent\textbf{Case \(\delta\le\tau\le2\delta\).} In this range we use the
expansion~\eqref{eq:ups:expansion} of \(\Upsilon\) directly:
\[
2\Upsilon(\tau)-\Upsilon(\tau-\delta)-\Upsilon(\tau+\delta)
=2\sqrt\pi\bigl(-2\sqrt\tau+\sqrt{\tau-\delta}+\sqrt{\tau+\delta}\bigr)+O(\delta^{3/2}).
\]
Set \(u=\delta/\tau\in[\tfrac12,1]\) and \(g(u)=-2+\sqrt{1-u}+\sqrt{1+u}\), so
that \(-2\sqrt\tau+\sqrt{\tau-\delta}+\sqrt{\tau+\delta}=\sqrt\tau\,g(u)\). Since
\(g(u)/u^{2}\) is continuous on \([\tfrac12,1]\), there is a constant \(K\) with
\(|g(u)|\le Ku^{2}\), and hence
\[
|2\Upsilon(\tau)-\Upsilon(\tau-\delta)-\Upsilon(\tau+\delta)|
\le2\sqrt\pi K\,\frac{\delta^{2}}{\tau^{3/2}}+O(\delta^{3/2}).
\]
The error term is absorbed: using \(\tau\le2\delta\) and \(\delta<1\),
\[
\delta^{3/2}=\delta^{3/2}\,\frac{\tau^{3/2}}{\tau^{3/2}}
\le\delta^{3/2}\,\frac{(2\delta)^{3/2}}{\tau^{3/2}}
=2^{3/2}\,\delta\,\frac{\delta^{2}}{\tau^{3/2}}
\le C\,\frac{\delta^{2}}{\tau^{3/2}} .
\]
Therefore, uniformly in \(\tau\ge\delta\),
\begin{equation}\label{croos_time}
	\mathbb{E}[\Gamma(s,\delta)\Gamma(t,\delta)]=O(\delta^{2})+O\bigl((t-s)^{-3/2}\delta^{2}\bigr).
\end{equation}

\medskip
\noindent\textbf{Higher moments via Wick's theorem.}
Using Wick's theorem for the Gaussian increments (see Section~2
in~\cite{PospisilTribe2007}) together with~\eqref{square_time}
and~\eqref{croos_time}, we obtain the following relations, valid uniformly for
\(s,t\in[T_1,T_2]\) with \(\delta\le t-s\):
\begin{align}
	\mathbb{E}[\Gamma^{4}(s,\delta)]&=\frac{3}{\pi}\,\delta+O(\delta^{2}),\label{four_time}\\
	\mathbb{E}[\Gamma^{8}(s,\delta)]&=O(\delta^{2}),\label{eq:eigth_time}\\
	\mathbb{E}[\Gamma^{4}(s,\delta)\Gamma^{4}(t,\delta)]
	&=\frac{9}{\pi^{2}}\,\delta^{2}+O(\delta^{3})+O\bigl((t-s)^{-3}\delta^{5}\bigr).\label{eq:croos_quartic_time}
\end{align}
Indeed, for a centred Gaussian pair \((X,Y)\) with variances \(a,b\) and covariance
\(\rho\), Wick's theorem gives \(\mathbb{E}[X^{4}]=3a^{2}\),
\(\mathbb{E}[X^{8}]=105a^{4}\) and
\(\mathbb{E}[X^{4}Y^{4}]=9a^{2}b^{2}+72ab\rho^{2}+24\rho^{4}\); inserting
\(a,b=\pi^{-1/2}\delta^{1/2}+O(\delta^{3/2})\) and
\(\rho=O(\delta^{2})+O((t-s)^{-3/2}\delta^{2})\) yields the three relations.

\begin{proposition}[Temporal quartic variation]\label{prop:temp_quartic_L2}
	Fix \(0<T_1<T_2\) and \(x\in\mathbb{T}\). For \(j=0,1,\dots,N\) consider the
	partition \(t_j=T_1+j\delta\) with \(\delta=(T_2-T_1)/N\). Then, in
	\(L^{2}(\Omega)\),
	\[
	\lim_{N\to\infty}\sum_{j=1}^{N}\bigl(\psi(t_j,x)-\psi(t_{j-1},x)\bigr)^{4}
	=\frac{3}{\pi}(T_2-T_1).
	\]
\end{proposition}

\begin{proof}
	Let \(c=3/\pi\). Convergence in \(L^{2}(\Omega)\) is equivalent to
	\[
	\lim_{N\to\infty}\sum_{j=1}^{N}\sum_{k=1}^{N}
	\mathbb{E}\Bigl[\bigl(\Gamma^{4}(t_{j-1},\delta)-c\delta\bigr)
	\bigl(\Gamma^{4}(t_{k-1},\delta)-c\delta\bigr)\Bigr]=0 .
	\]
	
	\noindent\textbf{Diagonal terms (\(j=k\)).} Using~\eqref{four_time}
	and~\eqref{eq:eigth_time},
	\[
	\begin{aligned}
	\mathbb{E}\bigl[(\Gamma^{4}(t_{j-1},\delta)-c\delta)^{2}\bigr]
	&=\mathbb{E}[\Gamma^{8}(t_{j-1},\delta)]-2c\delta\,\mathbb{E}[\Gamma^{4}(t_{j-1},\delta)]+c^{2}\delta^{2}\\
	&=O(\delta^{2})-2c\delta\bigl(c\delta+O(\delta^{2})\bigr)+c^{2}\delta^{2}=O(\delta^{2}),
	\end{aligned}
	\]
	and summing over the \(N\) diagonal terms gives
	\[
	\sum_{j=1}^{N}\mathbb{E}\bigl[(\Gamma^{4}(t_{j-1},\delta)-c\delta)^{2}\bigr]
	\le N\,O(\delta^{2})=O(\delta)\to0 .
	\]
	
	\noindent\textbf{Off-diagonal terms (\(j\neq k\)).} By symmetry it suffices to
	bound the sum over \(j<k\). For \(j<k\) we expand
	\[
	\begin{aligned}
	&\mathbb{E}\Bigl[\bigl(\Gamma^{4}(t_{j-1},\delta)-c\delta\bigr)\bigl(\Gamma^{4}(t_{k-1},\delta)-c\delta\bigr)\Bigr]\\
	&\qquad=\mathbb{E}[\Gamma^{4}(t_{j-1},\delta)\Gamma^{4}(t_{k-1},\delta)]
	-c\delta\bigl(\mathbb{E}[\Gamma^{4}(t_{j-1},\delta)]+\mathbb{E}[\Gamma^{4}(t_{k-1},\delta)]\bigr)
	+c^{2}\delta^{2}.
	\end{aligned}
	\]
	By~\eqref{four_time} each marginal expectation equals \(c\delta+O(\delta^{2})\), so
	the middle term equals \(-2c^{2}\delta^{2}+O(\delta^{3})\); and
	by~\eqref{eq:croos_quartic_time},
	\[
	\mathbb{E}[\Gamma^{4}(t_{j-1},\delta)\Gamma^{4}(t_{k-1},\delta)]
	=c^{2}\delta^{2}+O(\delta^{3})+O\bigl((t_k-t_j)^{-3}\delta^{5}\bigr),
	\]
	with constants uniform in \(j,k,N\). The \(c^{2}\delta^{2}\) terms cancel and we
	obtain
	\[
	\Bigl|\mathbb{E}\Bigl[\bigl(\Gamma^{4}(t_{j-1},\delta)-c\delta\bigr)\bigl(\Gamma^{4}(t_{k-1},\delta)-c\delta\bigr)\Bigr]\Bigr|
	\le C\bigl(\delta^{3}+(t_k-t_j)^{-3}\delta^{5}\bigr)
	\]
	for an absolute constant \(C\). Summing over \(1\le j<k\le N\) and using
	\(t_k-t_j=(k-j)\delta\),
	\[
	\sum_{1\le j<k\le N}\delta^{3}\le N^{2}\delta^{3}=O(\delta),
	\qquad
	\sum_{1\le j<k\le N}(t_k-t_j)^{-3}\delta^{5}
	=\delta^{2}\sum_{m=1}^{N-1}\frac{N-m}{m^{3}}
	\le\delta^{2}N\sum_{m=1}^{\infty}\frac{1}{m^{3}}=O(\delta).
	\]
	Thus the whole off-diagonal sum is \(O(\delta)\) and tends to \(0\) as
	\(N\to\infty\).
\end{proof}

\medskip
\noindent\textbf{Almost sure convergence.}
The almost sure convergence follows by the same hypercontractivity argument as in
Proposition~\ref{prop:spatial_qv_as}, because the centred sums again belong to a
fixed Wiener chaos and the \(L^{2}\) rate is \(O(1/N)\); an analogous argument
can be found in~\cite[Section~4, Proposition~9]{AssaadTudor2021}.

\begin{proposition}[Almost sure convergence of the temporal quartic variation]\label{prop:AS_temp_quartic}
	Under the notation and hypotheses of Proposition~\ref{prop:temp_quartic_L2},
	\[
	\lim_{N\to\infty}\sum_{j=1}^{N}\bigl(\psi(t_j,x)-\psi(t_{j-1},x)\bigr)^{4}
	=\frac{3}{\pi}(T_2-T_1)
	\qquad\text{almost surely.}
	\]
\end{proposition}

\begin{proof}
	Let \(c=3/\pi\) and define the centred sum
	\[
	V_N=\sum_{j=1}^{N}\bigl(\Gamma^{4}(t_{j-1},\delta)-\mathbb{E}[\Gamma^{4}(t_{j-1},\delta)]\bigr).
	\]
	From the proof of Proposition~\ref{prop:temp_quartic_L2} we have the moment
	estimate~\eqref{four_time} and the off-diagonal bounds derived
	from~\eqref{eq:croos_quartic_time}, which together give
	\(\mathbb{E}[V_N^{2}]=O(N^{-1})\): there is \(C_0>0\) with
	\(\mathbb{E}[V_N^{2}]\le C_0/N\) for all \(N\ge1\). Each \(\Gamma(t_{j-1},\delta)\)
	is a centred Gaussian variable, so its centred fourth power belongs to the direct
	sum of the second and fourth Wiener chaoses; consequently \(V_N\) lies in a finite
	Wiener chaos of order at most \(4\). By hypercontractivity, for every integer
	\(p\ge1\) there is a constant \(C_p\), depending only on \(p\) and on the order of
	the chaos (hence independent of \(N\)), such that
	\[
	\mathbb{E}\bigl[|V_N|^{2p}\bigr]\le C_p\bigl(\mathbb{E}[V_N^{2}]\bigr)^{p}\le C_pC_0^{p}N^{-p}.
	\]
	Fix \(\varepsilon\in(0,\tfrac12)\). By Markov's inequality,
	\[
	\mathbb{P}\bigl(|V_N|\ge N^{-\varepsilon}\bigr)
	\le N^{2\varepsilon p}\,\mathbb{E}\bigl[|V_N|^{2p}\bigr]\le C_pC_0^{p}N^{(2\varepsilon-1)p}.
	\]
	Since \(2\varepsilon-1<0\), we can choose \(p\) so large that
	\((2\varepsilon-1)p<-1\); the series \(\sum_{N\ge1}N^{(2\varepsilon-1)p}\) then
	converges and the Borel--Cantelli lemma yields
	\(\mathbb{P}(|V_N|\ge N^{-\varepsilon}\text{ for infinitely many }N)=0\). Thus
	\(V_N\to0\) almost surely. The deterministic part is handled
	by~\eqref{four_time}:
	\[
	\sum_{j=1}^{N}\mathbb{E}[\Gamma^{4}(t_{j-1},\delta)]
	=\sum_{j=1}^{N}\bigl(c\delta+O(\delta^{2})\bigr)
	=c(T_2-T_1)+N\cdot O(\delta^{2})=c(T_2-T_1)+O(N^{-1}),
	\]
	because \(\delta=(T_2-T_1)/N\). Putting everything together,
	\[
	\sum_{j=1}^{N}\Gamma^{4}(t_{j-1},\delta)-\frac{3}{\pi}(T_2-T_1)
	=V_N+O(N^{-1})\xrightarrow{\ \text{a.s.}\ }0 .
	\qedhere
	\]
\end{proof}

\begin{remark}[Recovering the general \(\sigma>0\)]\label{rem:scaling_variations}
	All the results stated for \(\sigma=1\) extend immediately to arbitrary
	\(\sigma>0\) by linearity: \(\psi\) is multiplied by \(\sigma\), so the quadratic
	variation picks up a factor \(\sigma^{2}\) and the quartic variation a factor
	\(\sigma^{4}\). Under the general noise intensity,
	\[
	\lim_{N\to\infty}\sum_{j=1}^{N}\bigl(\psi(t,x_j)-\psi(t,x_{j-1})\bigr)^{2}
	=\frac{\sigma^{2}}{2}(A_2-A_1),
	\]
	\[
	\lim_{N\to\infty}\sum_{j=1}^{N}\bigl(\psi(t_j,x)-\psi(t_{j-1},x)\bigr)^{4}
	=\frac{3\sigma^{4}}{\pi}(T_2-T_1),
	\]
	both in \(L^{2}(\Omega)\) and almost surely.
\end{remark}

\section{Time scaling and the role of the diffusivity parameter \texorpdfstring{$\theta$}{theta}}\label{sec:escaling}

In this section we analyse the effect of the diffusivity \(\theta>0\) on the
solution. The results of Sections~\ref{sec:regularidad}
and~\ref{sec:variaciones} concern the equation with \(\theta=1\); we show here
that the general case reduces to it by an exact change of time, which is the
basis of the estimation procedure of Section~\ref{sec:estimadores}. Consider
\begin{equation}\label{eq:theta}
	du_\theta=\theta\,\partial_x^{2}u_\theta\,dt+f(u_\theta)\,dt
	+g(u_\theta)\,\partial_x u_\theta\,dt+\sigma\,dW(t),
	\qquad (t,x)\in[0,T]\times\mathbb{T},
\end{equation}
which is~\eqref{eq:Burgers_intro}, with initial datum \(u_0\). The key
observation is a simple time scaling: if we set
\[
\hat u(t,x):=u_\theta\Bigl(\frac{t}{\theta},x\Bigr),\qquad t\in[0,\theta T],
\]
then, at a formal level, \(\hat u\) satisfies
\[
d\hat u=\partial_x^{2}\hat u\,dt+\frac1\theta f(\hat u)\,dt
+\frac1\theta g(\hat u)\,\partial_x\hat u\,dt+\frac{\sigma}{\sqrt\theta}\,d\widetilde W(t),
\]
where \(\widetilde W\) is the time-rescaled noise of~\eqref{eq:wn_scaling} below.
The rest of the section makes this rigorous at the level of the mild
formulation and of the decomposition \(u=\psi+U+v\).

\subsection{The rescaled noise and the scaled stationary process}
Under the time change \(t\mapsto t/\theta\), the noise \(W\) is not preserved:
it becomes a different cylindrical Wiener process, and the stationary solution
becomes a rescaled copy of the one driven by that new noise. This subsection
identifies both objects precisely, since it is on their pathwise identification,
rather than on equality in law, that the scaling identity of
Proposition~\ref{thm:scaling_equivalence} rests.

\begin{lemma}\label{lem:wn_scaling}
	The process
	\begin{equation}\label{eq:wn_scaling}
		\widetilde W(t):=\sqrt\theta\,W\Bigl(\frac t\theta\Bigr),\qquad t\in\mathbb{R},
	\end{equation}
	is again a standard cylindrical Wiener process on
	\(L^{2}(\mathbb{T};\mathbb{R}^{n})\).
\end{lemma}

\begin{proof}
	Recall from Section~\ref{subsec:psi_decomp} that \(\langle W(t),h\rangle\)
	denotes the real Brownian motion obtained by testing the cylindrical process
	against \(h\in H=L^{2}(\mathbb{T};\mathbb{R}^{n})\), with
	covariance~\eqref{eq:cyl_cov}. The family
	\(\langle\widetilde W(t),h\rangle:=\sqrt\theta\,\langle W(t/\theta),h\rangle\) is
	centred, Gaussian and linear in \(h\), and for \(h,k\in H\),
	by~\eqref{eq:cyl_cov},
	\[
	\mathbb{E}\bigl[\langle\widetilde W(t),h\rangle\langle\widetilde W(s),k\rangle\bigr]
	=\theta\,\mathbb{E}\Bigl[\bigl\langle W(\tfrac t\theta),h\bigr\rangle
	\bigl\langle W(\tfrac s\theta),k\bigr\rangle\Bigr]
	=\theta\,\Bigl(\frac t\theta\wedge\frac s\theta\Bigr)\langle h,k\rangle
	=(t\wedge s)\,\langle h,k\rangle .
	\qedhere
	\]
\end{proof}

Let \(\psi^{(\theta)}\) be the stationary solution of the linear equation
\[
d\psi^{(\theta)}=\theta(\partial_x^{2}-1)\psi^{(\theta)}\,dt+\sigma\,dW(t),
\qquad (t,x)\in\mathbb{R}\times\mathbb{T},
\]
which admits the mild representation
\[
\psi^{(\theta)}(t,x)=\sigma\int_{-\infty}^{t}\!\int_{\mathbb{T}}
p^{(\theta)}_{t-s}(x-y)\,W(ds,dy),
\qquad p^{(\theta)}_t(x)=e^{-\theta t}p_{\theta t}(x),
\]
\(p^{(\theta)}\) being the kernel of \(e^{t\theta(\partial_x^{2}-1)}\). Let
also \(\widetilde\psi\) denote the stationary solution
of~\eqref{eq:estacionario_H}, with the \emph{same} noise intensity \(\sigma\),
driven by \(\widetilde W\) instead of \(W\):
\[
\widetilde\psi(t,x)=\sigma\int_{-\infty}^{t}\!\int_{\mathbb{T}}
p^{E}_{t-s}(x-y)\,\widetilde W(ds,dy).
\]

\begin{lemma}\label{lem:scaling_psi}
	For every \(t\in\mathbb{R}\) and \(x\in\mathbb{T}\), almost surely,
	\begin{equation}\label{eq:scaling_psi}
		\psi^{(\theta)}\Bigl(\frac t\theta,x\Bigr)=\frac{1}{\sqrt\theta}\,\widetilde\psi(t,x).
	\end{equation}
	Consequently the time-changed process
	\(\varphi(t,x):=\psi^{(\theta)}(t/\theta,x)=\theta^{-1/2}\widetilde\psi(t,x)\)
	solves
	\[
	d\varphi=(\partial_x^{2}-1)\varphi\,dt+\frac{\sigma}{\sqrt\theta}\,d\widetilde W(t).
	\]
\end{lemma}

\begin{proof}
	In the mild representation of \(\psi^{(\theta)}(t/\theta,x)\) substitute
	\(s=r/\theta\). The kernel becomes
	\[
	p^{(\theta)}_{t/\theta-r/\theta}=e^{-(t-r)}\,p_{t-r}=p^{E}_{t-r},
	\]
	and the martingale measure \(W(d(r/\theta),dy)\) is
	\(\theta^{-1/2}\widetilde W(dr,dy)\) by~\eqref{eq:wn_scaling}. Hence
	\[
	\psi^{(\theta)}\Bigl(\frac t\theta,x\Bigr)
	=\frac{\sigma}{\sqrt\theta}\int_{-\infty}^{t}\!\int_{\mathbb{T}}
	p^{E}_{t-r}(x-y)\,\widetilde W(dr,dy)
	=\frac{1}{\sqrt\theta}\,\widetilde\psi(t,x).
	\]
	The equation for \(\varphi\) follows from that of \(\widetilde\psi\) by
	linearity.
\end{proof}

\begin{remark}[The law of \(\widetilde\psi\)]\label{rem:law_only}
	Since \(\widetilde W\) and \(W\) are both standard cylindrical Wiener processes,
	\(\widetilde\psi\) has the same law as \(\psi\); as processes they differ,
	\(\widetilde\psi\) being a time-rescaling of \(\psi\). Everything proved about
	\(\psi\) in Section~\ref{sec:variaciones} — the covariance~\eqref{eq:psi_cov},
	Propositions~\ref{prop:L2_spatial_conv}--\ref{prop:AS_temp_quartic} and
	Remark~\ref{rem:scaling_variations} — is a statement about the law of the
	process, and therefore holds verbatim for \(\widetilde\psi\). This is what will
	be used in Section~\ref{sec:estimadores}. Likewise \(\widetilde\psi\) admits a
	canonical lift \(\widetilde{\boldsymbol\Psi}=(\widetilde\Psi,\widetilde{\boldsymbol\Psi})\)
	by Lemma~\ref{lem:rp_levantamiento}, and \(K_T\) of~\eqref{eq:KT} has an
	obvious counterpart for it.
\end{remark}

\subsection{The rescaled equation with unit diffusivity}

The previous observation suggests that we study the auxiliary equation
\begin{equation}\label{eq:uhat}
	d\hat u=\partial_x^{2}\hat u\,dt
	+\frac1\theta\bigl(f(\hat u)\,dt+g(\hat u)\,\partial_x\hat u\,dt\bigr)
	+\frac{\sigma}{\sqrt\theta}\,d\widetilde W(t),
	\qquad t\in[0,\theta T],
\end{equation}
with initial condition \(\hat u(0,\cdot)=u_0\). Its coefficients \(f/\theta\),
\(g/\theta\) satisfy Assumption~\ref{ass:standing}(A2), so
Theorem~\ref{theorem:Existence_Uq} applies: \eqref{eq:uhat} has a unique mild
solution on \([0,\theta T]\), in the sense of
Definition~\ref{def:mild_principal} with \(\psi\) replaced by the stationary
process \(\varphi\) of Lemma~\ref{lem:scaling_psi}. Explicitly,
\begin{equation}\label{eq:uhat_mild_solution}
	\hat u=\varphi+U+v^{\varphi},
	\qquad U_t=S_t(u_0-\varphi_0),
\end{equation}
where \(\varphi_0=\psi^{(\theta)}(0,\cdot)=\theta^{-1/2}\widetilde\psi_0\), and
the remainder \(v^{\varphi}\in C([0,\theta T];C^{1})\) satisfies, for
\(t\in[0,\theta T]\) and \(x\in\mathbb{T}\),
\begin{equation}\label{eq:vvarphi_mild}
	\begin{aligned}
		v^{\varphi}_t(x)
		&=\int_0^t S_{t-s}\Bigl(\frac{g(\hat u_s)}{\theta}\,(\partial_x v^{\varphi}_s+\partial_x U_s)
		+\frac{f(\hat u_s)}{\theta}+\varphi_s\Bigr)(x)\,ds\\
		&\quad+\int_0^t\!\int_{\mathbb{T}}p_{t-s}(x-y)\,\frac{g(\hat u_s(y))}{\theta}\,d\Phi_s(y)\,ds .
	\end{aligned}
\end{equation}
Here \(\boldsymbol\Phi=(\Phi,\boldsymbol\Phi)\) denotes the canonical lift of
\(\varphi\). Since \(\varphi=\theta^{-1/2}\widetilde\psi\), the two levels scale
differently — the first is homogeneous of degree one in the path and the second,
being an iterated integral, of degree two — so that
\begin{equation}\label{eq:lift_scaling}
	\Phi_t=\theta^{-1/2}\,\widetilde\Psi_t,
	\qquad
	\boldsymbol\Phi_t=\theta^{-1}\,\widetilde{\boldsymbol\Psi}_t ,
\end{equation}
and in particular
\(\|\Phi_t\|_{C^{\alpha}}=\theta^{-1/2}\|\widetilde\Psi_t\|_{C^{\alpha}}\),
\(\|\boldsymbol\Phi_t\|_{2\alpha}=\theta^{-1}\|\widetilde{\boldsymbol\Psi}_t\|_{2\alpha}\).
The rough integral is homogeneous of degree one under this dilation: if \(Y\) is
controlled by \(\widetilde\Psi_s\) with derivative process \(Y'\), it is
controlled by \(\Phi_s\) with derivative process \(\theta^{1/2}Y'\), and the
compensated Riemann sums~\eqref{eq:rough_riemman} for the two integrals differ
exactly by the factor \(\theta^{-1/2}\):
\begin{equation}\label{eq:rough_int_scaling}
	\int_0^{2\pi}Y(y)\,d\Phi_s(y)=\theta^{-1/2}\int_0^{2\pi}Y(y)\,d\widetilde\Psi_s(y)
\end{equation}
(see~\cite[Remark~2.4]{Hairer2011}). By Lemma~\ref{lem:rp_composition} the
constants of Section~\ref{sec:regularidad} for the rescaled equation depend on
\(\theta\), but \(\theta\) is fixed throughout, so Proposition~\ref{prop:regularity_v}
and Corollary~\ref{cor:R_properties} apply to \(v^{\varphi}\) and \(R:=U+v^{\varphi}\)
without change.

\subsection{The scaling identity}

We now transport the decomposition~\eqref{eq:uhat_mild_solution} back to
\(u_\theta\). For \(t\in[0,T]\) and \(x\in\mathbb{T}\) define
\begin{equation}\label{eq:scaling_theta}
	\begin{aligned}
		v^{(\theta)}(t,x)
		&:=\int_0^t S_{\theta(t-s)}\Bigl(g(\hat u_{\theta s})\,(\partial_x v^{\varphi}_{\theta s}+\partial_x U_{\theta s})
		+f(\hat u_{\theta s})+\sqrt\theta\,\widetilde\psi_{\theta s}\Bigr)(x)\,ds\\
		&\quad+\frac{1}{\sqrt\theta}\int_0^t\!\int_{\mathbb{T}}p_{\theta(t-s)}(x-y)\,
		g(\hat u_{\theta s}(y))\,d\widetilde\Psi_{\theta s}(y)\,ds ,
	\end{aligned}
\end{equation}
and
\[
U^{(\theta)}_t:=S_{\theta t}\bigl(u_0-\psi^{(\theta)}_0\bigr).
\]
The integrands in~\eqref{eq:scaling_theta} are expressed through the processes
\(\hat u\), \(v^{\varphi}\), \(U\), already constructed on \([0,\theta T]\), and
through the lift \(\widetilde{\boldsymbol\Psi}\); by~\eqref{eq:lift_scaling}
and~\eqref{eq:rough_int_scaling} one may rewrite the rough integral in terms of
\(\boldsymbol\Phi\), and we have chosen the present form because it matches
the mild formulation of~\eqref{eq:theta} directly.

\begin{proposition}[Time scaling equivalence]\label{thm:scaling_equivalence}
	Let \(u_\theta\) be the mild solution of~\eqref{eq:theta} on \([0,T]\) with
	initial datum \(u_0\), and let \(\hat u\) be the mild solution
	of~\eqref{eq:uhat} on \([0,\theta T]\) with the same initial datum. Then
	\begin{equation}\label{eq:main_identity}
		\hat u(t,x)=u_\theta\Bigl(\frac t\theta,x\Bigr),
		\qquad t\in[0,\theta T],\ x\in\mathbb{T},
	\end{equation}
	and the decomposition is preserved by the scaling: with
	\(\hat u=\varphi+U+v^{\varphi}\) as in~\eqref{eq:uhat_mild_solution},
	\[
	u_\theta=\psi^{(\theta)}+U^{(\theta)}+v^{(\theta)},
	\]
	with the explicit relations
	\begin{equation}\label{eq:vtheta_scaling}
		\psi^{(\theta)}(t,x)=\frac{1}{\sqrt\theta}\,\widetilde\psi(\theta t,x),\qquad
		U^{(\theta)}(t,x)=U(\theta t,x),\qquad
		v^{(\theta)}(t,x)=v^{\varphi}(\theta t,x),
	\end{equation}
	for \(t\in[0,T]\). Moreover \(v^{(\theta)}\) satisfies
	\begin{equation}\label{eq:mild_solution_theta}
		\begin{aligned}
			v^{(\theta)}_t(x)
			&=\int_0^t S_{\theta(t-s)}\Bigl(g(u_\theta(s))\,(\partial_x v^{(\theta)}_s+\partial_x U^{(\theta)}_s)
			+f(u_\theta(s))+\sqrt\theta\,\widetilde\psi_{\theta s}\Bigr)(x)\,ds\\
			&\quad+\frac{1}{\sqrt\theta}\int_0^t\!\int_{\mathbb{T}}p_{\theta(t-s)}(x-y)\,
			g\bigl(u_\theta(s,y)\bigr)\,d\widetilde\Psi_{\theta s}(y)\,ds ,
		\end{aligned}
	\end{equation}
	that is, \(u_\theta\) is a mild solution of~\eqref{eq:theta} in the sense of
	Definition~\ref{def:mild_principal}, with the semigroup \(S_{\theta t}\)
	generated by \(\theta\partial_x^{2}\).
\end{proposition}

\begin{proof}
	The first relation in~\eqref{eq:vtheta_scaling} is Lemma~\ref{lem:scaling_psi}
	with \(t\) replaced by \(\theta t\). The second is immediate:
	\(U^{(\theta)}(t/\theta,x)=S_t(u_0-\psi^{(\theta)}_0)(x)=S_t(u_0-\varphi_0)(x)=U_t(x)\),
	using \(\psi^{(\theta)}_0=\varphi_0\).
	
	For the third, the change of variables \(s\mapsto s/\theta\) in the definition
	\eqref{eq:scaling_theta}, together with \(S_{\theta(t/\theta-s/\theta)}=S_{t-s}\),
	\(p_{\theta(t/\theta-s/\theta)}=p_{t-s}\), \(\sqrt\theta\,\widetilde\psi_s/\theta=\varphi_s\)
	and~\eqref{eq:rough_int_scaling}, gives
	\[
	\begin{aligned}
		v^{(\theta)}\Bigl(\frac t\theta,x\Bigr)
		&=\int_0^t S_{t-s}\Bigl(\frac{g(\hat u_s)}{\theta}\,(\partial_x v^{\varphi}_s+\partial_x U_s)
		+\frac{f(\hat u_s)}{\theta}+\varphi_s\Bigr)(x)\,ds\\
		&\quad+\int_0^t\!\int_{\mathbb{T}}p_{t-s}(x-y)\,\frac{g(\hat u_s(y))}{\theta}\,d\Phi_s(y)\,ds
		\;=\;v^{\varphi}(t,x),
	\end{aligned}
	\]
	the last equality being~\eqref{eq:vvarphi_mild}. Hence the process
	\(w(t,x):=\psi^{(\theta)}(t,x)+U^{(\theta)}(t,x)+v^{(\theta)}(t,x)\) satisfies
	\(w(t/\theta,x)=\varphi(t,x)+U_t(x)+v^{\varphi}(t,x)=\hat u(t,x)\), and
	substituting \(\hat u_{\theta s}=w_s\), \(v^{\varphi}_{\theta s}=v^{(\theta)}_s\),
	\(U_{\theta s}=U^{(\theta)}_s\) into~\eqref{eq:scaling_theta} shows that
	\(v^{(\theta)}=w-\psi^{(\theta)}-U^{(\theta)}\) satisfies~\eqref{eq:mild_solution_theta}
	with \(u_\theta\) replaced by \(w\). Thus \(w\) is a mild solution
	of~\eqref{eq:theta} on \([0,T]\) with initial datum \(u_0\) in the sense of
	Definition~\ref{def:mild_principal}; by the uniqueness part of
	Theorem~\ref{theorem:Existence_Uq}, \(w=u_\theta\), which
	gives~\eqref{eq:main_identity}, the third relation in~\eqref{eq:vtheta_scaling},
	and~\eqref{eq:mild_solution_theta}.
\end{proof}

\begin{remark}\label{rem:local_compatibility}
	Under Assumption~\ref{ass:standing}(A2) both \(u_\theta\) and \(\hat u\) are
	global, so the intervals \([0,T]\) and \([0,\theta T]\) match exactly and no
	compatibility question arises. Without the boundedness hypotheses one would
	work on \([0,T_0]\) and \([0,\theta T_0]\), where
	\(T_0=\min\{T,\widehat T/\theta\}\) and \(\widehat T\) is the existence time
	of~\eqref{eq:uhat}; the argument is otherwise unchanged.
\end{remark}

\section{Estimation of the diffusivity parameter}\label{sec:estimadores}

Everything is now in place. Theorem~\ref{theorem:Existence_Uq} provides the mild
solution \(u_\theta\) of~\eqref{eq:theta} for almost every realisation of the
driving rough path; Proposition~\ref{thm:scaling_equivalence} identifies it, up
to the time change \(t\mapsto\theta t\), with the solution \(\hat u\) of the
unit-diffusivity equation~\eqref{eq:uhat}, whose decomposition
\(\hat u=\varphi+R\), \(R:=U+v^{\varphi}\), we control: the Gaussian part
\(\varphi=\theta^{-1/2}\widetilde\psi\) has the explicit variations of
Section~\ref{sec:variaciones}, by Remark~\ref{rem:law_only}, and the remainder
\(R\) has the regularity (R1)--(R2) of Corollary~\ref{cor:R_properties}. All of
these statements hold on a single event of full probability, which we fix once
and for all; every convergence below is understood on that event.

Convention~\ref{conv:component} remains in force: the variations are those of a
single, arbitrarily fixed component of the solution, so that all quantities in
the sums are real-valued, and \(N\) denotes the number of intervals of a
partition. As in Section~\ref{sec:variaciones} we write partitions with
\(N+1\) points and sum over \(j=1,\dots,N\).

\begin{remark}[The noise intensity is assumed known]\label{rem:sigma_known}
	Throughout, \(\sigma\) is a known constant and only \(\theta\) is estimated, as
	stated in Assumption~\ref{ass:standing}(A1). This is not a limitation of the
	method. The two variations computed below converge to
	\(\sigma^{2}(A_2-A_1)/(2\theta)\) and \(3\sigma^{4}(T_2-T_1)/(\pi\theta)\),
	which depend on \((\sigma^{2},\theta)\) through the two independent
	combinations \(\sigma^{2}/\theta\) and \(\sigma^{4}/\theta\); the pair is
	therefore identifiable from a single trajectory, the ratio of the two limits
	eliminating \(\theta\). Explicitly, with \(S_{N,t}\) and \(T_{N,x}\) as
	in~\eqref{eq:S_Nt} and~\eqref{eq:T_Nx} below,
	\[
	\widehat{\sigma^{2}}_{N}
	:=\frac{\pi\,(A_2-A_1)\,T_{N,x}(u_\theta)}{6\,(T_2-T_1)\,S_{N,t}(u_\theta)},
	\qquad
	\widehat\theta_{N}
	:=\frac{\pi\,(A_2-A_1)^{2}\,T_{N,x}(u_\theta)}{12\,(T_2-T_1)\,S_{N,t}(u_\theta)^{2}},
	\]
	are strongly consistent for \(\sigma^{2}\) and \(\theta\) respectively, by
	Theorems~\ref{thm:spatial_qv_solution} and~\ref{thm:temporal_quartic_solution}
	together with the continuous mapping theorem and the positivity argument of
	Theorem~\ref{thm:spatial_estimator}. We do not pursue the joint estimation
	further, since the point of this paper is that variation-based inference is
	available at all in the rough-path setting; but the reader should keep in mind
	that assuming \(\sigma\) known costs nothing.
\end{remark}

\subsection{Spatial estimator}

Let \(t\in(0,T]\) and an interval \([A_1,A_2]\subset\mathbb{T}\) be fixed, and
suppose that the solution is observed at the points
\[
x_j=A_1+j\delta,\qquad \delta=\frac{A_2-A_1}{N},\qquad j=0,\dots,N .
\]
For a space--time function \(X\) we define its spatial quadratic variation along
this partition by
\begin{equation}\label{eq:S_Nt}
	S_{N,t}(X)=\sum_{j=1}^{N}\bigl(X(t,x_j)-X(t,x_{j-1})\bigr)^{2},
\end{equation}
and we write \(\Delta X_j:=X(t,x_j)-X(t,x_{j-1})\) for the spatial increment;
for \(X=\psi\) this is \(\Delta(x_{j-1},\delta)\) in the notation of
Section~\ref{sec:variaciones}.

\begin{theorem}[Spatial quadratic variation of the solution]\label{thm:spatial_qv_solution}
	Let \(u_\theta=\psi^{(\theta)}+U^{(\theta)}+v^{(\theta)}\) be the mild solution
	of~\eqref{eq:theta} on \([0,T]\times\mathbb{T}\), with noise intensity
	\(\sigma\), and let Convention~\ref{conv:component} be in force. Then, for every
	fixed \(t\in(0,T]\) and \([A_1,A_2]\subset\mathbb{T}\), almost surely,
	\[
	\lim_{N\to\infty}S_{N,t}(u_\theta)
	=\lim_{N\to\infty}\sum_{j=1}^{N}\bigl(u_\theta(t,x_j)-u_\theta(t,x_{j-1})\bigr)^{2}
	=\frac{\sigma^{2}}{2\theta}\,(A_2-A_1).
	\]
\end{theorem}

\begin{proof}
	By the identity~\eqref{eq:main_identity}, \(u_\theta(t,x)=\hat u(\theta t,x)\),
	the spatial increments of \(u_\theta\) at time \(t\) are those of \(\hat u\) at
	time \(\theta t\), so
	\[
	S_{N,t}(u_\theta)=S_{N,\theta t}(\hat u),
	\]
	and it suffices to compute the spatial quadratic variation of \(\hat u\) at the
	fixed time \(\theta t\). Recall from~\eqref{eq:uhat_mild_solution} that
	\(\hat u=\varphi+R\) with \(R=U+v^{\varphi}\), so that
	\(\Delta\hat u_j=\Delta\varphi_j+\Delta R_j\). The claim is therefore
	equivalent to
	\[
	\lim_{N\to\infty}\Bigl|\sum_{j=1}^{N}(\Delta\hat u_j)^{2}
	-\sum_{j=1}^{N}(\Delta\varphi_j)^{2}\Bigr|=0
	\qquad\text{almost surely},
	\]
	since, by Proposition~\ref{prop:spatial_qv_as} applied to \(\widetilde\psi\)
	(Remark~\ref{rem:law_only}) and by Remark~\ref{rem:scaling_variations},
	\begin{equation}\label{eq:varphi_spatial_limit}
		\sum_{j=1}^{N}(\Delta\varphi_j)^{2}
		=\frac1\theta\sum_{j=1}^{N}\bigl(\widetilde\psi(\theta t,x_j)-\widetilde\psi(\theta t,x_{j-1})\bigr)^{2}
		\xrightarrow[N\to\infty]{\ \text{a.s.}\ }\frac{\sigma^{2}}{2\theta}(A_2-A_1).
	\end{equation}
	
	For real numbers \(a,b\) one has \(|(a+b)^{2}-a^{2}|=|b^{2}+2ab|\le b^{2}+2|a||b|\).
	Applying this with \(a=\Delta\varphi_j\), \(b=\Delta R_j\) and summing,
	\begin{equation}\label{eq:qv_diff_bound}
		\Bigl|\sum_{j=1}^{N}(\Delta\hat u_j)^{2}-\sum_{j=1}^{N}(\Delta\varphi_j)^{2}\Bigr|
		\le\sum_{j=1}^{N}(\Delta R_j)^{2}+2\sum_{j=1}^{N}|\Delta\varphi_j|\,|\Delta R_j| .
	\end{equation}
	
	\emph{The remainder term.} By (R1) of Corollary~\ref{cor:R_properties}, applied
	to the rescaled equation at the fixed time \(\theta t>0\), the map
	\(x\mapsto R(\theta t,x)\) is Lipschitz on \(\mathbb{T}\): there is an almost
	surely finite \(L\), independent of \(N\), with \(|\Delta R_j|\le L\delta\) for
	all \(j\). Hence
	\[
	\sum_{j=1}^{N}(\Delta R_j)^{2}\le N L^{2}\delta^{2}
	=L^{2}(A_2-A_1)\,\frac{1}{N}\xrightarrow[N\to\infty]{}0 .
	\]
	
	\emph{The mixed term.} By the Cauchy--Schwarz inequality,
	\[
	\sum_{j=1}^{N}|\Delta\varphi_j|\,|\Delta R_j|
	\le\Bigl(\sum_{j=1}^{N}(\Delta\varphi_j)^{2}\Bigr)^{1/2}
	\Bigl(\sum_{j=1}^{N}(\Delta R_j)^{2}\Bigr)^{1/2}.
	\]
	The first factor converges almost surely to the finite
	limit~\eqref{eq:varphi_spatial_limit}, hence is almost surely bounded, while the
	second tends to zero by the previous step. Both terms on the right
	of~\eqref{eq:qv_diff_bound} therefore vanish, and the theorem follows.
\end{proof}

Since Theorem~\ref{thm:spatial_qv_solution} gives
\(S_{N,t}(u_\theta)\to\sigma^{2}(A_2-A_1)/(2\theta)\), solving for \(\theta\)
leads us to define, for fixed \(t\in(0,T]\),
\begin{equation}\label{eq:theta_hat_espacial}
	\widehat\theta_{N,t}:=\frac{\sigma^{2}}{2}\,\frac{A_2-A_1}{S_{N,t}(u_\theta)} .
\end{equation}

\begin{theorem}[Strong consistency of the spatial estimator]\label{thm:spatial_estimator}
	Under Convention~\ref{conv:component}, the estimator \(\widehat\theta_{N,t}\)
	defined in~\eqref{eq:theta_hat_espacial} is strongly consistent:
	\(\widehat\theta_{N,t}\to\theta\) almost surely as \(N\to\infty\).
\end{theorem}

\begin{proof}
	Set \(S_\infty:=\sigma^{2}(A_2-A_1)/(2\theta)>0\). A direct computation gives
	\[
	\widehat\theta_{N,t}-\theta
	=\theta\,\frac{S_\infty-S_{N,t}(u_\theta)}{S_{N,t}(u_\theta)} .
	\]
	By Theorem~\ref{thm:spatial_qv_solution}, \(S_{N,t}(u_\theta)\to S_\infty\)
	almost surely, so for almost every \(\omega\) there is \(N_0(\omega)\) with
	\(S_{N,t}(u_\theta)\ge S_\infty/2\) for all \(N\ge N_0(\omega)\). On this event
	\[
	|\widehat\theta_{N,t}-\theta|\le\frac{2\theta}{S_\infty}\,\bigl|S_\infty-S_{N,t}(u_\theta)\bigr|,
	\]
	and the right-hand side tends to zero almost surely by
	Theorem~\ref{thm:spatial_qv_solution}.
\end{proof}

\subsection{Temporal estimator}

Similarly, let \(x\in\mathbb{T}\) and a time interval \([T_1,T_2]\subset(0,T]\)
with \(T_1>0\) be fixed, and suppose that the solution is observed at the times
\[
t_j=T_1+j\delta,\qquad \delta=\frac{T_2-T_1}{N},\qquad j=0,\dots,N .
\]
Throughout this subsection \([T_1,T_2]\) is a time interval; \([A_1,A_2]\) was
reserved for the spatial interval of the previous one. For a space--time
function \(X\) we define its temporal quartic variation along this partition by
\begin{equation}\label{eq:T_Nx}
	T_{N,x}(X)=\sum_{j=1}^{N}\bigl(X(t_j,x)-X(t_{j-1},x)\bigr)^{4},
\end{equation}
and we write \(\Gamma X_j:=X(t_j,x)-X(t_{j-1},x)\) for the temporal increment;
for \(X=\psi\) this is \(\Gamma(t_{j-1},\delta)\) in the notation of
Section~\ref{sec:variaciones}.

\begin{remark}[The time change and the partition]\label{rem:temporal_mesh}
	Recall that \(u_\theta\) lives on \([0,T]\times\mathbb{T}\) while \(\hat u\)
	lives on \([0,\theta T]\times\mathbb{T}\). By~\eqref{eq:main_identity},
	\(u_\theta(t,x)=\hat u(\theta t,x)\), so the increments of \(u_\theta\) over the
	partition \(t_j=T_1+j\delta\) of \([T_1,T_2]\) coincide with those of \(\hat u\)
	over the rescaled partition \(s_j:=\theta t_j\) of \([\theta T_1,\theta T_2]\),
	whose mesh is \(\theta\delta\). We write
	\[
	T_{N,x}(u_\theta)=T^{\theta}_{N,x}(\hat u),
	\]
	the right-hand side denoting the quartic variation of \(\hat u\) at the same
	spatial point \(x\), over \([\theta T_1,\theta T_2]\), along
	\((s_j)_{j=0}^{N}\). The spatial point is unaffected by the time change, and the
	length of the interval on which \(\hat u\) is observed is \(\theta(T_2-T_1)\);
	this factor is what distinguishes the constants below from those of
	Section~\ref{sec:variaciones}.
\end{remark}

\begin{theorem}[Temporal quartic variation of the solution]\label{thm:temporal_quartic_solution}
	Let \(u_\theta=\psi^{(\theta)}+U^{(\theta)}+v^{(\theta)}\) be the mild solution
	of~\eqref{eq:theta} on \([0,T]\times\mathbb{T}\), with noise intensity
	\(\sigma\), and let Convention~\ref{conv:component} be in force. Then, for every
	fixed \(x\in\mathbb{T}\) and \([T_1,T_2]\subset(0,T]\) with \(T_1>0\), almost
	surely,
	\[
	\lim_{N\to\infty}T_{N,x}(u_\theta)
	=\lim_{N\to\infty}\sum_{j=1}^{N}\bigl(u_\theta(t_j,x)-u_\theta(t_{j-1},x)\bigr)^{4}
	=\frac{3\sigma^{4}}{\pi\theta}\,(T_2-T_1).
	\]
\end{theorem}

\begin{proof}
	By Remark~\ref{rem:temporal_mesh} it suffices to compute
	\(T^{\theta}_{N,x}(\hat u)\), the quartic variation of \(\hat u\) over
	\([\theta T_1,\theta T_2]\) along \(s_j=\theta t_j\). Recall
	\(\hat u=\varphi+R\) with \(R=U+v^{\varphi}\); writing
	\(\Gamma\varphi_j=\varphi(s_j,x)-\varphi(s_{j-1},x)\) and
	\(\Gamma R_j=R(s_j,x)-R(s_{j-1},x)\), so that
	\(\Gamma\hat u_j=\Gamma\varphi_j+\Gamma R_j\), and expanding the fourth power,
	\begin{align}
		T^{\theta}_{N,x}(\hat u)
		&=\sum_{j=1}^{N}(\Gamma\varphi_j+\Gamma R_j)^{4}
		=T^{\theta}_{N,x}(\varphi)\notag\\
		&\quad+4\sum_{j=1}^{N}(\Gamma\varphi_j)^{3}(\Gamma R_j)\label{alig:1}\\
		&\quad+6\sum_{j=1}^{N}(\Gamma\varphi_j)^{2}(\Gamma R_j)^{2}\label{alig:2}\\
		&\quad+4\sum_{j=1}^{N}(\Gamma\varphi_j)(\Gamma R_j)^{3}\label{alig:3}\\
		&\quad+T^{\theta}_{N,x}(R).\label{alig:4}
	\end{align}
	
	\emph{Leading term.} Since \(\varphi=\theta^{-1/2}\widetilde\psi\), the quartic
	variation of \(\varphi\) over an interval of length \(\ell\) is \(\theta^{-2}\)
	times that of \(\widetilde\psi\), namely \(\theta^{-2}\cdot3\sigma^{4}\ell/\pi\)
	by Proposition~\ref{prop:AS_temp_quartic}, which applies to \(\widetilde\psi\)
	by Remark~\ref{rem:law_only}, and by Remark~\ref{rem:scaling_variations}. Here
	\(\ell=\theta(T_2-T_1)\), so
	\begin{equation}\label{eq:varphi_temporal_limit}
		T^{\theta}_{N,x}(\varphi)
		\xrightarrow[N\to\infty]{\ \text{a.s.}\ }
		\frac{3\sigma^{4}}{\pi\theta^{2}}\cdot\theta\,(T_2-T_1)
		=\frac{3\sigma^{4}}{\pi\theta}\,(T_2-T_1).
	\end{equation}
	
	\emph{Remainder term.} By (R2) of Corollary~\ref{cor:R_properties}, applied to
	the rescaled equation on the interval \([\theta T_1,\theta T_2]\subset(0,\theta T]\),
	the map \(s\mapsto R(s,x)\) is \(\rho\)-Hölder there for some \(\rho>\tfrac12\):
	there is an almost surely finite \(C_R\), independent of \(N\), with
	\(|\Gamma R_j|\le C_R(\theta\delta)^{\rho}\) for all \(j\). Hence
	\[
	T^{\theta}_{N,x}(R)=\sum_{j=1}^{N}|\Gamma R_j|^{4}
	\le C_R^{4}\,N\,(\theta\delta)^{4\rho}
	=C_R^{4}\,\theta^{4\rho}\,(T_2-T_1)\,\delta^{4\rho-1}
	\xrightarrow[N\to\infty]{}0,
	\]
	since \(4\rho-1>1>0\).
	
	\emph{Mixed terms.} By Hölder's inequality with exponents \((4/3,4)\), \((2,2)\)
	and \((4,4/3)\) respectively,
	\begin{align*}
		\Bigl|\sum_{j=1}^{N}(\Gamma\varphi_j)^{3}(\Gamma R_j)\Bigr|
		&\le\bigl(T^{\theta}_{N,x}(\varphi)\bigr)^{3/4}\bigl(T^{\theta}_{N,x}(R)\bigr)^{1/4},\\
		\Bigl|\sum_{j=1}^{N}(\Gamma\varphi_j)^{2}(\Gamma R_j)^{2}\Bigr|
		&\le\bigl(T^{\theta}_{N,x}(\varphi)\bigr)^{1/2}\bigl(T^{\theta}_{N,x}(R)\bigr)^{1/2},\\
		\Bigl|\sum_{j=1}^{N}(\Gamma\varphi_j)(\Gamma R_j)^{3}\Bigr|
		&\le\bigl(T^{\theta}_{N,x}(\varphi)\bigr)^{1/4}\bigl(T^{\theta}_{N,x}(R)\bigr)^{3/4}.
	\end{align*}
	In each case the first factor is almost surely bounded, since
	\(T^{\theta}_{N,x}(\varphi)\) converges by~\eqref{eq:varphi_temporal_limit}, and
	the second tends to zero by the previous step. Hence the
	terms~\eqref{alig:1}--\eqref{alig:4} all vanish almost surely, and
	\[
	T^{\theta}_{N,x}(\hat u)\xrightarrow[N\to\infty]{\ \text{a.s.}\ }\frac{3\sigma^{4}}{\pi\theta}\,(T_2-T_1),
	\]
	which by Remark~\ref{rem:temporal_mesh} is the assertion.
\end{proof}

Since Theorem~\ref{thm:temporal_quartic_solution} gives
\(T_{N,x}(u_\theta)\to3\sigma^{4}(T_2-T_1)/(\pi\theta)\), solving for \(\theta\)
leads us to define, for fixed \(x\in\mathbb{T}\),
\begin{equation}\label{eq:theta_hat_temporal}
	\widehat\theta_{N,x}:=\frac{3\sigma^{4}}{\pi}\,\frac{T_2-T_1}{T_{N,x}(u_\theta)} .
\end{equation}

\begin{theorem}[Strong consistency of the temporal estimator]\label{thm:temporal_estimator}
	Under Convention~\ref{conv:component}, the estimator \(\widehat\theta_{N,x}\)
	defined in~\eqref{eq:theta_hat_temporal} is strongly consistent:
	\(\widehat\theta_{N,x}\to\theta\) almost surely as \(N\to\infty\).
\end{theorem}

\begin{proof}
	Identical to that of Theorem~\ref{thm:spatial_estimator}, with
	Theorem~\ref{thm:temporal_quartic_solution} in place of
	Theorem~\ref{thm:spatial_qv_solution} and
	\(T_\infty:=3\sigma^{4}(T_2-T_1)/(\pi\theta)>0\) in place of \(S_\infty\).
\end{proof}

\begin{remark}[On a central limit theorem]\label{rem:no_clt}
	It is natural to ask whether the estimators above satisfy a central limit
	theorem. The obstruction is quantitative and worth recording. Writing
	\(T_{N,x}(\hat u)=T_{N,x}(\varphi)+P_{N,x}\), where \(P_{N,x}\) collects the
	terms~\eqref{alig:1}--\eqref{alig:4}, the dominant contribution
	is~\eqref{alig:1}, and by Hölder's inequality as above,
	\[
	\Bigl|\sum_{j=1}^{N}(\Gamma\varphi_j)^{3}\Gamma R_j\Bigr|
	\le\bigl(T_{N,x}(\varphi)\bigr)^{3/4}\bigl(T_{N,x}(R)\bigr)^{1/4}.
	\]
	If \(R\) is \(\rho\)-Hölder in time, then \(T_{N,x}(R)=O(N^{1-4\rho})\) by the
	computation in the proof of Theorem~\ref{thm:temporal_quartic_solution}, so
	that
	\[
	\sqrt N\,\Bigl|\sum_{j=1}^{N}(\Gamma\varphi_j)^{3}\Gamma R_j\Bigr|
	=O\bigl(N^{\frac{3-4\rho}{4}}\bigr),
	\]
	which tends to zero only if \(\rho>\tfrac34\). Proposition~\ref{prop:regularity_v}
	gives every \(\rho<(1+\alpha)/2\), and \((1+\alpha)/2<\tfrac34\) because
	\(\alpha<\tfrac12\): the remainder is smooth enough for the law of large numbers
	behind Theorem~\ref{thm:temporal_quartic_solution}, but not for the
	fluctuations. A central limit theorem would therefore require either a sharper
	description of \(P_{N,x}\) than the pathwise bound used here, or an analogue for
	\(\widetilde\psi\) of the fourth-moment results of~\cite{CimpNachTudor}. We leave
	this question open; see also Remark~\ref{rem:one_component} on the role of the
	number of components.
\end{remark}

\begin{appendices}
	
	\noindent The appendices collect the proofs of three technical results used in
	the paper. Appendix~\ref{appendix:A} proves Lemma~\ref{lem:heat_prop}, the scaling
	identities for the heat kernel that feed Proposition~\ref{prop:hk_result};
	Appendix~\ref{appendix:B} bounds the deterministic integrals appearing in the
	proof of Proposition~\ref{prop:regularity_v}; Appendix~\ref{appendix:C}
	establishes the properties of the covariance function \(\Upsilon\) used in the
	temporal computations of Section~\ref{sec:variaciones}.
	
	\section{Proof of Lemma~\ref{lem:heat_prop}}\label{appendix:A}
	
	To present the proof of Lemma~\ref{lem:heat_prop}, we first state and prove several auxiliary results.
	
	\begin{lemma}\label{lem:F-norm-finite}
		Let \(F:\mathbb{R}\longrightarrow\mathbb{R}\) be the \(C^{\infty}\) function given by
		\[
		F(x) = \frac{1}{\sqrt{4\pi}} e^{-\frac{x^2}{4}}.
		\]
		For every integer \(\ell \ge 0\), let \(F^{(\ell)}\) denote the \(\ell\)-th derivative of \(F\). Then
		\[
		\|F^{(\ell)}\|_{1,1} < \infty.
		\]
	\end{lemma}
	
	\begin{proof}
		For each \(\ell\ge0\), there exists a polynomial \(P_\ell\) of degree \(\ell\) such that
		\[
		F^{(\ell)}(x)=P_\ell(x)\, e^{-\frac{x^2}{4}}.
		\]
		We shall use the elementary polynomial bound
		\begin{equation}\label{eq:polynomial-bound}
			|P_\ell(x)|\leq C_\ell \bigl(1+|x|^\ell\bigr),\qquad x\in\mathbb R,
		\end{equation}
		for some constant \(C_\ell>0\).
		
		Now fix \(\ell\ge0\). By definition of the norm \(\|\cdot\|_{1,1}\),
		\[
		\|F^{(\ell)}\|_{1,1}=\sum_{m \in \mathbb{Z}}\sup_{0 \le s \le 2\pi}\left(|F^{(\ell)}(2\pi m+s)|+|F^{(\ell+1)}(2\pi m+s)|\right).
		\]
		For \(|m| \ge 1\), we have
		\begin{equation}\label{eq:Ap_A_lower}
			|2\pi m+s|\ge|2\pi m| - |s|\ge2\pi(|m| - 1)\ge0,
		\end{equation}
		and also
		\begin{equation}\label{eq:Ap_A_upper}
			|2\pi m+s|\le2\pi|m|+2\pi\le2\pi(|m|+1).
		\end{equation}
		Hence, applying \eqref{eq:polynomial-bound}  we obtain
		\[
		|F^{(\ell)}(2\pi m+s)| + |F^{(\ell+1)}(2\pi m+s)|\leq C_\ell\bigl(1+|m|^{\ell+1}\bigr)e^{-\pi^2(|m|-1)^2},
		\]
		for some \(C_\ell>0\) independent of \(m\) and \(s\). Consequently,
		\[
		\|F^{(\ell)}\|_{1,1}\leq C_\ell\sum_{|m| \ge 1}\left( (1+|m|^{\ell+1}) e^{-\pi^2(|m|-1)^2}\right)+\mathcal{C^{(\ell)}},
		\]
		where
		\[
		\mathcal{C^{(\ell)}}:=\sup_{0 \le s \le 2\pi}\left(|F^{(\ell)}(s)| + |F^{(\ell+1)}(s)|\right),
		\] 
		accounts for the contribution of the term \(m=0\). The exponential decay ensures that the series converges. Therefore
		\[
		\boxed{\|F^{(\ell)}\|_{1,1} < \infty.}
		\]
	\end{proof}

	The proof relies on the following localization. For \(t>0\), set
	\begin{equation}\label{eq:Ap_A_compact}
		K_t:=\left[-\frac{\pi}{\sqrt t},\frac{\pi}{\sqrt t}\right].
	\end{equation}
	If \(f_t\) is defined on \(K_t\) by
	\[
	f_t(y)=\sqrt t\,p_t(\sqrt t\,y),\qquad y\in K_t,
	\]
	then for every \(x\in[-\pi,\pi]\), the rescaled variable \(y=x/\sqrt t\) belongs to \(K_t\), and therefore
	\[
	\frac1{\sqrt t}\,f_t\!\left(\frac{x}{\sqrt t}\right)=p_t(x).
	\]
	Thus it is enough to define \(f_t\) on \(K_t\). For convenience, we work on \([-\pi,\pi]\) instead of \([0,2\pi]\); by \(2\pi\)-periodicity of the heat kernel, estimates on \([-\pi,\pi]\) extend to any interval of length \(2\pi\), in particular to \([0,2\pi]\).

	Observe that \(K_t\) expands as \(t\to0^{+}\), but remains compact for each fixed \(t\in(0,1]\). Accordingly, we consider smooth cut-off functions \(\{\chi_t\}_{t\in(0,1]}\) satisfying the following properties:
	\begin{enumerate}
		\item \(\chi_t\in C_c^\infty(\mathbb R)\);
		\item \(0\le \chi_t\le 1\);
		\item \(\chi_t\equiv 1\) on \(K_t\);
		\item \(\operatorname{supp}(\chi_t)\subset \left[-\dfrac{2\pi}{\sqrt t},\dfrac{2\pi}{\sqrt t}\right]\);
		\item \(\chi_t\left(\pm\dfrac{2\pi}{\sqrt t}\right)=0\);
		\item for every \(k\ge0\), there exists a constant \(C_k>0\), independent
		of \(t\in(0,1]\), such that
		\[
		\|\chi_t^{(k)}\|_\infty\le C_k.
		\]
	\end{enumerate}
	
	The construction of such cut-off functions is standard; we omit the details.
	
	\begin{remark}\label{rem:Ap_A_supp}
		Let \(m\in\mathbb Z\) and \(s\in[0,2\pi]\). If
		\[
		|m|>1+\frac{1}{\sqrt t} \quad \text{then} \quad
		|2\pi m+s|\ge2\pi |m|-|s|>2\pi\Bigl(1+\frac1{\sqrt t}\Bigr)-2\pi=\frac{2\pi}{\sqrt t} .
		\]
		Therefore
		\[
		2\pi m+s\notin \operatorname{supp}(\chi_t),
		\]
		and hence
		\[
		\chi_t(2\pi m+s)=0.
		\]
		Equivalently, the only possibly nonzero terms involving \(\chi_t\) occur when
		\[
		|m|\le 1+\frac{1}{\sqrt t}.
		\]
	\end{remark}
	
	We next prove an auxiliary convergence lemma.
	
	\begin{lemma}\label{lem:gaussian-tail}
		For every integer \(r \ge 0\), there exists a constant \(C_r > 0\) such that
		\[
		\sum_{k=4}^{\infty} (k+3)^r e^{-\frac{\pi^2}{t}(k-3)^2} \leq C_r\, t^{\frac{r+1}{2}},\qquad \text{for all } t \in (0,1].
		\]
	\end{lemma}
	
	\begin{proof}
		Set \(a := \pi^2/t\). Since \(t \in (0,1]\), we have \(a \ge \pi^2 > 0\).  After the change of index \(j = k-3\), we have \(j\ge1\) and \(k+3 = j+6 \le 7j\). Hence,
		\[
		\sum_{k=4}^{\infty}(k+3)^r e^{-\frac{\pi^2}{t}(k-3)^2}\leq 7^r\sum_{j=1}^{\infty}j^r e^{-a j^2},
		\]
		Thus, it remains to estimate \(S_r(a) := \sum_{j=1}^{\infty} j^r e^{-a j^2}\).
		
		For \(x\ge0\), consider \(\varphi_r(x) = x^r e^{-a x^2/2}\). If \(r=0\), then \(\varphi_0(x)\le1\). If \(r\ge1\), the function \(\varphi_r\) attains its maximum at \(x_* = \sqrt{r/a}\), and its maximum value satisfies \(\varphi_r(x_*) \le C_r a^{-r/2}\). 
		Consequently, for all \(x\ge0\), we have \( x^r e^{-a x^2/2} \le C_r a^{-r/2}\). Applying this estimate at \(x=j\) yields
		\[
		j^r e^{-a j^2}=\left(j^r e^{-a j^2/2}\right) e^{-a j^2/2}\leq C_r a^{-r/2} e^{-a j^2/2}.
		\]
		Summing over \(j\ge1\) and using the integral test for the decreasing function \(x\mapsto e^{-a x^2/2}\), we get
		\[
		\sum_{j=1}^{\infty} e^{-a j^2/2}\leq \int_0^\infty e^{-a x^2/2}\,dx =\sqrt{\frac{\pi}{2a}}.
		\]
		Combining the preceding bounds, since \(a = \pi^2/t\), it follows that \(a^{-(r+1)/2} = \pi^{-(r+1)} t^{(r+1)/2}\). Absorbing the constants that depend only on \(r\) into \(C_r\), we arrive at
		\[
		\boxed{\sum_{k=4}^{\infty} (k+3)^r e^{-\frac{\pi^2}{t}(k-3)^2}\leq C_r t^{\frac{r+1}{2}},\qquad t\in(0,1].}
		\]
		
	\end{proof}
	
	\noindent Finally, we introduce the following auxiliary terms. We define for fixed \(t\in(0,1]\)  \(\mathcal{C}_{k,m}^{(\ell)}\) for integers \( k,m\) and \(\ell\geq0\) by
	\begin{equation}\label{eq:def-C}
		\mathcal{C}^{(\ell)}_{k,m}:=\sup_{0 \le s \le 2\pi}\left(\left|F^{(\ell)}\left(2\pi m+s - \frac{2\pi k}{\sqrt{t}}\right)\right|+ \left| F^{(\ell+1)}\left(2\pi m+s - \frac{2\pi k}{\sqrt{t}}\right) \right| \right),
	\end{equation}
	with \(\mathcal{C}^{(0)}_{k,m}=\mathcal{C}_{k,m}\)
	\begin{lemma}\label{lem:translation-bound}
		Fixed \(t\in(0,1]\). For every integer \(\ell\ge 0\) and every \(k\in\mathbb Z\),
		\[
		\sum_{m\in\mathbb Z} \mathcal{C}_{k,m}^{(\ell)}\le2\|F^{(\ell)}\|_{1,1}.
		\]
		Consequently, for every finite set \(K\subset\mathbb Z\),
		\[
		\sum_{k\in K}\sum_{m\in\mathbb Z} \mathcal{C}_{k,m}^{(\ell)}\le2|K|\,\|F^{(\ell)}\|_{1,1}.
		\]
	\end{lemma}
	
	\begin{proof}
		For \(k=0\), the inequality is immediate from the definition of \(\|F^{(\ell)}\|_{1,1}\). For \(k\neq0\), each interval 
		\[
		I_{m,k}:=\left[2\pi m-\frac{2\pi k}{\sqrt t},\,2\pi m-\frac{2\pi k}{\sqrt t}+2\pi\right],\qquad m\in\mathbb Z
		\] 
		is contained in the union of two consecutive intervals of the original partition, namely \([2\pi j_m, 2\pi j_m+2\pi]\) and \([2\pi j_m+2\pi, 2\pi j_m+4\pi]\) with \(j_m = \lfloor m - k/\sqrt t\rfloor\). The map \(m\mapsto j_m\) is a bijection of \(\mathbb Z\). Therefore, summing over \(m\in\mathbb Z\) and reindexing gives at most twice the sum over all canonical intervals, i.e.
		\[
		\sum_{m\in\mathbb Z} \mathcal{C}_{k,m}^{(\ell)}\leq 2\sum_{j\in\mathbb Z} \sup_{x\in[2\pi j,2\pi j+2\pi]} \bigl(|F^{(\ell)}(x)|+|F^{(\ell+1)}(x)|\bigr)= 2\|F^{(\ell)}\|_{1,1}.
		\]
		The finite-sum estimate follows by summing over \(k\in K\).
	\end{proof}
	
	With all the preceding tools, we now present the proof of Lemma~\ref{lem:heat_prop}.
	
	\begin{proof}[Proof of Lemma~\ref{lem:heat_prop}]
		For every integer \(\ell\ge 0\) and \(x\in\mathbb{R}\), define
		\[
		\widetilde f_t^{(\ell)}(x):=\sum_{m\in\mathbb Z}F^{(\ell)}\left(x-\frac{2\pi m}{\sqrt t}\right).
		\]
		For \(t\in(0,1]\), set
		\[
		f_t^{(\ell)}(x):=\chi_t(x)\,\widetilde f_t^{(\ell)}(x).
		\]
		For \(t>1\), fix \(T>1\). Define the fixed cut-off
		\[
		\chi^{T}(x):=\chi_{1}(x),
		\]
		and, with a slight abuse of notation,
		\[
		\prescript{}{T}{f_t^{(\ell)}}(x):=\chi^{T}(x)\,\widetilde f_t^{(\ell)}(x).
		\]
		When no confusion arises, we omit the left subscript \(T\).
		
		We first prove the uniform bound for an arbitrary fixed \(\ell\ge 0\) and \(t\in(0,1]\). For \(|m|\le 1+\frac1{\sqrt t}\), \(s\in[0,2\pi]\), and \(|k|\ge 4\). 
		
		We have the lower bound
		\begin{equation}\label{eq:Ap_A_main_lower}
			\left|2\pi m+s-\frac{2\pi k}{\sqrt t}\right|\geq \frac{2\pi}{\sqrt t}|k|-\frac{2\pi}{\sqrt t}-4\pi\ge \frac{2\pi}{\sqrt t}(|k|-3),
		\end{equation}
		and the upper bound
		\begin{equation}\label{eq:Ap_A_main_upper}
			\left|2\pi m+s-\frac{2\pi k}{\sqrt t}\right| \leq 2\pi|m|+2\pi+\frac{2\pi|k|}{\sqrt t}\le \frac{2\pi}{\sqrt t}(|k|+3).
		\end{equation}
		By \eqref{eq:polynomial-bound}, \eqref{eq:Ap_A_main_lower}, \eqref{eq:Ap_A_main_upper} and \(t\in(0,1]\), we get for \(|k|\ge 4\),
		\[
		\boxed{|F^{(\ell)}(2\pi m+s-\frac{2\pi k}{\sqrt t})|+|F^{(\ell+1)}(2\pi m+s-\frac{2\pi k}{\sqrt t})|\leq C_\ell\frac{(|k|+3)^{\ell+1}}{t^{(\ell+1)/2}}e^{-\frac{\pi^2}{t}(|k|-3)^2}.}
		\]
		
		Now, by Remark~\ref{rem:Ap_A_supp}, the only possibly nonzero contributions to \(\|f_t^{(\ell)}\|_{1,1}\) occur when \(|m|\le 1+\frac1{\sqrt t}\). Moreover,
		\[
		(f_t^{(\ell)})'(x)=\chi_t'(x)\,\widetilde f_t^{(\ell)}(x)+\chi_t(x)\,\widetilde f_t^{(\ell+1)}(x),
		\]
		where
		\[
		\widetilde f_t^{(\ell+1)}(x)=\sum_{n\in\mathbb Z} F^{(\ell+1)}\left(x-\frac{2\pi n}{\sqrt t}\right).
		\]
		Since \(F^{(\ell)}\) and \(F^{(\ell+1)}\) decay exponentially, the series converge uniformly on compact sets. Because \(\chi_t\) and \(\chi_t'\) are compactly supported and uniformly bounded,
		\[
		|(f_t^{(\ell)})'(2\pi m+s)|\leq C_\chi\left(|\widetilde f_t^{(\ell)}(2\pi m+s)|+|\widetilde f_t^{(\ell+1)}(2\pi m+s)|\right),
		\]
		where
		\[
		C_\chi:=\sup_{0<t\le1}\left(\|\chi_t\|_\infty+\|\chi_t'\|_\infty\right).
		\]
		
		Thus,
		\[
		\begin{split}
			\|f_t^{(\ell)}\|_{1,1} \leq C_\chi\sum_{|m|\le 1+\frac1{\sqrt t}}\left\{\sum_{|k|\ge 4}C_\ell t^{-(\ell+1)/2}(|k|+3)^{\ell+1}e^{-\frac{\pi^2}{t}(|k|-3)^2}+\sum_{|k|\le 3}\mathcal C_{k,m}^{(\ell)}\right\}.
		\end{split}
		\]
		Since
		\[
		\#\Bigl\{m\in\mathbb Z:\ |m|\le 1+\frac1{\sqrt t}\Bigr\}\le3+\frac2{\sqrt t}\le\frac5{\sqrt t},
		\]
		we obtain
		\[
		\begin{split}
			\|f_t^{(\ell)}\|_{1,1}\leq C_\chi\left\{C_\ell t^{-(\ell+1)/2}\frac5{\sqrt t}\sum_{|k|\ge 4}(|k|+3)^{\ell+1}e^{-\frac{\pi^2}{t}(|k|-3)^2}+\sum_{m\in\mathbb Z}\sum_{|k|\le 3}\mathcal C_{k,m}^{(\ell)}\right\}.
		\end{split}
		\]
		By Lemma~\ref{lem:gaussian-tail} with \(r=\ell+1\), the first term is bounded by a constant independent of \(t\in(0,1]\). The second term is bounded by Lemma~\ref{lem:F-norm-finite} and Lemma~\ref{lem:translation-bound}. Hence
		\[
		\sup_{t\in(0,1]}\|f_t^{(\ell)}\|_{1,1}< \infty.
		\]
		Now fix \(T>1\) and consider \(t\in(1,T]\). For this range, we use the fixed cut-off
		\[
		\chi^T:=\chi_1.
		\]
		It follows that
		\[
		\chi^T(2\pi m+s)=0 \qquad\text{for}\qquad |m|\ge2.
		\]
		Thus only \(m=0,-1\) can actually contribute; summing over \(|m|\le 1\) is a harmless overcount.
		
		For \(|m|\le 1\), we have the analogue of
		\eqref{eq:Ap_A_main_lower} and \eqref{eq:Ap_A_main_upper}:
		\[
		\left|2\pi m+s-\frac{2\pi k}{\sqrt t}\right|\geq 2\pi\left(\frac{|k|}{\sqrt t}-3\right)\ge 2\pi\left(\frac{|k|}{\sqrt T}-3 \right),
		\]
		and
		\[
		\left|2\pi m+s-\frac{2\pi k}{\sqrt t}\right|\leq 2\pi\left(|k|+3\right).
		\]
		Thus, for \(|k|>3\sqrt T\), the same polynomial estimate gives
		\[
		\boxed{|F^{(\ell)}(2\pi m+s-\frac{2\pi k}{\sqrt t})|+|F^{(\ell+1)}(2\pi m+s-\frac{2\pi k}{\sqrt t})|\leq C_\ell (|k|+3)^{\ell+1} e^{-\pi^2\left(\frac{|k|}{\sqrt T}-3\right)^2}.}
		\]
		Therefore,
		\begin{align}
			\|\prescript{}{T}{f_t^{(\ell)}}\|_{1,1} 
			&\leq C_\chi\sum_{|m|\le 1}\Bigg\{\sum_{|k|>3\sqrt T} C_\ell(|k|+3)^{\ell+1}e^{-\pi^2\left(\frac{|k|}{\sqrt T}-3\right)^2}+\sum_{|k|\le 3\sqrt T}\mathcal C_{k,m}^{(\ell)}\Bigg\} \notag \\
			&\leq C_\chi\Bigg\{3C_\ell\sum_{|k|>3\sqrt T}(|k|+3)^{\ell+1}e^{-\pi^2\left(\frac{|k|}{\sqrt T}-3\right)^2}+\sum_{|k|\le 3\sqrt T}\sum_{|m|\le 1}\mathcal C_{k,m}^{(\ell)}\Bigg\}. \label{eq:Ap_A_bound_T}
		\end{align}
		The first sum is bounded uniformly in \(t\in(1,T]\) by exponential decay in \(|k|\). The second sum is finite by Lemma~\ref{lem:F-norm-finite} and  Lemma~\ref{lem:translation-bound}. Hence
		\[
		\sup_{t\in(1,T]} \|\prescript{}{T}{f_t^{(\ell)}}\|_{1,1} < \infty.
		\]
		
		Combining the cases \(t\in(0,1]\) and \(t\in(1,T]\), we obtain
		\[
		\sup_{t\in(0,T]} \|f_t^{(\ell)}\|_{1,1} <\infty, \qquad \ell\ge 0,
		\]
		where for \(t>1\) we understand \(f_t^{(\ell)}\) as
		\(\prescript{}{T}{f_t^{(\ell)}}\).
		
		It remains to verify the scaling identities. For \(t\in(0,1]\) and \(x\in[-\pi,\pi]\), set \(y=x/\sqrt t\). Then \(y\in K_t\), and since \(\chi_t\equiv1\) on \(K_t\), all derivatives of \(\chi_t\) vanish at \(y\). Hence
		\[
		(f_t^{(0)})^{(\ell)}(y) = f_t^{(\ell)}(y).
		\]
		On the other hand, by the chain rule,
		\[
		\partial_x^\ell p_t(x) = \frac{1}{\sqrt t} \left(\frac{1}{\sqrt t}\right)^\ell (f_t^{(0)})^{(\ell)}(y) = \frac{1}{t^{(\ell+1)/2}} f_t^{(\ell)}(y).
		\]
		The same computation applies for \(t>1\) with the fixed cut-off \(\chi^T\), because \(\chi^T\equiv1\) on \([-\pi,\pi]\) and therefore \(\chi^T\) and all its derivatives vanish at \(y=x/\sqrt t\). Since \(p_t\) is \(2\pi\)-periodic, these identities extend from \([-\pi,\pi]\) to \([0,2\pi]\). This completes the proof of Lemma~\ref{lem:heat_prop}.
	\end{proof}
	
	\begin{remark}
		The most delicate case is \(0<t\le 1\), because as \(t\to 0^{+}\) the interval
		\[
		K_t=\left[-\frac{\pi}{\sqrt t},\frac{\pi}{\sqrt t}\right],
		\]
		expands indefinitely, and the rescaled Gaussian sum must be controlled uniformly. For \(t\geq 1\), the interval \(K_t\) is contained in \([-\pi,\pi]\). Therefore, for fixed \(T>1\) there exists a constant \(C_T<\infty\) such that
		\[
		\|\prescript{}{T}f_t^{(\ell)}\|_{1,1}\le C_T.
		\]
		This constant is finite for each fixed \(T\) and grows at most like \(\sqrt T\) as \(T\to\infty\), since by Lemma~\ref{lem:translation-bound} the finite sum over \(|k|\le 3\sqrt T\) in \eqref{eq:Ap_A_bound_T} contains \(O(\sqrt T)\) terms.
	\end{remark}
	
	\begin{remark}
		In the proof of Lemma~\ref{lem:heat_prop}, a general scheme was used to treat all derivatives of the heat kernel simultaneously. For the applications of the lemma, however, only \(p_t\) and \(\partial_x^2 p_t\) are needed; in the notation of the lemma, these correspond to \(f_t\) and \(h_t\), respectively.
	\end{remark}

	\section{Uniform bounds for the deterministic integrals of Proposition~\ref{prop:regularity_v}}\label{appendix:B}
	
	The functions \(\IA(z)\) and \(\IB(z)\) given by
	\[
	\IA(z) = \int_0^z \int_0^1 (z-x+y)^{\frac{\alpha}{2} - \frac{3}{2}} \, dy \, dx, \qquad
	\IB(z) = \int_0^z \int_0^1 (z-x+y)^{\frac{\alpha}{2} - \frac{3}{2}} x^{\nu} \, dy \, dx,
	\]
	where \(\nu:=(\beta-2\alpha)/2\), are uniformly bounded for all \(z>0\).  Indeed, integrating with respect to \(y\) and then changing variable as \(u = z - x\) gives
	\[
	\IA(z) = \frac{1}{a}\int_0^z \bigl(u^{-a} - (u+1)^{-a}\bigr)\,du, \qquad \IB(z) = \frac{1}{a}\int_0^z \bigl(u^{-a} - (u+1)^{-a}\bigr)(z-u)^{\nu}\,du,
	\]
	where \(a=(1-\alpha)/2\).  Note that \(-1<\nu<0\): indeed \(\nu<0\) because \(2\alpha>2/3>1/2>\beta\), and \(\nu>-1/3>-1\) because \(\beta>1/3\) and \(\alpha<1/2\). Recall that \(\Beta(\cdot,\cdot)\) denotes the Euler beta function.
	
	\noindent \textbf{Bound for \(\IA(z)\).} The integrand \(\phi(u)=u^{-a}-(u+1)^{-a}>0\), so \(\IA(z)\) is increasing. Its limit is
	\[
	\lim_{z\to\infty} \IA(z) = \frac{1}{a}\int_0^\infty \phi(u)\,du = \frac{1}{a(1-a)} = \frac{4}{1-\alpha^2}.
	\]
	Hence
	\[
	\boxed{\IA(z) \le \frac{4}{1-\alpha^2}, \qquad \forall z>0}
	\]
	
	\noindent \textbf{Bound for \(\IB(z)\).} We use \(\phi(u)\le \min(u^{-a},\, a u^{-a-1})\). Splitting at \(u=1\):
	\[
	\IB(z) \le \frac{1}{a}\underbrace{\int_0^{\min(z,1)} u^{-a}(z-u)^{\nu}\,du}_{=:J_1} + \underbrace{\int_{\min(z,1)}^{z} u^{-a-1}(z-u)^{\nu}\,du}_{=:J_2}.
	\]
	For \(J_1\), if \(z\le1\), with \(u=zt\), \(J_1 = z^{1-a+\nu}\Beta(1-a,1+\nu)\le \Beta(1-a,1+\nu)\); if \(z>1\), since \(z-u\ge1-u\) and \(\nu<0\), \(J_1\le \int_0^1 u^{-a}(1-u)^{\nu}du = \Beta(1-a,1+\nu)\). Thus \(J_1\le \Beta(1-a,1+\nu)\), the Beta integral being convergent because \(a<1\) and \(\nu>-1\).
	
	For \(J_2\), if \(z>2\), split \([1,z/2]\) and \([z/2,z]\):
	\[
	\int_1^{z/2} u^{-a-1}(z-u)^{\nu}du \le (z/2)^{\nu}\int_1^{z/2}u^{-a-1}du \le \frac{1}{a}(z/2)^{\nu}\le \frac{1}{a},
	\]
	\[
	\int_{z/2}^{z} u^{-a-1}(z-u)^{\nu}du \le (z/2)^{-a-1}\int_0^{z/2}v^{\nu}dv = \frac{1}{\nu+1}(z/2)^{-a+\nu}\le \frac{1}{\nu+1}.
	\]
	For \(1<z\le2\) one has \(J_2\le \int_{1}^{z}(z-u)^{\nu}du=\frac{(z-1)^{\nu+1}}{\nu+1}\le\frac{1}{\nu+1}\), while \(J_2=0\) for \(z\le1\). Hence \(J_2\le 1/a + 1/(\nu+1)\) in all cases. Therefore,
	\[
	\boxed{\IB(z) \le \frac{1}{a}\Beta(1-a,1+\nu) + \frac{1}{a} + \frac{1}{\nu+1} =: K_B(\alpha,\beta)}
	\]

	In view of Remark~\ref{rem:global_vs_local}, to extend the regularity estimates from Proposition~\ref{prop:regularity_v} to all of \([0,T]\), it is also necessary to bound the following integral.
	\begin{equation*}
		I_C := \int_{t}^{t+h} (t+h-s)^{\frac{\alpha}{2}-\frac{1}{2}} \Bigl(1+s^{\frac{\beta-2\alpha}{2}}\Bigr)\,ds,
	\end{equation*}
	We observe that $\frac{\alpha}{2}-\frac{1}{2} = -a$, with the same notation as above. Performing the change of variables $u=t+h-s$, so that $ds=-du$, we obtain
	\[
	I_C = \int_{0}^{h} u^{-a} \Bigl(1+(t+h-u)^{\nu}\Bigr)\,du = I_D + I_E,
	\]
	where
	\begin{align*}
		I_D &:= \int_{0}^{h} u^{-a}\,du = \frac{h^{1-a}}{1-a} = \frac{2}{\alpha+1}\,h^{\frac{\alpha+1}{2}},\\[4pt]
		I_E &:= \int_{0}^{h} u^{-a}(t+h-u)^{\nu}\,du.
	\end{align*}
	To bound $I_E$ uniformly in $t$, we use that $\nu<0$ and $t+h-u\ge h-u$, hence $(t+h-u)^{\nu}\le (h-u)^{\nu}$:
	\[
	I_E \le \int_{0}^{h} u^{-a}(h-u)^{\nu}\,du = h^{-a+\nu+1}\Beta(1-a,1+\nu) = h^{\frac{\beta-\alpha+1}{2}}\Beta\!\Bigl(\frac{\alpha+1}{2},\frac{\beta-2\alpha+2}{2}\Bigr).
	\]
	Moreover, since $\frac{\alpha+1}{2} > \frac{\beta-\alpha+1}{2}$ and \(0<h<1\), we have $h^{\frac{\alpha+1}{2}} < h^{\frac{\beta-\alpha+1}{2}}$. Therefore,
	\begin{equation*}
		\boxed{I_C \le \Bigl(\frac{2}{\alpha+1} + \Beta(1-a,1+\nu)\Bigr)\,h^{\frac{\beta-\alpha+1}{2}}},
	\end{equation*}
	
	\begin{remark}\label{rem:uniform_case_Prop_reg}
		The exponent \(\nu=(\beta-2\alpha)/2\) is the one produced by the bound of Lemma~\ref{lem:rp_composition}; see Lemma~\ref{lem:Lambda}. The bound of \(I_B\) and \(I_C\) above is used only in Remark~\ref{rem:global_vs_local}, the proof of Proposition~\ref{prop:regularity_v} requiring \(\IA\) alone.
	\end{remark}

	\section{Properties of \texorpdfstring{$\Upsilon$}{Upsilon}}\label{appendix:C}
	In this appendix we define and study the auxiliary function \(\Upsilon\), which plays a central role in the computation of the temporal quartic variations of the stationary process \(\psi\).
	
	\medskip\noindent\textbf{Definition and basic properties of \(\Upsilon\).}
	
	For \(x\ge0\), define
	\begin{equation}\label{eq:ups_definition}
		\Upsilon(x) := \sum_{k=-\infty}^{\infty} \frac{e^{-(k^2+1)x}}{k^2+1}.     
	\end{equation}
	
	The series converges absolutely and uniformly on \([0,\infty)\).
	
	\noindent \textbf{Value at zero.}
	\[
	\Upsilon(0) = \sum_{k=-\infty}^{\infty} \frac{1}{k^2+1} = \pi\coth\pi. 
	\]
	a standard identity (see, e.g., \cite{GR}, Formula 1.421.4).
	
	\noindent \textbf{Monotonicity and convexity.} For \(x>0\),
	\begin{align}
		&\Upsilon'\!(x) = -\sum_{k=-\infty}^{\infty} e^{-(k^2+1)x}  < 0,\label{eq:ups_1_der} \\[2pt]
		&\Upsilon''\!(x) = \sum_{k=-\infty}^{\infty} (k^2+1)e^{-(k^2+1)x} > 0, \label{eq:ups_2_der}\\[2pt]
		&\Upsilon'''\!(x) = -\sum_{k=-\infty}^{\infty} (k^2+1)^{2}e^{-(k^2+1)x} < 0. \label{eq:ups_3_der}
	\end{align}
	
	\medskip\noindent\textbf{Integral representation and Jacobi theta function.}
	
	The third Jacobi theta function (with argument \(z=0\)) is
	\[
	\vartheta_3(q) = \sum_{k=-\infty}^{\infty} q^{k^2} = 1 + 2\sum_{k=1}^{\infty} q^{k^2}, \quad |q|<1. 
	\]
	Using the identity \(\frac{1}{k^2+1} = \int_0^\infty e^{-(k^2+1)u}\,du\) and interchanging the sum and integral (justified by absolute convergence), we obtain
	\begin{equation}\label{eq:ups_integral}
		\Upsilon(x) = \int_0^\infty e^{-(x+u)} \sum_{k=-\infty}^{\infty} e^{-k^2(x+u)}\,du = \int_x^\infty e^{-v}\vartheta_3(e^{-v})\,dv.     
	\end{equation}

	\medskip\noindent\textbf{Modular transformation and asymptotic expansion for small \(x\).}
	
	Set
	\[
	f(v) := e^{-v}\vartheta_3(e^{-v}), \qquad v>0,
	\]
	so that \(\Upsilon(x) = \int_x^\infty f(v)\,dv\) and
	\[
	\Upsilon(0)-\Upsilon(x) = \int_0^x f(v)\,dv.
	\]
	
	\medskip\noindent\textbf{Jacobi's modular identity and bounds for the theta remainder.}
	
	Jacobi's modular identity (see \cite{WW}, Chapter~21) states that for every \(v>0\),
	\[
	\vartheta_3(e^{-v}) = \sqrt{\frac{\pi}{v}}\;\vartheta_3\!\bigl(e^{-\pi^2/v}\bigr).
	\]
	Define the exponentially small remainder
	\[
	E_0(v) := \vartheta_3\!\bigl(e^{-\pi^2/v}\bigr) - 1 = 2\sum_{n=1}^{\infty} e^{-\pi^2 n^2/v}.
	\]
	For \(0<v\le 1\) we can bound \(E_0(v)\) by a geometric series:
	\[
	0 \le E_0(v) \le 2\sum_{n=1}^{\infty} e^{-\pi^2 n/v} = \frac{2e^{-\pi^2/v}}{1-e^{-\pi^2/v}} \le C_0 e^{-\pi^2/v},
	\]
	with \(C_0 := \frac{2}{1-e^{-\pi^2}}\).  
	Differentiating term by term (legitimate for \(v>0\)) yields
	\[
	E_0'(v) = 2\sum_{n=1}^{\infty} \frac{\pi^2 n^2}{v^2} e^{-\pi^2 n^2/v}.
	\]
	To bound \(E_0'(v)\), rewrite the summand using the inequality \(t^2 e^{-t} \le 4e^{-t/2}\) (valid for all \(t\ge0\)). Setting \(t = \pi^2 n^2/v\) we obtain
	\[
	\frac{\pi^2 n^2}{v^2} e^{-\pi^2 n^2/v}= \frac{1}{\pi^2 n^2}\, t^2 e^{-t}\le \frac{4}{\pi^2 n^2}\, e^{-t/2}= \frac{4}{\pi^2 n^2}\, e^{-\pi^2 n^2/(2v)}\le \frac{4}{\pi^2 n^2}\, e^{-\pi^2/(2v)},
	\]
	where in the last step we used \(n\ge1\). Summing over \(n\) and using \(\sum_{n=1}^\infty 1/n^2 = \pi^2/6\),
	\[
	|E_0'(v)| \le 2\sum_{n=1}^{\infty} \frac{4}{\pi^2 n^2} e^{-\pi^2/(2v)}= \frac{8}{\pi^2}\frac{\pi^2}{6} e^{-\pi^2/(2v)} = \frac{4}{3} e^{-\pi^2/(2v)}.
	\]
	Thus, for \(0<v\le1\), both \(E_0(v)\) and \(E_0'(v)\) are bounded by constants times an exponential that decays faster than any power of \(v\).
	
	\medskip\noindent\textbf{Truncated Taylor expansion of \(e^{-v.}\)}
	
	For any integer \(M\ge0\) and \(v>0\), Taylor's formula with Lagrange remainder gives
	\[
	e^{-v} = \sum_{m=0}^{M} \frac{(-v)^m}{m!} + R_M(v),
	\qquad
	|R_M(v)| \le \frac{v^{M+1}}{(M+1)!}.
	\]
	
	\medskip\noindent\textbf{Finite expansion of \(f(v)\) with explicit error bound.}
	
	Insert the modular identity and the expansion of \(e^{-v}\) into \(f(v)\):
	\[
	f(v) = \sqrt{\frac{\pi}{v}}\, e^{-v}\bigl(1 + E_0(v)\bigr)= \sqrt{\pi}\, v^{-1/2} \Bigl( \sum_{m=0}^{M} \frac{(-v)^m}{m!} + R_M(v) \Bigr)  \bigl(1 + E_0(v)\bigr).
	\]
	Distributing,
	\[
	f(v) = \underbrace{\sqrt{\pi} \sum_{m=0}^{M} \frac{(-1)^m}{m!} v^{m-1/2}}_{=:S_M(v)}+ T_1(v) + T_2(v) + T_3(v),
	\]
	where
	\[
	\begin{aligned}
		T_1(v) &:= \sqrt{\pi}\, v^{-1/2} R_M(v),\\
		T_2(v) &:= \sqrt{\pi}\, v^{-1/2} \Bigl(\sum_{m=0}^{M} \frac{(-v)^m}{m!}\Bigr) E_0(v),\\
		T_3(v) &:= \sqrt{\pi}\, v^{-1/2} R_M(v) E_0(v).
	\end{aligned}
	\]
	
	We now bound each remainder term for \(0<v\le 1\).
	
	\begin{itemize}
		\item For \(T_1\): using the bound on \(R_M\),
		\[
		|T_1(v)| \le \sqrt{\pi}\, v^{-1/2} \frac{v^{M+1}}{(M+1)!}= \frac{\sqrt{\pi}}{(M+1)!}\, v^{M+1/2}.
		\]
		
		\item For \(T_2\): the finite sum is bounded by \(\sum_{m=0}^{M} \frac{v^m}{m!} \le e^v \le e\) for \(v\le1\). Moreover, \(|E_0(v)|\le C_0 e^{-\pi^2/v}\). Hence
		\[
		|T_2(v)| \le \sqrt{\pi}\, v^{-1/2} \cdot e \cdot C_0 e^{-\pi^2/v} = e C_0\sqrt{\pi}\; v^{-1/2} e^{-\pi^2/v}.
		\]
		
		\item For \(T_3\): using both bounds,
		\[
		|T_3(v)| \le \sqrt{\pi}\, v^{-1/2} \frac{v^{M+1}}{(M+1)!}\, C_0 e^{-\pi^2/v} = \frac{C_0\sqrt{\pi}}{(M+1)!}\, v^{M+1/2} e^{-\pi^2/v}.
		\]
	\end{itemize}
	
	The crucial observation is that the exponential \(e^{-\pi^2/v}\) decays faster than any power \(v^{N}\) for $N>0$. To handle this rigorously, fix an integer \(M\) and consider the function
	\[
	g(v) = 
	\begin{cases}
		v^{-M-1} e^{-\pi^2/v}, & 0<v\le 1,\\0, & v=0.
	\end{cases}
	\]
	Since \(\lim_{v\to0^+} g(v)=0\), \(g\) is continuous on the compact interval \([0,1]\); hence it attains a maximum \(D_M < \infty\). Therefore, for all \(v\in(0,1]\),
	\[
	v^{-1/2} e^{-\pi^2/v} = v^{M+1/2} g(v) \le D_M\, v^{M+1/2}.
	\]
	Consequently, both \(|T_2(v)|\) and \(|T_3(v)|\) can be bounded by constant multiples of \(v^{M+1/2}\). Combining the three estimates we obtain a constant \(K_M\) (depending only on \(M\)) such that
	\begin{equation}\label{eq:f_bound}
		|f(v) - S_M(v)| \le K_M\, v^{M+1/2}, \qquad \forall\,v\in(0,1].     
	\end{equation}
	
	Thus, for any integer \(M\ge0\), the finite sum \(S_M(v)\) approximates \(f(v)\) with an error \(O(v^{M+1/2})\) uniformly on \((0,1]\).
	
	\medskip\noindent\textbf{Integration to obtain the expansion of \(\Upsilon\).}
	
	Fix an integer \(N\ge0\) and choose \(M = N+1\) in the above expansion. Then for any \(0<x\le 1\),
	\[
	\int_0^x f(v)\,dv = \int_0^x S_{N+1}(v)\,dv + \int_0^x \bigl(f(v)-S_{N+1}(v)\bigr)\,dv.
	\]
	
	The first integral is elementary:
	\[
	\int_0^x S_{N+1}(v)\,dv = \sqrt{\pi} \sum_{m=0}^{N+1} \frac{(-1)^m}{m!} \int_0^x v^{m-1/2}\,dv= \sqrt{\pi} \sum_{m=0}^{N+1} \frac{(-1)^m}{m!}\, \frac{x^{m+1/2}}{m+1/2}.
	\]
	
	For the second integral, \eqref{eq:f_bound} with \(M=N+1\) gives
	\[
	\biggl|\int_0^x \bigl(f(v)-S_{N+1}(v)\bigr)\,dv\biggr|\le \int_0^x K_{N+1}\, v^{(N+1)+1/2}\,dv= \frac{K_{N+1}}{N+5/2}\, x^{N+5/2}.
	\]
	
	We obtain, for every \(x\) with \(0<x\le 1\),
	\[
	\Upsilon(0)-\Upsilon(x) = \sqrt{\pi} \sum_{m=0}^{N+1} \frac{(-1)^m}{m!}\, \frac{x^{m+1/2}}{m+1/2}  + O\bigl(x^{N+5/2}\bigr),
	\]
	where the constant implied by the \(O\)-notation depends on \(N\) but not on \(x\).
	Letting \(N\) be arbitrary and using \(\Upsilon(0)=\pi\coth\pi\), we obtain the complete asymptotic expansion, valid as \(x\to0^+\) (in fact, uniformly for \(0<x\le 1\)):
	\begin{equation}\label{eq:ups_asymp_full}
		\Upsilon(x) \sim \pi\coth\pi - \sqrt{\pi}\sum_{m=0}^{\infty} \frac{(-1)^m}{m!\,(m+1/2)}\, x^{m+1/2}.
	\end{equation}
	where the symbol \(\sim\) denotes an asymptotic expansion in the sense of Poincaré (see~\cite[Chapter~1]{Olver}): truncating the series at any finite \(m\) gives an error of the order of the first omitted term.  Explicitly, for \(0<x\leq1\),
	\begin{equation}\label{eq:ups:expansion}
		\boxed{\Upsilon(x) = \pi\coth\pi - 2\sqrt{\pi x} + O(x^{3/2})}
	\end{equation}
	
	\medskip\noindent\textbf{Asymptotic expansions of the derivatives.}
	
	The derivatives of \(\Upsilon\) are given by the series
	\[
	\Upsilon'(x) = -\sum_{k=-\infty}^{\infty} e^{-(k^2+1)x} = -f(x), \qquad \Upsilon''(x) = \sum_{k=-\infty}^{\infty} (k^2+1)e^{-(k^2+1)x} = -f'(x),
	\]
	where \(f(v)=e^{-v}\vartheta_3(e^{-v})\) as before.  
	The expansion of \(\Upsilon'(x)\) follows immediately from the work of the previous section; it suffices to change the sign in \eqref{eq:f_bound} and the subsequent formulas. Thus, for every integer \(M\ge0\) there is a finite sum
	\[
	\widetilde S_M(v) := -\sqrt{\pi}\sum_{m=0}^{M} \frac{(-1)^m}{m!} v^{m-1/2}
	\]
	and a constant \(\widetilde K_M\) such that
	\[
	|\Upsilon'(v) - \widetilde S_M(v)| \le \widetilde K_M\, v^{M+1/2}, \qquad 0<v\le1.
	\]
	Consequently, as \(x\to0^+\) (in fact, uniformly for \(0<x\le1\)),
	\begin{equation}\label{eq:ups:expansion1}
		\Upsilon'(x) = -\sqrt{\frac{\pi}{x}} + \sqrt{\pi}\,x^{1/2} + O(x^{3/2}).    
	\end{equation}

	We now obtain the expansion of \(\Upsilon''(x)\).
	From \eqref{eq:ups_integral} we have the explicit expression
	\[
	f'(v) = \sqrt{\pi}\,\frac{d}{dv}\Bigl[ v^{-1/2} e^{-v}\bigl(1+E_0(v)\bigr) \Bigr].
	\]
	Computing the derivative,
	\begin{equation}\label{eq:fprime_exact}
		\begin{aligned}
			f'(v) = \sqrt{\pi}\, e^{-v}\Bigl( &-\frac12 v^{-3/2}\bigl(1+E_0(v)\bigr)   - v^{-1/2}\bigl(1+E_0(v)\bigr) \\     
			&+ v^{-1/2} E_0'(v) \Bigr).
		\end{aligned}
	\end{equation}
	Recall the bounds obtained above: for \(0<v\le1\),
	\[
	|E_0(v)| \le C_0 e^{-\pi^2/v}, \qquad  |E_0'(v)| \le \frac43 e^{-\pi^2/(2v)}.
	\]
	Both are exponentially small and will be absorbed into the remainder estimates.
	
	Fix an integer \(M\ge0\) and expand \(e^{-v}\) with the Lagrange remainder:
	\[
	e^{-v} = \sum_{m=0}^{M} \frac{(-v)^m}{m!} + R_M(v), \qquad  |R_M(v)| \le \frac{v^{M+1}}{(M+1)!}.
	\]
	Insert this into \eqref{eq:fprime_exact} and separate the terms that do not contain \(E_0\) or \(E_0'\):
	\[
	\begin{aligned}
		f'(v) &= \sqrt{\pi}\sum_{m=0}^{M} \frac{(-1)^m}{m!} v^{m} \Bigl( -\frac12 v^{-3/2} - v^{-1/2} \Bigr) \\
		&\quad + \sqrt{\pi}\, R_M(v) \Bigl( -\frac12 v^{-3/2} - v^{-1/2} \Bigr) \\
		&\quad + \sqrt{\pi}\, e^{-v} \Bigl( -\frac12 v^{-3/2} E_0(v) - v^{-1/2} E_0(v) + v^{-1/2} E_0'(v) \Bigr).
	\end{aligned}
	\]
	The last line contains all the exponentially small contributions; we denote it by \(T_{\exp}(v)\).  
	The main part is the finite sum
	\[
	S^{(2)}_M(v) := \sqrt{\pi}\sum_{m=0}^{M} \frac{(-1)^m}{m!}  \Bigl( -\frac12 v^{m-3/2} - v^{m-1/2} \Bigr).
	\]
	
	We now estimate the two remainder terms for \(0<v\le1\).
	
	\begin{itemize}
		\item \textbf{Polynomial remainder term.}
		\[
		\bigl|\sqrt{\pi}\, R_M(v) (-\frac12 v^{-3/2} - v^{-1/2})\bigr|\le \sqrt{\pi}\frac{v^{M+1}}{(M+1)!}\bigl(\frac12 v^{-3/2}+v^{-1/2}\bigr) \le \frac{3\sqrt{\pi}}{2(M+1)!} v^{M-1/2}.
		\]
		
		\item \textbf{Exponentially small term.}
		Using \(|e^{-v}|\le1\) for \(v\ge0\) and the bounds on \(E_0, E_0'\), we get
		\[
		|T_{\exp}(v)| \le \sqrt{\pi}\Bigl( \frac{C_0}{2} v^{-3/2}  e^{-\pi^2/v} +C_0 v^{-1/2}  e^{-\pi^2/v} + \frac43v^{-1/2}  e^{-\pi^2/(2v)} \Bigr).
		\]
		As before, the exponentials dominate any power. Define a continuous function on \([0,1]\) by
		\[
		h(v) = 
		\begin{cases}
			v^{-M-1} e^{-\pi^2/(2v)}, & 0<v\le1,\\ 0, 
			& v=0.
		\end{cases}
		\]
		Since \(\lim_{v\to0^+} h(v)=0\), \(h\) is bounded by some constant \(H_M\) on \([0,1]\).  Consequently,
		\[
		v^{-1/2} e^{-\pi^2 / v}\leq v^{-1/2} e^{-\pi^2 / (2v)} = v^{M+1/2} h(v) \leq H_M v^{M+1/2} \leq H_M v^{M-1/2}
		\]
		for \(0 < v \leq 1\). Similarly
		\[
		v^{-3/2} e^{-\pi^2 / v} \leq v^{-3/2} e^{-\pi^2 / (2v)} = v^{M-1/2} h(v) \leq H_M v^{M-1/2}
		\]
		
		All terms in \(T_{\exp}(v)\) are therefore bounded by a constant multiple of \(v^{M-1/2}\).
	\end{itemize}
	
	Summing the two bounds, there exists a constant \(C'_M\) such that
	\[
	|f'(v) - S^{(2)}_M(v)| \le C'_M\, v^{M-1/2}, \qquad 0<v\le1.
	\]
	
	Finally, since \(\Upsilon''(x) = -f'(x)\), we obtain for every integer \(M \geq 0\),
	
	\[
	\boxed{\Upsilon''(x) = \sqrt{\pi} \sum_{n=0}^{M} \frac{(-1)^{n+1}}{n!} \left(n - \frac{1}{2}\right) x^{n-3/2} + O\left(x^{M-1/2}\right)}
	\]
	
	for \(0 < x \leq 1\).
	Choosing \(M\) large enough (e.g., \(M=3\) to see the first few terms) and expanding,
	
	\begin{align*}
		\Upsilon''(x) &= \sqrt{\pi} \left[ \frac{1}{2} x^{-3/2} + \frac{1}{2} x^{-1/2} - \frac{3}{4} x^{1/2} + \frac{5}{12} x^{3/2} \right] + O(x^{5/2}) \\
		&= \frac{\sqrt{\pi}}{2} x^{-3/2} + \frac{\sqrt{\pi}}{2} x^{-1/2} - \frac{3\sqrt{\pi}}{4} x^{1/2} + \frac{5\sqrt{\pi}}{12} x^{3/2} + O(x^{5/2}).
	\end{align*}
	
	In particular, we have
	\begin{equation}\label{eq:ups:expansion2}
		\Upsilon''(x)= \frac{\sqrt{\pi}}{2}x^{-3/2} + \frac{\sqrt{\pi}}{2}x^{-1/2} + O(x^{1/2})
	\end{equation}
	This expansion is uniform for \(0<x\le1\) and provides the exact singular behaviour of the second derivative near the origin. Together with the expansion of \(\Upsilon\) itself, it will be instrumental in computing the variations of \(\psi\).

\end{appendices}

\bibliographystyle{plain}

\section*{Data availability statement}

Data sharing is not applicable to this article as no data sets 
were generated or analysed during the current study.

 \section*{Acknowledgements}

The author Christian Olivera is partially supported by  by FAPESP-ANR by the grant Stochastic and Deterministic Analysis for Irregular Models$-2022/03379-0$  and CNPq by the grant $422145/2023-8$.  C. Tudor  acknowledges support from   the  ANR project SDAIM 22-CE40-0015, the MATHAMSUD grant 24-MATH-04 SDE-EXPLORE,  and  by the Ministry of Research, Innovation and Digitalization (Romania), grant CF-194-PNRR-III-C9-2023. C. Tudor also acknowledges the support of the CDP C2EMPI, together with the French State under the France-2030 programme, the University of Lille, the Initiative of Excellence of the University of Lille, the European Metropolis of Lille for their funding and support of the R-CDP-24-004-C2EMPI project.   The author Simon Chony Acosta is supported by the fellowship FAPESP $2024/06348-3$.

 \section*{Conflict of Interest}
There is no conflict of interests.

\bibliographystyle{plain}

\end{document}